\documentclass{amsart}[12pt,letter]
\usepackage{graphicx}
\graphicspath{ {./images/} }

\usepackage{amsfonts}
\usepackage{amssymb}
\usepackage{amsthm}
\usepackage[full]{textcomp}

\usepackage{mathtools}
\usepackage[utf8]{inputenc}
\usepackage{color,soul}
\usepackage{tikz-cd}
\usepackage{mathabx}
\usepackage{thmtools}
\usepackage{thm-restate}
\usepackage{hyperref}
\hypersetup{colorlinks=true}
\usepackage[all]{hypcap}
\usepackage[capitalise]{cleveref}
\usepackage{enumerate}
\usepackage[normalem]{ulem}

\usepackage[margin=1.3in]{geometry}

\newtheorem{theorem}{Theorem}
\newtheorem*{theorem*}{Theorem}
\newtheorem*{definition*}{Definition}
\newtheorem{definition}{Definition}[section]
\newtheorem{corollary}[definition]{Corollary}
\newtheorem*{corollary*}{Corollary}
\newtheorem{lemma}[definition]{Lemma}
\newtheorem{example}[definition]{Example}
\newtheorem*{lemma*}{Lemma}
\newtheorem{remark}[definition]{Remark}
\newtheorem{proposition}[definition]{Proposition}
\newtheorem*{proposition*}{Proposition}
\newtheorem{conjecture}[definition]{Conjecture}
\newtheorem{theoremaaa}[definition]{Theorem}
\newtheorem*{conjecture*}{Conjecture}
\newtheorem{question*}[definition]{Question}

\numberwithin{equation}{section}

\newcommand{\R}{\mathbb{R}}

\def\bM{\mathbb{M}}

\def\bQ{\mathbb{Q}}
\def\bR{\mathbb{R}}
\def\bS{\mathbb{S}}

\def\bZ{\mathbb{Z}}

\newcommand{\cC}{\mathcal{C}}
\newcommand{\cD}{\mathcal{D}}

\newcommand{\cF}{\mathcal{F}}

\newcommand{\cL}{\mathcal{L}}

\newcommand{\cP}{\mathcal{P}}

\newcommand{\eps}{\varepsilon}

\renewcommand{\H}{\mathcal{H}}

\renewcommand{\P}{\mathcal P}

\newcommand{\LL}{\mathcal{L}}

\let\oldtocsection=\tocsection
\let\oldtocsubsection=\tocsubsection
\let\oldtocsubsubsection=\tocsubsubsection

\renewcommand{\tocsection}[2]{\hspace{0em}\oldtocsection{#1}{#2}}
\renewcommand{\tocsubsection}[2]{\hspace{1em}\oldtocsubsection{#1}{#2}}
\renewcommand{\tocsubsubsection}[2]{\hspace{2em}\oldtocsubsubsection{#1}{#2}}

\title[Homotopy invariance of coproduct \& fixed-point theory]{Obstructions to homotopy invariance of loop coproduct via parameterized fixed-point theory} 
\author{Lea Kenigsberg and Noah Porcelli}

\newcommand{\Addresses}{{
  \bigskip
  \footnotesize

  \textsc{UC Davis department of Mathematics, 1 Shields Ave, Davis, CA 95616, United States}\par\nopagebreak
  \textit{kenigsberg.lea@gmail.com}

  \textsc{Imperial College Department of Mathematics, Huxley Building, 180 Queen's Gate, London SW7 2RH, U.K.}\par\nopagebreak
  \textit{n.porcelli@imperial.ac.uk}

}}

\begin{document}

\begin{abstract}
 Given $f:M \to N$ a homotopy equivalence of compact manifolds  with boundary, we use a construction of Geoghegan and Nicas to define its Reidemeister trace  $[T] \in  \pi_1^{st}(\LL N, N)$. We realize the Goresky-Hingston coproduct as a map of spectra, and show that the failure of $f$ to entwine the spectral coproducts can be characterized by Chas-Sullivan multiplication with $[T]$. In particular, when $f$ is a simple homotopy equivalence, the spectral coproducts of $M$ and $N$ agree. 
    
\end{abstract}

\maketitle

\tableofcontents

\section{Introduction}
    Let $M$ be a closed smooth oriented manifold, and $\LL M$ its free loop space. There are various structures one can define on the homology of $\LL M$. The first to be introduced was the Chas-Sullivan product \cite{chas-sullivan}:
    $$\mu^{CS}: H_{*}(\LL M)\otimes H_{*}(\LL M) \to H_{*-n}(\LL M),$$ which, roughly speaking, takes two generic families of loops in $M$ and concatenates them when their starting points agree.
    
    There is also the Goresky-Hingston coproduct \cite{Goresky_2009}: 
    $$\Delta^{GH}: H_*\left(\LL M\right) \to \tilde H_{*+1-n}\left(\frac{\LL M}M \wedge \frac{\LL M}M\right)$$ which takes a generic family of loops, and for each loop $\gamma$ in the family and $s \in [0,1]$ such that $\gamma(0) = \gamma(s)$, contributes the pair of loops $(\gamma|_{[0,s]}, \gamma|_{[s,1]})$.  See \cref{fig:coprod-heur}.

    \begin{figure}[h] 
        \label{fig:coprod-heur}
        \centering
        $\left\{
        \begin{minipage}{.23 \textwidth}
            \begin{tikzpicture}[scale=.8,
            v/.style={draw,shape=circle, fill=black, minimum size=1.3mm, inner sep=0pt, outer sep=0pt},
            vred/.style={draw,shape=circle, fill=red, minimum size=1mm, inner sep=0pt, outer sep=0pt},]
                
                \draw[dashed] (-1.5, 0) -- (1.5,0);

                \draw[color=black, ->] (-1.5,0) to[out=90, in=90] (-0.5, 1.5);
                \draw[color=black] (-0.5,1.5) to[out=270, in=90] (-1.2, 0);
                
                \draw[color=black, -<] (-1.5,0) to[out=-90, in=-90] (-0.5, -1.5);
                \draw[color=black] (-0.5,-1.5) to[out=-270, in=-90] (-1.2, 0);

                \draw[color=black, ->] (0,0) to[out=90, in=90] (1, 1.5);
                \draw[color=black] (1,1.5) to[out=270, in=90] (0, 0);
                
                \draw[color=black, -<] (0,0) to[out=-90, in=-90] (1, -1.5);
                \draw[color=black] (1,-1.5) to[out=-270, in=-90] (-0, 0);

                \draw[color=black, ->] (1.5,0) to[out=90, in=90] (2.5, 1.5);
                \draw[color=black] (2.5,1.5) to[out=270, in=90] (1.2, 0);
                
                \draw[color=black, -<] (1.5,0) to[out=-90, in=-90] (2.5, -1.5);
                \draw[color=black] (2.5,-1.5) to[out=-270, in=-90] (1.2, 0);

                \node[v] at (-1.5, 0) {};
                
                \node[v] at (-0, 0) {}; 
                \node[vred] at (-0, 0) {};

                \node[v] at (1.5, 0) {};
            \end{tikzpicture}
        \end{minipage}
        \right\} \mapsto$
        $\left(
        \begin{minipage}{.15 \textwidth}
            \begin{tikzpicture}[scale=.8,
            v/.style={draw,shape=circle, fill=black, minimum size=1.3mm, inner sep=0pt, outer sep=0pt},
            vred/.style={draw,shape=circle, fill=red, minimum size=1mm, inner sep=0pt, outer sep=0pt},]

                \draw[color=black, ->] (-1,0) to[out=90, in=90] (0, 1.5);
                \draw[color=black] (0,1.5) to[out=270, in=90] (-1, 0);

                \draw[color=black, -<] (0.5,0) to[out=-90, in=-90] (1.5, -1.5);
                \draw[color=black] (1.5,-1.5) to[out=-270, in=-90] (0.5, 0);

                \node[v] at (-1, 0) {}; 
                \node[vred] at (-1, 0) {};

                \node at (0,0) {,};

                \node[v] at (0.5, 0) {}; 
                \node[vred] at (0.5, 0) {};
            \end{tikzpicture}
        \end{minipage}
        \right)$
        \caption{Heuristic picture of the coproduct in the case $*=1$, $n=2$: left shows a $1$-parameter family of loops, right shows the output of the coproduct, a $0$-parameter family of pairs of loops.}
    \end{figure}
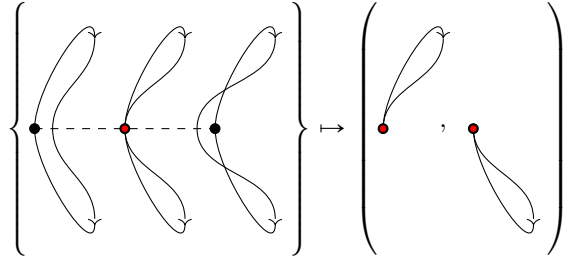
    
    There are many other structures and constructions of this flavor, all fall under the general umbrella term of string topology. For instance, there is a Lie bracket on equivariant homology $H_*^{S^1}(\LL M)$ \cite{chas-sullivan}.  Another example is Cohen-Jones' construction of a unital ring spectrum \cite{cohen2002homotopy}:
    \begin{equation}\label{eq: CS spec}
        \LL M^{-TM}\wedge \LL M^{-TM} \to \LL M^{-TM}. 
    \end{equation}  This structure  recovers the Chas-Sullivan product by taking homology, but also gives operations in other generalised homology theories.
    
The first offering of our paper is a generalization of the Goresky-Hingston coproduct to non-oriented manifolds with corners, and to a map of spectra:
\begin{equation}\label{eq: cop spec}
         \Delta: \frac{ \LL M^{-TM}}{\partial \LL M ^{-TM}}\wedge S^1\to \Sigma^\infty \frac{\LL M}{M}\wedge \frac{\LL M}{M},
 \end{equation}  where $\partial \LL M := \LL M|_{\partial M}$ is the space of loops $\gamma \in \LL M$ with $\gamma(0) \in \partial M$.

        \begin{remark}
    Note that $\Delta$ does not define a coring structure in the usual algebraic sense, since it is not of the form $A \to A \otimes A$ for any $A$. We still refer to $\Delta$ as a coproduct since when $M$ is a closed oriented manifold, $\Delta$ is a natural generalisation of the Goresky-Hingston coproduct, which does define a (non-unital) coalgebra structure on $H_{*+n-1}(\cL M, M; k)$ where $k$ is a field. See Section \ref{sec: cop hom comp} for an exact statement and proof. It would be interesting to understand the nature of the algebraic structure that $\Delta$ defines.
\end{remark}

It was shown in \cite{Cohen-Klein-Sullivan}, \cite{Crabb} and \cite{Gruher-Salvatore} that the Chas-Sullivan product is preserved by homotopy equivalences, and by Rivera-Wang \cite{Rivera-Wang} that for simply-connected manifolds the Goresky-Hingston coproduct over $\bQ$ is preserved by homotopy equivalences.\par

Motivated by a computation of Naef  \cite{Naef}, showing that the Goresky-Hingston coproduct is not a homotopy invariant in general, the main goal of this paper is to characterize the failure of the spectral Goresky-Hingston coproduct to be a homotopy invariant.\par

More precisely, let $f: N \to Z$ be a homotopy equivalence of compact manifolds with boundary. Then $f$ induces 
equivalences of spectra $f: \Sigma^\infty \cL N/N \to \Sigma^\infty \cL Z/Z$ and    
    \begin{equation*}
        f_!: \frac{\cL N^{-TN}}{\partial \cL N^{-TN}} \xrightarrow{\simeq}  \frac{\cL Z^{-TZ}}{\partial \cL Z^{-TZ}}.
    \end{equation*}
(see \cref{eq: h.e on -TN}). We study the failure of the diagram \begin{equation}\label{diag: fail of hom inv }
        \begin{tikzcd}
                    \frac{\cL N^{-TN}}{\partial \cL N^{-TN}} \wedge S^1\arrow[r, "\Delta^N"] \arrow[d, "f_! \wedge Id_{S^1}"] &
                    \Sigma^\infty \frac{\cL N}N \wedge \frac{\cL N}N \arrow[d, "f \wedge f"] \\
                    \frac{\cL Z^{-TZ}}{\partial \cL Z^{-TZ}} \wedge S^1 \arrow[r, "\Delta^Z"] &
                    \Sigma^\infty \frac{\cL Z}Z \wedge \frac{\cL Z}Z
                \end{tikzcd}
                \end{equation} 
                to commute. \par 
                As a first step to addressing the general case, we assume that $f$ is a codimension 0 embedding, and that the complement $W := Z \setminus N$ is an $h$-cobordism. We then define operations

$$\Xi_{{{r}}}, \Xi_{{{l}}}: \frac{\cL Z^{-TZ}}{\partial \cL Z^{-TZ}} \wedge S^1 \to \Sigma^\infty \frac{\cL Z}Z \wedge \frac{\cL Z}Z,$$
in the spirit of parameterized Reidemeister traces, following ideas of Geoghegan-Nicas \cite{GN} and Malkiewich \cite{malkiewichparametrized}. See \cref{sec: torsion} for further explanation. 

The first theorem of this paper is then:

\begin{theorem}[Theorem \ref{thm: cop Xi}]\label{thm: 1 thm}
Assume $f: N \to Z$ is a codimension 0 embedding such that the complement $W$ is an $h$-cobordism. Then the failure of diagram (\ref{diag: fail of hom inv }) to commute is given by $\Xi_{{{l}}}$ and $ \Xi_{{{r}}}$. That is: 
    \begin{equation}
                \Delta^Z \circ (f_! \wedge Id_{S^1}) - (f \wedge f) \circ \Delta^N \simeq \Xi_{{{l}}} - \Xi_{{{r}}}.
            \end{equation}
\end{theorem}

We next characterize the discrepancy $\Xi_{{{l}}} - \Xi_{{{r}}}$ in terms of existing operations and invariants. The content of the following theorem says that $\Xi_{{{l}}}$ and $\Xi_{{{r}}}$ can be interpreted as the Chas-Sullivan product with certain fixed point invariants: 

\begin{theorem}[Theorem \ref{thm: T Xi}]\label{thm Xi=mu}
   Assume $f: N \to Z$ is a codimension 0 embedding such that the complement $W$ is an $h$-cobordism. There are homotopies of maps of spectra:
    \begin{equation*}
        \Xi_{{{l}}}  \simeq \mu_{{{l}}}(\cdot \times [Z],[T_{diag}]) \textrm{    and      }
        \Xi_{{{r}}}  \simeq \mu_{{{r}}}([\overline T_{diag}],[Z] \times \cdot),
    \end{equation*} where: \begin{itemize}
    \item The classes:  $$[T_{diag}], [\overline T_{diag}]: \Sigma^\infty S^1\to \Sigma^{\infty}\frac{\cL N \times  \cL N}{N\times N}.$$ are parameterized fixed-point invariants associated to $f$. 
They are constructed as in Geoghegan-Nicas \cite{GN}, and are given by the fixed points of a strong deformation retraction $F: W\times I \to W$. See \cref{sec: torsion} and \cref{sec: tr is product} for more details.
\item The operations \begin{equation}
        \mu_{{{l}}}: \frac{\cL M^{-TM}}{\partial \cL M^{-TM}} \wedge \Sigma^\infty_+ \cL M \to\Sigma^\infty_+ \cL M,
    \end{equation}
    and
    \begin{equation}
        \mu_{{{r}}}: \Sigma^\infty_+ \cL M \wedge \frac{\cL M^{-TM}}{\partial \cL M^{-TM}} \to \Sigma^\infty_+ \cL M,
    \end{equation}
     are spectral Chas-Sullivan products in $Z\times Z$ and are defined in \cref{sec: prod}. After passing to homology, they realize the usual homology-level Chas-Sullivan products.
     \item The class $$[Z]: \bS \to Z^{-TZ}/\partial Z^{-TZ} \to \cL Z^{-TZ}/\partial \cL Z^{-TZ}$$ denote the fundamental class of $Z$, included in $\cL Z^{-TZ}/\partial \cL Z^{-TZ}$ by constant loops.

\end{itemize}

\end{theorem}

In the statement of Theorems \ref{thm: 1 thm} and \ref{thm Xi=mu} we considered homotopy equivalences which are codimension 0 embeddings. In order to deduce the general case from this setting, we first prove the following stability property:

\begin{theorem}[Theorem \ref{thm: stability} and Proposition \ref{prop: thom cop}]\label{thm: intro stab}
    Let $e: M \hookrightarrow \bR^L$ be an embedding with normal bundle $\nu$; let $D\nu$ be the total space of the unit disc bundle of $\nu$, also a compact manifold. Then the coproducts for $M$ and $D\nu$ agree.  

    Moreover,  assume $M$ is oriented and let $E \to M$ be an oriented vector bundle with disc bundle $DE$. Then the homology-level coproducts $H_*(\Delta)$ for $DE$ and for $M$ are compatible under the Thom isomorphism.
\end{theorem}

The following corollary is immediate from Theorem \ref{thm: intro stab}:
\begin{corollary}\label{cor: si}
    Let $f: N \to Z$ be a simple homotopy equivalence between compact manifolds, then diagram (\ref{diag: fail of hom inv }) commutes, i.e. the spectral coproducts of $N$ and $Z$ agree.  

    If, in addition,  $N$ and $Z$ are closed and oriented and $f$ is orientation-preserving, then their homology-level coproducts agree: $\Delta^{GR}_Z \circ f_* = f_* \circ \Delta^{GH}_N$.

\end{corollary}
Note that \cref{cor: si} has also been proved in recent work of Naef-Safronov \cite{naef2024simple}; see also Remark \ref{rem: cite}. 

\begin{remark}
    Let $B$ be a finite CW complex (not necessarily a manifold). It follows from Corollary \ref{cor: si} that one can define a version of the string coproduct for $B$, by choosing a compact manifold $M$ which is simple homotopy equivalent to $M$, and defining the coproduct for $B$ to be that of $M$. Corollary \ref{cor: si} implies the resulting invariant is well-defined.
\end{remark}

Combining Theorems \ref{thm: 1 thm}, \ref{thm Xi=mu} and \ref{thm: intro stab}, and extending the classes $[T_{diag}], [\overline T_{diag}] $ to any homotopy equivalence $f: N \to Z$, we deduce the main result of our paper:

\begin{theorem}\label{main theorem}
    Let $f: N \to Z$ be a homotopy equivalence of compact manifolds with boundary (of any dimensions). Then the failure of $f$ to respect the spectral Goresky-Hingston coproduct is given by: 
    \begin{equation}
        \Delta^Z \circ (f_! \wedge Id_{S^1}) - (f \wedge f) \circ \Delta^N \simeq   \mu_{{{l}}}(\cdot \times [Z],[T_{diag}]) - \mu_{{{r}}}([\overline T_{diag}], [Z] \times \cdot).
    \end{equation}
\end{theorem}

We now give the corresponding statement on homology. Let $h_*: \Omega^{fr}_*(\cdot) \to H_*(\cdot)$ be the Hurewicz homomorphism.
Using the results of Sections \ref{sec: cop hom comp} and \ref{sec: hom com pro},  which show that after taking homology our spectral constructions agree with their homological counterparts,  we obtain the corresponding homological statement:

\begin{corollary}\label{main cor}

Let $f: N \to Z$ be an orientation-preserving homotopy equivalence of closed oriented manifolds.  Then for all $x \in H_p(\cL N)$: 
\begin{equation}
     \begin{split}
               & \Delta^{GH}\circ f_*(x) - (f\times f)_*\circ \Delta^{GH}(x)\\ &= 
                (-1)^{np+n}\mu^{CS}(f_*(x) \times [Z], h_*[T_{diag}]) - (-1)^{p+n} \mu^{CS}(h_*[\overline T_{diag}], [Z] \times f_*(x)).
            \end{split}
        \end{equation}
        where we take the Chas-Sullivan product in $Z \times Z$.
\end{corollary}

\begin{remark}\label{rem: cite}
    A variant of formula (\ref{cor: mn cor pi 2 =0}), first conjectured by Naef in \cite{Naef}, has been recently proved by Naef-Safronov \cite{naef2024simple} using different methods. Their formula is similar but instead of $h_*[T]$ uses a different homology class; \cref{eq: pi20 case} below implies that when $\pi_2=0$, these homology classes agree. In particular, we expect that in the case $\pi_2=0$, Corollary \ref{cor: mn cor pi 2 =0} recovers \cite[Theorem A]{naef2024simple}.\par 
    
    Another variant of this formula has also been proved recently by Wahl \cite{hingston2019invariance}, using a differently defined obstruction class.  It is natural to conjecture that all of these obstruction classes agree.
\end{remark}
 
        Lastly, when we assume $\pi_2(N)=0$, we can invoke a theorem of Geoghegan and Nicas \cite{GN} which further identifies $[T]$ with the Dennis trace of the Whitehead torsion of $f$. More precisely, let 
        $$tr: K_1(\bZ[\pi_1(M)]) \to HH_1(\bZ[\pi_1(N)])$$ 
        be the classical Dennis trace.  Then after identifying $HH_1(\bZ[\pi_1(N)]\cong H_1(\LL N)$ (which requires the $\pi_2=0$ assumption), and projecting away from constant loops, the content of \cite[Theorem 7.2]{GN} implies that 
        \begin{equation} \label{eq: pi20 case}
            tr(\tau) = h_*[T],
        \end{equation}
        where $\tau$ is the Whitehead torsion of $f$. See Proposition \ref{lem: comp} for a more precise statement. 
        \begin{remark}
            We expect that the condition $\pi_2 =0$ can be removed by lifting the invariants of \cite{GN} to live in topological, rather than ordinary, Hochschild homology. See \cref{conj: para tr}.
        \end{remark}
        Let $tr(\tau)_{diag}$ and $\overline{tr(\tau)}_{diag}$ be the images of $tr(\tau)$ under the antidiagonal maps. Then combining (\ref{eq: pi20 case}) and \cref{main cor} we obtain:

\begin{corollary}\label{cor: mn cor pi 2 =0}

Let $f: N \to Z$ be an orientation-preserving homotopy equivalence of closed oriented manifolds. Suppose that $\pi_2(N)=0$. Then for all $x \in H_p(\cL N)$: 
\begin{equation}
     \begin{split}
        & \Delta^{GH}\circ f_*(x) - (f\times f)_*\circ \Delta^{GH}(x)\\ &= 
        (-1)^{np+n}\mu^{CS}(f_*(x) \times [M], tr(\tau)_{diag}) - (-1)^{p+n} \mu^{CS}( \overline{tr(\tau)}_{diag}, [M] \times f_*(x)).
    \end{split}
\end{equation}
\end{corollary}

     \subsection{Generalisations}

 Let $E \to B$ a be smooth fiber bundle with fiber a smooth closed manifold $M$. Suppose we are given a fiberwise homotopy equivalence $f: E \to M \times B$ over $B$. In future work we hope to show that one can build spectral operations in families and define $\Delta_{\operatorname{fib}}^E$, $\Delta_{\operatorname{fib}}^{B\times M}$,  $\Xi_{{{r}}}^B$, $ \Xi_{{{l}}}^B$, $\mu_{{{r}}}^{M \times B}$ and $\mu_{{{l}}}^{M \times B}$ as morphisms of parametrized spectra. In particular, we conjecture that an analogue of \cref{thm: 1 thm} holds:

 \begin{conjecture}
     \begin{equation}\label{conj: Xi}
        \Delta_{\operatorname{fib}}^{B\times M} \circ  f_!  - f\wedge f \circ \Delta_{\operatorname{fib}}^E = \Xi_{{{r}}}^B - \Xi_{{{l}}}^B.\end{equation}
 \end{conjecture}

We further conjecture that  $\Xi_{{{r}}}^B -\Xi_{{{l}}}^B$ can be characterized in terms of multiplication by higher Reidemeister  traces. Namely, let $\H(M)$ be the stable $h$-cobordism space of $M$.  
Then we expect that one can extend the constructions of \cref{sec: Reid Tr} to define a map: $$RT: \H(M) \to \Omega^{\infty+1}\Sigma^{\infty} \frac{\LL M}{M},  $$ and show:
\begin{conjecture}
There are homotopies of maps of parametrised spectra:
    \begin{equation}
        \Xi_{{{l}}}^B \simeq \mu_{{{l}}}^{M \times B}(\cdot \times [M], [RT_{diag}]) \textrm{ and } \Xi_{{{r}}}^B \simeq \mu_{{{r}}}^{M \times B}( [\overline{RT}_{diag}], [M] \times \cdot).
    \end{equation}
\end{conjecture}

Lastly, to further relate these traces to higher Whitehead torsion, we conjecture a natural generalization of (\ref{eq: pi20 case}) of \cite{GN}:
  \begin{conjecture}\label{conj: para tr}

The following diagram commutes up to natural homotopy:
\[ \begin{tikzcd}
\Omega\Omega^\infty  K[\Sigma^\infty_+\Omega M] \arrow[r] \arrow[d, "\Omega tr"] & \H(M) \arrow[d,"RT"]\\
\Omega\Omega^\infty THH(\Sigma^\infty_+\Omega M) \arrow[r] & \Omega\Omega^\infty \Sigma^\infty(\LL M/M )
\end{tikzcd} \]
 where $tr$ is the Dennis trace on $THH$ due to B\"okstedt \cite{BokstedtTHH}, the top horizontal arrow is given by Waldhausen's splitting theorem, and the bottom arrow is the equivalence: $THH(\Sigma^\infty_+\Omega M) \simeq \Sigma_+^{\infty}\LL M.$
\end{conjecture}

Combined, these conjectures imply that the failure of the Goresky-Hingston coproduct to commute in families can be measured by (suitably interpreted) multiplication with traces of higher Whitehead torsions.

    \subsection{Structure of the paper}
        In Section \ref{sec: prelim} we set up conventions and notations. In Section \ref{sec: coproduct section} we define the spectral Goresky-Hingston coproduct. In Section \ref{sec: prod} we define a version of the spectral Chas-Sullivan product. In Sections \ref{sec: cop hom comp} and \ref{sec: hom com pro} we show that these recover the usual definitions after passing to homology; as an intermediate step, we use models for the string topology operations built using transversality.\par 
        In Section \ref{sec: stab} we show that the spectral string topology operations are invariant under replacing $M$ with the total space of certain disc bundles over $M$. From this, we deduce simple homotopy invariance of the coproduct.\par 
        In Section \ref{sec: torsion} we recall and define fixed-point invariants and operations. In Section \ref{sec: cop defect} we prove Theorem \ref{main theorem} in the special case that $N \to Z$ is a codimension 0 embedding such that the complement $Z \setminus N^\circ$ is an $h$-cobordism. In Section \ref{proof of main thm} we prove Theorem \ref{main theorem} in general, by using results of Section \ref{sec: stab} to reduce to the codimension 0 case.\par 
        Appendix \ref{sec: sp} recaps some notions from stable homotopy theory, with an emphasis on sign conventions.
        
    \subsection{Acknowledgements}
We are grateful to Florian Naef and Nathalie Wahl for helpful conversations, and to the anonymous referee as well as Nathalie Wahl for helpful comments and corrections. Lea would like to thank Mohammed Abouzaid, Roger Casals, Inbar Klang, and Cary Malkiewich for helpful conversations and support, and the president post doctoral fellowship program for professional development and creating excellent work conditions. Noah thanks Ilaria Di Dedda, and Oscar Randal-Williams for helpful conversations, and is supported by the Engineering and Physical Sciences Research Council [EP/W015889/1].

\section{Preliminaries}\label{sec: prelim}
    \subsection{Loops}\label{sec: loops}
        Let $M$ be a smooth Riemannian  manifold. In this section we recall from \cite{hingston2017product} a convenient model for the free loop space of $M$.  \par
        A loop $\gamma: I := [0,1] \to M$ is of \emph{Sobolev class $H^1$} if $\gamma$ and its weak derivative are of class $L^2$.  This means that $\gamma'(t)$ is defined almost everywhere, and the length: \[l(\gamma) = \int_0^1 \gamma'(t)\] is finite and well defined. 
        
        The inclusions:
        \[C^\infty\text{-loops} \subset \text{piecewise }C^\infty \text{-loops} \subset H^1 \text{-loops} \subset {C^0} \text{-loops}\] are homotopy equivalences. See \cite{hingston2017product} and references within. \par
        
        A \emph{constant speed path} is a path $\gamma$ such that $|\gamma'(t)|$ is constant where it is defined. For our model of  the free loop space, $\LL M$, we take the space of constant speed $H^1$ loops. By reparametrising, this space is homotopy equivalent to the space of all $H^1$-loops. 
        Note that this model depends on the metric on $M$, but if $g$ and $g'$ are different metrics on $M$, there is a canonical homeomorphism $\cL (M,g) \to \cL (M,g')$ given by reparametrising all loops; because of this, in the rest of the paper we do not include the metric in the notation.\par 
         
         In our formulas consisting of operations on loops, we always implicitly reparameterise so that the loops are of constant speed. This makes concatenation strictly associative. More explicitly, if $\gamma, \beta: [0,1] \to M $ are two constant speed loops, first define
        $$\sigma = \frac{l(\gamma)}{l(\gamma) +l(\beta)}.$$
        Then the concatenation $\alpha \star \beta $ is given by:
        \begin{equation}\label{eq:concat}
            \alpha\star \beta (t)=
            \begin{cases}
              \gamma(\frac{t}{\sigma}) &\textit{ if } 0\leq t \leq \sigma \\ 
              \beta(\frac{t-\sigma}{1-\sigma}) &\textit{ if } \sigma \leq t \leq 1 .\\
            \end{cases}
        \end{equation}
        The same convention is used in \cite[Section 1]{hingston2017product}.\par 
        For the purpose of readability, we use the following notation for concatenation of paths. Given a path $\gamma$ from $x$ to $y$ and a path $\delta$ from $y$ to $z$, we write
        \begin{equation}\label{eq: concat notation}
              x \overset \gamma \rightsquigarrow y \overset \delta \rightsquigarrow z
        \end{equation}
      
        for the \textit{constant speed} concatenation of the two paths.
    \subsection{Suspensions}
        We will write many explicit formulas for maps into or out of suspensions of based spaces so we choose which model for the suspension functor we work with. 
        \begin{definition}\label{def: susp}

            For $L \geq 0$, we give two models for $\Sigma^L X$:
            \begin{enumerate}
                \item \label{def: sus 1}
                    $$\frac{[-1,1]^L \times X} {\left(\partial [-1,1]^L \times X\right)\cup \left( [-1,1]^L \times \{*\}\right)}$$
                 \item \label{def: sus 2}
                    
                    $$\frac{\bR^L \times X} {\left((\bR^L \setminus (-1,1)^L) \times X\right) \cup \left( \bR^L \times \{*\}\right)}$$
                \end{enumerate} 
                in both cases based at the point which is the image of the collapsed subspace.\par 
                In both cases, if $X$ is equipped with a basepoint $x_0$, we further quotient by $[-1,1]^L \times \{x_0\}$.\par 
                We will use these two models interchangeably, noting they are canonically homeomorphic.
        \end{definition}
        
\section{Spectral Goresky-Hingston coproduct }\label{sec: coproduct section}
\subsection{Preamble}\label{sec: preamble cop}
Let $M$ be a compact smooth manifold, possibly with corners. The main goal of this section is to define and study a realization of the Goresky-Hingston coproduct as a map of spectra.

Fix an embedding $e: M \to \R^L$, and let $\nu_e$ be the normal bundle (defined to be the orthogonal complement of $de(TM)$) equipped with the pullback metric. Denote by $D\nu_e$ and $S\nu_e$ the corresponding unit disk and sphere bundles respectively. 

Let
\begin{equation}\label{eq: eval}
    \operatorname{ev}_0: \LL M \to M,
\end{equation} be the evaluation map sending  $\gamma \mapsto \gamma(0)$. We use $\operatorname{ev}_0$ to pull back $\nu_e$ to a bundle which, by abuse of notation, we write as $\nu_e \to \LL M$.  The Thom space, $\LL M^{D\nu_e}$, is defined by:
\begin{equation} \label{eq: Thom}
    \cL M^{D\nu_e} := \operatorname{Tot}(D\nu_e \to \cL M)/\operatorname{Tot}(S\nu_e \to \cL M),
\end{equation}
where  $\operatorname{Tot}$ refers to taking the total space. Similarly to the case of suspensions, this is canonically homeomorphic to:
\begin{equation}
    \cL M^{D\nu_e} \cong \operatorname{Tot}(\nu_e \to \cL M)/\left(\operatorname{Tot}(\nu_e \to \cL M) \setminus \operatorname{Tot}(D\nu_e \to \cL M)^\circ\right)
\end{equation}

Let  $\cL M^{-TM}$ be the spectrum given by desuspending this Thom space. That is, it is the sequential spectrum
whose $i^{th}$ space, for $i \gg 0$,  is given by:

 \begin{equation*}
    \cL M^{-TM}_i := \cL M^{D(\bR^{i-L} \oplus \nu_e)}.
\end{equation*}

In Section \ref{def: closed cop} we describe the Goresky-Hingston coproduct as a map of spectra:  
$$\Delta:  \LL M^{-TM} \wedge S^1  \to \Sigma^\infty \frac{\LL M}{M} \wedge\frac{\LL M}{M} $$ 
for a \textit{closed }smooth manifold. The definition in this case is more transparent and requires less choices than the general case, but already contains most of the main ideas.

In Sections \ref{sec: CD} and \ref{sec: Coproduct} we treat the more general case of smooth compact manifolds with corners, and define a map:

$$\Delta: \frac{\cL M^{-TM}}{\partial \cL M^{-TM}} \wedge S^1 \to \Sigma^\infty \frac{\cL M}M \wedge \frac{\cL M}M, $$
where $\partial \LL M := \LL M|_{\partial M}$ denotes the space of loops $\gamma$ such that $\gamma(0) \in \partial M$. 

 We keep track of all the choices involved in the definition, and prove independence of choices in \cref{lem: cop ind}. In Section \ref{sec: stability} we prove a stability property, from which we deduce simple homotopy invariance of the coproduct.

\subsection{The closed case}\label{sec: closed case}
In this section $M$ is a smooth closed manifold of dimension $n$. Let $e, \nu_e,$ and $D\nu_e$ be as in \cref{sec: preamble cop}.   We identify $D{\nu_e}$ with an $\eps$-tubular neighborhood $U \subset \R^L$ by an embedding $\rho:D{\nu_e} \to U $. Let $\pi: D{\nu_e} \to M$ be the projection and $r: U \to M$ the retraction defined by $r = \pi \circ \rho^{-1}$. Note that we can choose $\rho$ and $\eps $ so that $r(u)$ is always the closest point to $u$ in $M$. \par

Recall from \cref{eq:concat} and \cref{eq: concat notation} our conventions and notation for the concatenation of paths. Moreover, suppose $x, y \in U \subset \R^L$ are such that $U$ contains the straight line path between $x$ and $y$. Denote by \begin{equation}\label{eq: theta}
    x \overset{\theta} \rightsquigarrow y
\end{equation} its retraction to $M$ using $r$.

\begin{definition}\label{def: closed cop}

    Let $(v, \gamma, t) \in \cL M^{D\nu_e}\wedge S^1$. That is,  $\gamma \in \cL M$, $t \in S^1$ and $v \in (D\nu_e)_{\gamma(0)}$. The \emph{unstable coproduct} is the map of spaces:
    $$\Delta_{unst}: \cL M^{D\nu_e}\wedge S^1 \to \Sigma^L \frac{\LL M}{M} \wedge \frac{\LL M}{M}$$
    sending $(v, \gamma, t)$ to:
    
    \begin{equation}\label{eq:closed case}
      \begin{cases}
        \left(\frac{\sqrt L} \eps \left(v-\gamma(t)\right), 
        \gamma(0)\overset{\gamma|_{[0,t]}}\rightsquigarrow \gamma(t)\overset{\theta}\rightsquigarrow\gamma(0), 
        \gamma(0) \overset \theta \rightsquigarrow 
        \gamma(t) \overset{\gamma|_{[t,1]}}\rightsquigarrow \gamma(0)\right) & \textrm{ if } \lVert v-\gamma(t)\rVert \leq \eps\\
        * & \textrm{ otherwise}.
    \end{cases}
    \end{equation}
    where we perform the subtraction in $\bR^L$. This is illustrated in Figure \ref{fig:coprod_closed}.
    
    Here, and in other formulae, we write $\lVert \cdot \rVert$ to denote the $L^2$-norm on $\bR^L$.
    
    The \emph{(stable) coproduct}: 
    $$\Delta:  \LL M^{-TM} \wedge S^1  \to \Sigma^\infty \frac{\LL M}{M} \wedge\frac{\LL M}{M} $$
    is obtained from the unstable coproduct by desuspending $\Delta_{unst}$ $L$ times (see Lemma \ref{lem: spectra unst}).
    
\end{definition}

\begin{figure}[h]  
    \centering
    \begin{minipage}{.25 \textwidth}
        \begin{tikzpicture}[scale=.8,
        v/.style={draw,shape=circle, fill=black, minimum size=1.3mm, inner sep=0pt, outer sep=0pt},
        vred/.style={draw,shape=circle, fill=red, minimum size=1mm, inner sep=0pt, outer sep=0pt},
        vsmall/.style={draw,shape=circle, fill=black, minimum size=1mm, inner sep=0pt, outer sep=0pt}]

            \draw[color=black, ->] (0,0) to[out=90, in=90] (2,2);
            \draw[color=black] (2,2) to[out=270,in=90] (0.5, 0);

            \draw[color=black] (0,0) to[out=-90, in=-90] (2,-2);
            \draw[color=black] (2,-2) to[out=-270,in=-90] (0.5, 0);

            \draw[thick, teal, dashed] (0.8,0.4) arc (0:360: 1.2);

            \draw[thin, cyan, ->](0,0) to (-0.2,0.2);
            \draw[thin, cyan] (-0.2, 0.2) to (-0.4, 0.4);
            \node at (-0.6, 0.6) {\tiny ${\color{cyan} v}$};

            \draw[thin, purple, ->] (0.5, 0) to (0.25, 0);
            \draw[thin, purple] (0.25, 0) to (0,0);
            \node at (0.3, 0.2) {\tiny ${\color{purple} \theta}$};

            \node at (0, 2.5) {$(v, \gamma, t)$};
            \node at (-0.3, -0.3) {\tiny ${\color{red}  \gamma(0)}$};
            \node at (1, -0.1) {\tiny ${\gamma(t)}$};
            \node at (2.1, 1.5) {\tiny $\gamma$};

            \node at (-1.5, 1.6) {\tiny ${\color{teal}B_\eps(v)}$};

            \node[v] at (0,0) {};
            \node[vsmall] at (0.5,0) {};
            \node[vsmall] at (-0.4, 0.4) {};
            \node[vred] at (0,0) {};
        \end{tikzpicture}
    \end{minipage}
    $\mapsto \left(
    \begin{minipage}{.37 \textwidth}
        \begin{tikzpicture}[scale=.8,
        v/.style={draw,shape=circle, fill=black, minimum size=1.3mm, inner sep=0pt, outer sep=0pt},
        vred/.style={draw,shape=circle, fill=red, minimum size=1mm, inner sep=0pt, outer sep=0pt},
        vsmall/.style={draw,shape=circle, fill=black, minimum size=1mm, inner sep=0pt, outer sep=0pt}]
        
            \draw[color=purple, -<] (-1.15,-0.15) to (-0.9,-0.15);
            \draw[color=purple] (-0.9,-0.15) to (-0.65,-0.15);
            \draw[color=cyan] (-1.15, -0.15) to (-1.55, 0.25);
            \draw[color=cyan, ->] (-1.15, -0.15) to (-1.35, 0.05);

            \begin{scope}[shift={(-0.6,0)}]
            \node at (0.3, -0.2) {,};
        
            \draw[color=black, ->] (0.7,0) to[out=90, in=90] (2.7,2);
            \draw[color=black] (2.7,2) to[out=270,in=90] (1.2, 0);

            \draw[thin, purple, ->] (1.2, 0) to (0.95, 0);
            \draw[thin, purple] (0.95, 0) to (0.7,0);

            \node[v] at (0.7,0) {};
            \node[vsmall] at (1.2,0) {};
            \node[vred] at (0.7,0) {};

            \node at (2.2, -0.2) {,};

            \draw[color=black, -<] (2.7,0) to[out=-90, in=-90] (4.7,-2);
            \draw[color=black] (4.7,-2) to[out=-270,in=-90] (3.2, 0);

            \draw[thin, purple, -<] (3.2, 0) to (2.95, 0);
            \draw[thin, purple] (2.95, 0) to (2.7,0);

            \node[v] at (2.7,0) {};
            \node[vsmall] at (3.2,0) {};
            \node[vred] at (2.7,0) {};

            \node at (3.2,1.5) {\tiny $\gamma|_{[0,t]}$};
            \node at (2.8, -1.5) {\tiny $\gamma|_{[t,1]}$};
            \end{scope}
        \end{tikzpicture}
    \end{minipage}
    \right)$
    \caption{Coproduct in the closed case: the figure on the left shows a triple $(v,\gamma, t)$ in the domain of the coproduct. The figure on the right shows the output: the first component is the sum of the two vectors indicated, scaled appropriate by a factor of ${\sqrt L}/\eps$. The incidence condition holds because $\gamma(t)$ lies in the ball $B_\eps(v)$.
    }
  \label{fig:coprod_closed}
\end{figure}
\begin{remark}\label{rmk: radius}
    The $\sqrt L$ arises as the maximum of the $L^2$-norm on $[-1, 1]^L$.
\end{remark}
    
For sufficiently small $\eps$, the map $\Delta_{unst}$ is a well-defined continuous map. Indeed, first note that for sufficiently small $\eps$, if $\lVert v - \gamma(t)\rVert \leq \eps$ then the straight-line path connecting $v$ and $\gamma(t)$ lives in $U$, so the paths $\gamma(t)\overset{\theta}\rightsquigarrow\gamma(0)$ and  $\gamma(0) \overset \theta \rightsquigarrow 
        \gamma(t)$ are well defined. 

\begin{definition}
    For equations of the form of (\ref{eq:closed case}), we call the ``if'' condition (so $\lVert v-\gamma(t)\rVert \leq \eps$ in the case of (\ref{eq:closed case})) the \emph{incidence condition}.
\end{definition}

Secondly, we defined $\Delta_{unst}$ using coordinates on $\operatorname{Tot}(D\nu_e \to \LL M) \times I$. To show that it descends to the quotient $\cL M^{D\nu_e}\wedge S^1$, we need to check that when either $\lvert \rho^{-1}(v) \rvert = 1$, $t=0$, or $t=1$, $(v, \gamma, t)$ is sent to the basepoint (here $|\cdot|$ is the norm coming from the chosen metric on the vector bundle). Note that $v$ is a normal vector at $\gamma(0)$ and that we chose the tubular neighborhood $U$ so that $\gamma(0)$ is the closest point to $v$ in $e(M)$. This means that when $\lvert \rho^{-1}(v) \rvert = 1$, $\lVert v -\gamma(t) \rVert \geq \varepsilon$ for every $t$, hence the first entry in \cref{eq:closed case} has $\lVert \cdot \rVert \geq \sqrt L$ and so (using Remark \ref{rmk: radius}) $(v, \gamma, t)$ is sent to the basepoint.

Moreover, when $t=0$, the retraction of the straight line path from $v$ to $\gamma(0)$ is the constant path at $\gamma(0)$, since $\gamma(0)$ is the closest point to $v$ in $M$. This implies that the second argument in \cref{eq:closed case} is sent to the base point. The case of $t=1$ is similar. \par 
We treat independence of choices when we deal with the general case in Lemma \ref{lem: cop ind}.

\subsection{Choices}\label{sec: CD}
In this section we collect all the choices required for our definition of the coproduct when $M$ is a manifold with corners. 

 To define the coproduct we require an embedding $e: M \to \R^L$, and a tubular neighborhood of $e(M)$. In order to extend the definition of a tubular neighborhood to manifolds with corners, we consider a small ``extension'' of $M$, denoted $M^{ext}$, and containing $M$ as a codimension $0$ submanifold: 
\begin{definition}\label{def: ext mnfld}
 Let $M$ be a smooth compact manifold with corners. As a topological manifold, $M^{ext}$ is given by $$M^{ext} := M \cup_{\partial M} \partial M \times [0,1].$$ To equip  $M^{ext}$ with a smooth structure we choose a vector field  on $M$ which points strictly inwards at the boundary.  Let $\{\phi^s\}_{s \geq 0}$ be the associated flow. Then there is a homeomorphism $\Phi: M^{ext} \to M$ sending $x \in M$ to $\phi^1(x)$, and $(y,t) \in M \times [0,1]$ to $\phi^{1-t}(y)$. We equip $M^{ext}$ with the pullback of the smooth structure on $M$. Note that $M^{ext}$ contains a copy of $M$, which is a codimension 0 submanifold with corners. Furthermore the canonical projection map $M^{ext} \to M$ is piecewise smooth.
\end{definition}

The auxiliary data required to define the string coproduct for $M$ is as follows:
\begin{definition}\label{def: cop dat}
    Let $L \geq 0$ be an integer. A choice of \emph{embedding data} of \emph{rank $L$} is a tuple $(e, \rho^{ext}, \zeta, V, \eps, \lambda)$ consisting of:
    \begin{enumerate}[(i).]
        \item \label{it: embedding}A smooth embedding $e: M^{ext} \hookrightarrow \R^L$. 
        
        \noindent We write $\nu_e$ for the normal bundle of this embedding, defined to be the orthogonal complement of $TM^{ext}$. Note that $e$ canonically equips both $TM^{ext}$ and $\nu_e$ with metrics, by pulling back the Euclidean metric on $\R^L$. Let $\pi_e: \nu_e \to M^{ext}$ be the projection map.
        \item \label{it: tubular ngh} 
        A tubular neighbourhood $\rho^{ext}: D_2 \nu_e \hookrightarrow \bR^L$, where $D_2$ denotes the length-2 disc bundle.
        
        \noindent More precisely, a smooth embedding, restricting to $e$ on the zero-section. We let $\tilde U$ be the image of $\rho^{ext}$. We let $\rho$ be the restriction of $\rho^{ext}$ to the unit disc bundle of $\nu_e$ over $M$, and $U$ the image of $\rho$. In symbols: $\rho :=\rho^{ext}|_{D_1 \nu_e|_M}$, $U:= \operatorname{Im}(\rho)$ and $\tilde U = \operatorname{Im}(\rho^{ext})$. From the choices above we obtain a retraction $r: \tilde U \to M$ defined to be the composition of $(\rho^{ext})^{-1}$, the projection to $M^{ext}$, and the natural map $M^{ext} \to M$.
        
        \noindent We require that along the zero section, the derivative of $\rho^{ext}$ in the fibre direction agrees with the canonical inclusion of vector bundles $\nu_e \to \bR^L$.
        \item \label{item: zeta small cop} A real number $\zeta > 0$. 
        
        \noindent We require that \textrm{$\zeta$} is small enough that whenever $x, y \in M$ satisfy $\lVert x-y\rVert \leq \zeta$, the straight-line path between them $[x,y]$ lies inside $\tilde U$.
        \item \label{item: V small} An inwards-pointing vector field, $V$, on $M$. 
        
        \noindent We write $\{\phi_s\}_{s \geq 0}$ for the flow of this vector field. We require that $V$ is small enough that the following condition holds: 
        for each $x \in M$, the length of the path $\{\phi_s(x)\}_{s \in [0,1]}$ is $\leq \zeta/4$.
      
        \item \label{it: eps bounds} A real number $\eps > 0$ sufficiently small such that:

        \begin{enumerate}[(a).]
            \item \label{item: eps small 1} \textrm{$U$ contains an $\eps$-neighbourhood of $M$ (with respect to the Euclidean distance in $\bR^L$).}
            \item \label{item: eps small 2}
            The Euclidean distance: $d\left(\rho(D\nu|_{\phi_1(M)}), \rho(D\nu|_{\partial M}))\right) \geq 2\eps$
            \item \label{item: eps small 3}
            If $x, y \in U$ and $\lVert x-y \rVert \leq \eps$, then the straight-line path $[x,y]$ lies in $\tilde U$, and $r([x,y])$ has length $\leq \zeta/4$.
        \end{enumerate}

        If this final condition holds, we write $\theta_{xy}$ (or just $\theta$ if the endpoints are clear from context) for the path $r([x,y])$.
    
        \item\label{item: lambda big} $\lambda > 0$, large enough such that:
        
         \[\lambda \cdot d(\rho(S\nu_e|_M), e(M)) \geq \sqrt L\]  
     where $S\nu_e$ is the unit sphere bundle of $\nu_e$; note that this distance on the left hand side is at least $\eps$, by (\ref{def: cop dat}.\ref{item: eps small 1}).
    \end{enumerate}
    We write $ED^L(M)$ for the simplicial set whose $k$-simplices consist of the set of continuously-varying families of tuples of {embedding data}, parametrised by the standard $k$-simplex. There is a forgetful map $ED^L(M) \to \operatorname{Emb}(M^{ext}, \R^L)$ to the simplicial set of embeddings $M^{ext} \hookrightarrow \R^L$, which forgets all the data except the embedding $e$.
\end{definition}

\begin{figure}[h] \label{fig: cop data}
    \centering
    \begin{minipage}{.5 \textwidth}
        \begin{tikzpicture}[scale=.8,
        v/.style={draw,shape=circle, fill=black, minimum size=1.3mm, inner sep=0pt, outer sep=0pt},
        vred/.style={draw,shape=circle, fill=red, minimum size=1mm, inner sep=0pt, outer sep=0pt},
        vsmall/.style={draw,shape=circle, fill=black, minimum size=1mm, inner sep=0pt, outer sep=0pt}]
            \draw[color=black] (0,1) to[out=240, in=120] (0,-1);
            \draw[color=black, dashed] (0,1) to[out=300, in=60] (0,-1);
            \draw[color=black] (0,1) to (-3,1);
            \draw[color=black] (0,-1) to (-3,-1);
            \draw[color=black] (-3,1) .. controls (-5.5, 1) and (-5.5, -1) .. (-3,-1);

            \draw[color=black] (-3,0.1) to[out=330, in=210] (-2,0.1);
            \draw[color=black] (-2.75, 0) to [out=30, in=150] (-2.25, 0);

            \node at (-2,-0.8) {\tiny ${\color{black}e(M)}$};

            \draw[color=purple] (0,1) to (1,1);
            \draw[color=purple] (0,-1) to (1,-1);
            \draw[color=purple] (1,1) to[out=240,in=120] (1,-1);
            \draw[color=purple, dashed] (1,1) to[out=300, in=60] (1,-1);

            \node at (1.3,-1.3) {\tiny ${\color{purple}e\left(M^{ext} \setminus M^\circ\right)}$};

            \draw[color=blue] (0,1.5) to[out=240, in=120] (0,-1.5);
            \draw[color=blue,dashed] (0,1.5) to[out=300,in=60] (0,-1.5);
            \draw[color=blue] (0,1.5) to (-3.5,1.5);
            \draw[color=blue] (0,-1.5) to (-3.5, -1.5);
            \draw[color=blue] (-3.5,1.5) .. controls (-6, 1.5) and (-6, -1.5) .. (-3.5,-1.5);

            \node at (-2,-1.3) {\tiny ${\color{blue}U}$};

            \draw[color=olive] (1,2) to[out=240, in=120] (1,-2);
            \draw[color=olive, dashed] (1,2) to[out=300,in=60] (1,-2);
            \draw[color=olive] (1,2) to (-4,2);
            \draw[color=olive] (1,-2) to (-4, -2);
            \draw[color=olive] (-4,2) .. controls (-6.5, 2) and (-6.5, -2) .. (-4,-2);

            \node at (-2,-1.8) {\tiny ${\color{olive}\tilde U}$};

            \draw[color=red, ->] (-0.3, 0) to (-1,0);
            \draw[color=red, ->] (-0.2, 0.5) to (-0.9,0.5);
            \draw[color=red, ->] (-0.2, -0.5) to (-0.9,-0.5);
            \draw[color=red, dashed, ->] (0.25, 0.3) to (-0.45, 0.3);
            \draw[color=red, dashed, ->] (0.25, -0.3) to (-0.45, -0.3);
            \node at (-1.2, 0.3) {\tiny ${\color{red}V}$};

        \end{tikzpicture}
    \end{minipage}
    \caption{Some choices in the definition of the coproduct: $e(M)$, ${\color{purple}e(M^{ext})}$, ${\color{blue}U}$, ${\color{olive}\tilde U}$ and ${\color{red}V}$ are shown.}
\end{figure}

\begin{remark}
    These conditions are used in Lemma \ref{lem: Delta well def} to ensure that the map we use to define the coproduct is well-defined. We indicate how they are used:
    \begin{itemize}
   \item  In Condition (\ref{def: cop dat}.\ref{it: tubular ngh})  we give a precise definition of the tubular neighborhood needed for the definition of the coproduct. The somewhat cumbersome definition stems from the fact that we are dealing with manifolds with boundary or corners. 
        \item Condition (\ref{def: cop dat}.\ref{item: zeta small cop}) is used in Lemma \ref{lem: B}, which allows us to discard small loops, of length $ < \zeta$.
        \item The choice of vector field, $V$ in (\ref{def: cop dat}.\ref{item: V small}), and the bounds (\ref{def: cop dat}.\ref{it: eps bounds}) are used so that the coproduct sends loops with starting point in $\partial M$ to the base point. 
        \item The choice of $\lambda$ in (\ref{def: cop dat}.\ref{item: lambda big}) is a logistical choice, so we can avoid excessive rescaling. It used in ensuring that the coproduct descends to the Thom space. 
    \end{itemize}
\end{remark}
\begin{remark}\label{remark: emb corn}
    In the case where a manifold $M$ has corners, a \emph{(codimension $k$) embedding} $M \to P$ means a smooth embedding on the interior $M^\circ$, which near the boundary $\partial M$ looks like the standard inclusion of $\bR^i \times \bR^{n-i}_{\geq 0} \to \bR^n \times \bR^k$ (with respect to some smooth charts on $M$ and $P$). In particular, even though $\partial M$ is not everywhere smooth, the restriction of the embedding to the boundary $\partial M \to P$ is a locally flat embedding of topological manifolds.

    For a manifold with corners $M$, a \emph{collar neighbourhood} $C: A \to M$ is the inclusion of a neighbourhood of the boundary $\partial M$, whose image is a smooth submanifold of codimension 0. For any $M$, there is always such a collar neighbourhood $C: A \to M$.
\end{remark}
\begin{lemma}\label{lem: kan}
    The forgetful map $ED^L(M) \to \operatorname{Emb}(M^{ext}, \R^L)$ is a trivial Kan fibration and hence a weak equivalence.
    
\end{lemma}
It follows that $ED^L(M)$ is $(L-2n-3)$-connected.
\begin{proof}
    We let $ED_i^L(M)$ be the simplicial set consisting of tuples consisting of the first $i$ pieces of data of a choice of {embedding data}; note that the conditions that each piece of data in Definition \ref{def: cop dat} must satisfy only involve earlier pieces of data. Then $ED_6^L(M) = ED^L(M)$ and $ED^L_1(M) = \operatorname{Emb}(M^{ext}, \R^L)$. There are forgetful maps $ED_i^L(M) \to ED_{i-1}^L(M)$; we argue that each of these is a trivial Kan fibration.\par 
    It is standard that $ED_1^L(M)$ is a Kan complex. We next argue that the first forgetful map $ED_2^L(M) \to ED^L_1(M)$ is a trivial Kan fibration. Let $k \geq 0$, and consider the lifting problem:
    \begin{equation} 
        \begin{tikzcd}
            \partial \Delta^k
            \arrow[d]
            \arrow[r, "\rho^{ext}"]
            &
            ED^L_2(M)
            \arrow[d]
            \\
            \Delta^k
            \arrow[r, "e"]
            \arrow[ur, dashed, "\widetilde \rho^{ext}"]
            &
            ED^L_1(M)
        \end{tikzcd}
    \end{equation}
    Explicitly, $e=\{e_t: M^{ext} \to \bR^L\}_{t \in \partial \Delta^k}$ and $\rho^{ext} = \{\rho^{ext}_t: D_2\nu_{e_t} \to \bR^L\}_{t \in \partial \Delta^k}$ are families of embeddings satisfying appropriate conditions, including that along the zero-sections, $\rho^{ext}|_{\partial \Delta^k}$ agrees with $e$. We must argue the lift $\widetilde \rho^{ext}$ exists.

    We may choose a family of maps (not necessarily embeddings) $\{\rho'_t: D_2\nu_{e_t} \to \bR^L\}_t$ satisfying the condition on their derivative (that the vertical derivative is the natural inclusion), by defining $\rho'_t$ to be given by $e_t$ along the zero-section and requiring it to be a straight line in the fibre direction. The $\rho'_t$ do not necessarily agree with the given $\rho^{ext}_t$ for $t \in \partial \Delta^k$, but we may homotope the $\rho'_t$ so that they do agree with the $\rho^{ext}_t$ over $\partial \Delta^k$, by linear interpolation (note this does not change what happens to the zero-sections, nor does it affect the derivative condition).

    Next, we ``scale down in the fibre direction''. Choose a family of positive numbers $\{a_t > 0\}_{t \in \Delta^k}$, such that $a_t = 1$ for $t \in \partial \Delta^k$. Let $\widetilde\rho^{ext}_t(x) = \rho'_t(a_t \cdot x)$ for $t \in \Delta^k$ and $x \in D_2\nu_{e_t}$. For $a_t$ small enough for $t \notin \partial \Delta^k$, the $\widetilde\rho^{ext}_t$ are embeddings, by the implicit function theorem. Furthermore, $\widetilde\rho^{ext}_t = \rho^{ext}_t$ for $t \in \partial \Delta^k$, thus providing the desired lift.

    For the second forgetful map, note that the condition for $\zeta$ holds for sufficiently small $\zeta$; similarly (\ref{def: cop dat}.\ref{item: V small}) holds for any sufficiently small vector fields $V$. Similarly for $\eps$ (respectively $\lambda$), any sufficiently small (respectively large) choice will satisfy the required conditions. All of these arguments also work for families over a simplex, implying that each forgetful map is a trivial Kan fibration.
\end{proof}

\subsubsection{Stabilization}\label{sec: st maps}
There are stabilisation maps: \begin{equation}\label{eq: st maps}
    st=st^{L,L+1}: ED^L(M) \to ED^{L+1}(M)
\end{equation}
constructed by sending $$(e, \rho^{ext}, \zeta, V, \eps, \lambda) \mapsto (e', \rho'^{ext}, \zeta, V, \eps, \tfrac{\sqrt{L+1}}{\sqrt{L}}\cdot \lambda).$$ Here $e'$ is given by composing $e$ with the standard embedding $\R^L \hookrightarrow \R\oplus \R^L = \R^{L+1} $, and $\rho'^{ext}$ is the composition:
$$\rho'^{ext}: D_2 \nu_{e'} = D_2  \left(\R \oplus \nu_{e'}\right) \subseteq [-2, 2] \times D_2 \nu_{e'}   \to  \R \oplus \R^L = \R^{L+1},$$
where the final arrow is inclusion on the first factor and  $\rho^{ext}$ on the last factor. It is clear that these are compatible with the natural inclusion, $st_{\operatorname{Emb}}: \operatorname{Emb}(M^{ext}, \R^L) \to \operatorname{Emb}(M^{ext}, \R^{L+1})$, given by composing with the inclusion $\bR^L \cong \{0\} \times \bR^L \hookrightarrow \bR^{L+1}$. Also note that there are natural identifications $\nu_{e'} = \bR \oplus \nu_e$. It is straightforward to check that this data does indeed define {embedding data}.\par 
For $L \leq L'$, we write $st^{L,L'}: ED^L(M) \to ED^{L'}(M)$ for the composition of $L'-L$ stabilisation maps.

\subsection{Coproduct}\label{sec: Coproduct}
Let $M$ be a smooth manifold with corners. In this section we define the coproduct as a map of spectra:
$$\Delta: \frac{\cL M^{-TM}}{\partial \cL M^{-TM}} \wedge S^1 \to \Sigma^\infty \frac{\cL M}M \wedge \frac{\cL M}M,$$ by defining it first unstably as  a map of spaces: 
\begin{equation}\label{eq: uns cop}
    \Delta_{unst}=\Delta^Q_{unst}: \frac{\cL M^{D\nu_e}}{\partial \cL M^{D\nu_e}}\wedge S^1 \to \Sigma^L \frac{\cL M}M \wedge \frac{\cL M}M,
\end{equation}
for a fixed choice of {embedding data} $Q$ for $M$.\par 

Before stating the definition of $\Delta_{unst}$ and $\Delta$, we define a map $$B: \LL M \to \LL M$$  which ``crushes'' small loops to constant loops. More precisely:  
\begin{lemma}\label{lem: B}
    Let $Q \in ED^L(M)$ be {embedding data}. Note that the embedding $e: M \to \R^L$ induces a metric on $M$.  Let $\LL M ^{\leq \zeta}$ be the subset of $\LL M$ consisting of loops of length less than $\zeta$.  Then there exists a map:
    $$B=B^Q: \cL M \to \cL M,$$ 
    homotopic to the identity (relative to the space of constant loops) and continuously varying in $Q$, which sends $\LL M ^{\leq \zeta}$ to constant loops.  
\end{lemma}
\begin{proof} 
    Let $M \subset \LL M$ be the inclusion of constant loops. 

    Let  $s_\gamma: \LL M \to [0,1]$ be the continuous function defined by  
    \[s_\gamma = \text{max } \{\textit{t } | \ell(\gamma_{[0,t]} )\leq \zeta\} \]
    where $\ell$ denotes Riemannian length. Define a homotopy $H: \LL M \times [0,1] \to \LL M$ to send $(\gamma, \tau)$ to

    \begin{equation} \label{eq: wsgwsoerh}
        \gamma(0) \overset{\gamma_{[0, \tau s_\gamma ]}}\rightsquigarrow \gamma(
    \tau s_\gamma) \overset{\theta} \rightsquigarrow \gamma(s_\gamma) \overset{\gamma_{[s_\gamma, 1]}} \rightsquigarrow \gamma(1),
    \end{equation}
    noting that the path $\gamma(
    \tau s_\gamma) \overset{\theta} \rightsquigarrow \gamma(s_\gamma)$ is well-defined, by (\ref{def: cop dat}.\ref{item: zeta small cop}). Then $H_1$ is the identity. Moreover, the subset $\LL M^{\leq \zeta}$ is sent by $H_0$ to the subset of constant loops, because $s_\gamma = 1$ whenever $\ell(\gamma) \leq \zeta$ by definition, so if $\tau=0$ all three paths in (\ref{eq: wsgwsoerh}) are constant.
\end{proof}
We now proceed with the definition of $\Delta_{unst}$. 
\begin{definition}\label{def: cop}
   Fix {embedding data} $Q$ for $M$. The \emph{unstable coproduct}, $\Delta_{unst}=\Delta^Q_{unst},$ is  the map of spaces:
    \begin{equation*}
        \Delta_{unst}: \frac{\cL M^{D\nu_e}}{\partial \cL M^{D\nu_e}} \wedge S^1 \to \Sigma^L \frac{\cL M}M \wedge \frac{\cL M}M
    \end{equation*}
    defined as follows. Let $(v, \gamma, t) \in  \frac{\cL M^{D\nu_e}}{\partial \cL M^{D\nu_e}} \wedge S^1$: so $t \in [0,1], \gamma \in \cL M$, and $v \in D\nu_e$ lies in the fibre over $\gamma(0)$. Then
    \begin{equation}\label{eq: Delta unst}
    \Delta_{unst}(v, \gamma, t)=\begin{cases}
    \begin{pmatrix}
        \lambda \left(v-\phi_1 \circ \gamma(t)\right),\\
        B\left(
        \gamma(0) \overset{\gamma|_{[0,t]}}\rightsquigarrow
        \gamma(t) \overset{\phi}\rightsquigarrow
        \phi_1 \circ \gamma(t) \overset{\theta}\rightsquigarrow
        \gamma(0)
         \right),\\
        B\left(
        \gamma(0) \overset{\theta}\rightsquigarrow
        \phi_1 \circ \gamma(t) \overset{\overline \phi}\rightsquigarrow 
        \gamma(t) \overset{\gamma|_{[t,1]}}\rightsquigarrow
        \gamma(0)
         \right)\\
    \end{pmatrix}
    & \textrm{ if } \lVert v - \phi_1 \circ \gamma(t)\rVert \leq \varepsilon\\
    * & \textrm{ otherwise.}
    \end{cases}\end{equation}
    Note that we have used Convention (\ref{def: susp}.\ref{def: sus 2}) for the target.  The path  $\gamma(0) \overset{\theta=\theta_{v, \gamma(0)}}\rightsquigarrow
        \phi_1 \circ \gamma(t)$  is defined as in \cref{eq: theta}, and $\gamma(t) \overset{\phi}\rightsquigarrow
        \phi_1 \circ \gamma(t)$ denotes the path given by the flow of $\phi$. 
\end{definition}
See Figure \ref{fig:coprod} for a picture.
\begin{remark}
    The second and third entries in (\ref{eq: Delta unst}) each consist of three paths concatenated, but not all are of equal importance: the paths $\phi, \overline \phi$ and $\theta$ are all ``small'' and their purpose is to ensure the start and endpoint of the path are the same, whereas the paths $\gamma|_{[0,t]}$ and $\gamma|_{[t,1]}$ are ``big'' and are the ones which are ``morally'' important.
\end{remark}
\begin{remark}
    When $M$ is closed, for an appropriate choice of {embedding data} $Q$, the coproduct in Definition \ref{def: cop} is homotopic to the coproduct in Definition \ref{def: closed cop}, by applying Lemma \ref{lem: B}.
\end{remark}

\begin{figure}[h]  
    \centering
    \begin{minipage}{.25 \textwidth}
        \begin{tikzpicture}[scale=.8,
        v/.style={draw,shape=circle, fill=black, minimum size=1.3mm, inner sep=0pt, outer sep=0pt},
        vred/.style={draw,shape=circle, fill=red, minimum size=1mm, inner sep=0pt, outer sep=0pt},
        vsmall/.style={draw,shape=circle, fill=black, minimum size=1mm, inner sep=0pt, outer sep=0pt}]

            \draw[color=black, ->] (0,0) to[out=90, in=90] (2,2);
            \draw[color=black] (2,2) to[out=270,in=90] (1, 0);

            \draw[color=black] (0,0) to[out=-90, in=-90] (2,-2);
            \draw[color=black] (2,-2) to[out=-270,in=-90] (1, 0);

            \draw[thick, teal, dashed] (0.7,0.4) arc (0:360: 1.1);

            \draw[thin, cyan, ->](0,0) to (-0.2,0.2);
            \draw[thin, cyan] (-0.2, 0.2) to (-0.4, 0.4);
            \node at (-0.6, 0.6) {\tiny ${\color{cyan} v}$};

            \draw[thin, orange, ->] (1,0) to (0.75,0.25);
            \draw[thin, orange] (0.75, 0.25) to (0.5,0.5);
            \node at (0.7,0) {\tiny ${\color{orange} \phi}$};

            \draw[thin, purple, ->] (0.5, 0.5) to (0.25, 0.25);
            \draw[thin, purple] (0.25, 0.25) to (0,0);
            \node at (0.2, 0.5) {\tiny ${\color{purple} \theta}$};

            \node at (0, 2.5) {$(v, \gamma, t)$};
            \node at (-0.3, -0.3) {\tiny ${\color{red}  \gamma(0)}$};
            \node at (1.4, -0.3) {\tiny ${\gamma(t)}$};
            \node at (1.2, 0.8) {\tiny $\phi_1 \circ \gamma(t)$};
            \node at (2.1, 1.5) {\tiny $\gamma$};

            \node at (-1.4, 1.5) {\tiny ${\color{teal}B_\eps(v)}$};

            \node[v] at (0,0) {};
            \node[vsmall] at (1,0) {};
            \node[vsmall] at (-0.4, 0.4) {};
            \node[vsmall] at (0.5,0.5) {};
            \node[vred] at (0,0) {};
        \end{tikzpicture}
    \end{minipage}
    $\mapsto \left(
    \begin{minipage}{.37 \textwidth}
        \begin{tikzpicture}[scale=.8,
        v/.style={draw,shape=circle, fill=black, minimum size=1.3mm, inner sep=0pt, outer sep=0pt},
        vred/.style={draw,shape=circle, fill=red, minimum size=1mm, inner sep=0pt, outer sep=0pt},
        vsmall/.style={draw,shape=circle, fill=black, minimum size=1mm, inner sep=0pt, outer sep=0pt}]
            \node at (-2, 0) {$\lambda \cdot ($};
            \node at (-0.5, 0) {$)$};
        
            \draw[color=purple, -<] (-1.15,-0.25) to (-0.9,0);
            \draw[color=purple] (-0.9,0) to (-0.65,0.25);
            \draw[color=cyan] (-1.15, -0.25) to (-1.55, 0.15);
            \draw[color=cyan, ->] (-1.15, -0.25) to (-1.35, -0.05);

            \begin{scope}[shift={(-0.6,0)}]
            \node at (0.3, -0.2) {,};
        
            \draw[color=black, ->] (0.7,0) to[out=90, in=90] (2.7,2);
            \draw[color=black] (2.7,2) to[out=270,in=90] (1.7, 0);

            \draw[thin, orange, ->] (1.7,0) to (1.45,0.25);
            \draw[thin, orange] (1.45, 0.25) to (1.2,0.5);

            \draw[thin, purple, ->] (1.2, 0.5) to (0.95, 0.25);
            \draw[thin, purple] (0.95, 0.25) to (0.7,0);

            \node[v] at (0.7,0) {};
            \node[vsmall] at (1.7,0) {};
            \node[vsmall] at (1.2,0.5) {};
            \node[vred] at (0.7,0) {};

            \node at (2.2, -0.2) {,};

            \draw[color=black, -<] (2.7,0) to[out=-90, in=-90] (4.7,-2);
            \draw[color=black] (4.7,-2) to[out=-270,in=-90] (3.7, 0);

            \draw[thin, orange, -<] (3.7,0) to (3.45,0.25);
            \draw[thin, orange] (3.45, 0.25) to (3.2,0.5);

            \draw[thin, purple, -<] (3.2, 0.5) to (2.95, 0.25);
            \draw[thin, purple] (2.95, 0.25) to (2.7,0);

            \node[v] at (2.7,0) {};
            \node[vsmall] at (3.7,0) {};
            \node[vsmall] at (3.2,0.5) {};
            \node[vred] at (2.7,0) {};

            \node at (3.2,1.5) {\tiny $\gamma|_{[0,t]}$};
            \node at (2.8, -1.5) {\tiny $\gamma|_{[t,1]}$};
            \end{scope}
        \end{tikzpicture}
    \end{minipage}
    \right)$
    \caption{Coproduct: the figure on the left shows a triple $(v,\gamma, t)$ in the domain of the coproduct. The figure on the right shows the output.
    }
  \label{fig:coprod}
\end{figure}
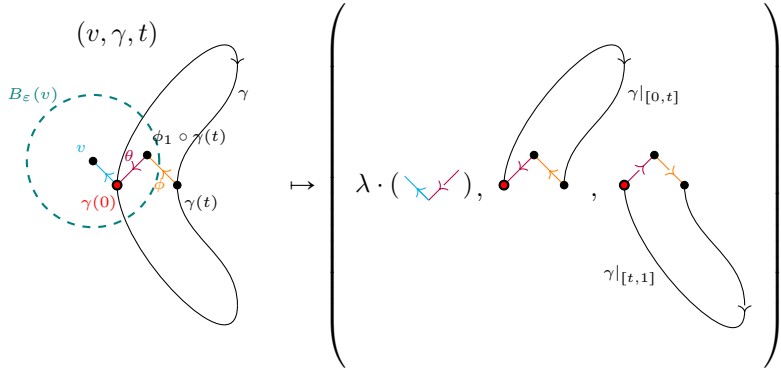

\begin{lemma}\label{lem: Delta well def}
    $\Delta_{unst}$ is a well-defined continuous map.
\end{lemma}

\begin{proof}
    We must check that (\ref{eq: Delta unst}) sends $(v, \gamma, t)$ to the basepoint whenever $t \in \{0,1\}$, $|v|=1$ or $\gamma(0) \in \partial M$. Once this is verified, it is clear that (\ref{eq: Delta unst}) defines a continuous map.\par 
    If $t =0$ and the incidence condition for $\Delta_{unst}$ holds (i.e. $\lVert v-\phi_1 \circ \gamma(t)\rVert \leq \varepsilon$), then the first loop in (\ref{eq: Delta unst}):
    $$B\left(\gamma(0) \overset{\gamma|_{[0,t]}}\rightsquigarrow \gamma(0) \overset{\phi}\rightsquigarrow
    \phi_1 \circ \gamma(0) \overset{\theta}\rightsquigarrow
    \gamma(0)\right)$$
    is a constant loop since the path inside the brackets has length $\leq \zeta$, by (\ref{def: cop dat}.\ref{item: V small}) and (\ref{def: cop dat}.\ref{item: eps small 3}).\par 
    Similarly if $t=1$ and the incidence condition holds, the second loop in (\ref{eq: Delta unst}) is constant for the same reason.\par 
    If $|v|=1$, the first entry in (\ref{eq: Delta unst}) lies outside of $[-1,1]^L$, by (\ref{def: cop dat}.\ref{item: lambda big}) (and Remark \ref{rmk: radius}), so (\ref{eq: Delta unst}) represents the basepoint.
    
    If $\gamma(0) \in \partial M$, then by (\ref{def: cop dat}.\ref{item: eps small 2}), the incidence condition can never hold (noting that $\lVert v-\gamma(0)\rVert \leq \eps$ and using the triangle inequality).
\end{proof}

\begin{definition}
    The \emph{(stable) string coproduct} is the map of spectra
    \begin{equation}
        \Delta=\Delta^Q: \frac{\cL M^{-TM}}{\partial \cL M^{-TM}}\wedge S^1 \to \Sigma^\infty \frac{\cL M}M \wedge \frac{\cL M}M
    \end{equation}
    obtained from the unstable coproduct by applying Lemma \ref{lem: spectra unst} to $\Delta_{unst}$.
\end{definition}
\begin{lemma}\label{lem: cop ind}
    The coproduct $$\Delta: \frac{\cL M^{-TM}}{\partial \cL M^{-TM}}\wedge S^1 \to \Sigma^\infty \frac{\cL M}M \wedge \frac{\cL M}M$$ is independent of choices.
\end{lemma}
\begin{proof}
    
Let $Q = (\ldots, \lambda)$ be a fixed choice of embedding data, such that $\lambda$ is sufficiently large. Note that  $\Delta^Q$ can be alternatively described on the $i^{th}$ space: 
    \begin{equation*}
    \left(\frac{\cL M^{-TM}}{\partial \cL M^{-TM}} \wedge S^1\right)_i := \frac{\cL M^{D(\bR^{i-L} \oplus \nu_e)}}{\partial\cL M^{D(\bR^{i-L} \oplus \nu_e)}} \wedge S^1
\end{equation*} 
by using in \cref{eq: Delta unst} the stabilized {embedding data},  ${st^{L,i}(Q)} = (\ldots, \tfrac{\sqrt{L+1}}{\sqrt{L}}\cdot \lambda)$, as defined in \cref{eq: st maps}
and noting that for $e'$ (the embedding associated to $st^{L,i}(Q)$), there is a natural identification 
$$\frac{\cL M^{D(\bR^{i-L} \oplus \nu_e)}}{\partial\cL M^{D(\bR^{i-L} \oplus \nu_e)}} \wedge S^1  = \frac{\cL M^{D\nu_{e'}}}{\partial \cL M^{D\nu_{e'}}} \wedge S^1.$$

  Indeed, this follows by noting that the structure maps:
\begin{equation}\label{eq: str maps}
    \Sigma\left( \frac{\cL M^{D(\bR^{i-L} \oplus \nu_e)}}{\partial\cL M^{D(\bR^{i-L} \oplus \nu_e)}}\wedge S^1\right) \to \frac{\cL M^{D(\bR^{1+i-L} \oplus \nu_e)}}{\partial\cL M^{D(\bR^{i+1-L} \oplus \nu_e)}} \wedge S^1
\end{equation}
send the $[-1,1]$ variable, corresponding to the first suspension factor on the left hand side, to the first variable in the $\bR^{1+i-L}$ on the right hand side, and by the identity in all other factors. Hence the diagram:
\begin{center}
    \begin{tikzcd}
            \Sigma \left( \frac{\cL M^{D(\bR^{i-L} \oplus \nu_e)}}{\partial \cL M^{D(\bR^{i-L} \oplus \nu_e)}} \wedge S^1\right) \arrow[rr, "\Sigma \Delta_{unst}"] \arrow[d] &&
            \Sigma \Sigma^L \frac{\cL M}M \wedge \frac{\cL M}M \arrow[d] \\
            \frac{\cL M^{D(\bR^{1+i-L} \oplus \nu_e)}}{\partial \cL M^{D(\bR^{1+i-L} \oplus \nu_e)}} \wedge S^1 \arrow[rr, "\Delta_{unst}"] &&
            \Sigma^{L+1} \frac{\cL M}M \wedge \frac{\cL M}M
        \end{tikzcd}
        \end{center} homotopy commutes, via a homotopy that linearly interpolates between the two $\lambda$-values. Here the vertical maps are the structure maps, and the bottom horizontal map is $\Delta_{unst}$ as in  \cref{eq: Delta unst} using the stabilized {embedding data}. Hence $\Delta$ can be defined on the $i^{th}$ space using the stabilised {embedding data}. 

        Now, for sufficiently large $L$ the space of choices $ED^L(M)$ is connected. Given {embedding data} $Q, Q' \in ED^L(M)$, there is a unique up to homotopy path from $Q$ to $ Q'$, giving a (canonical up to homotopy) equivalence of spectra associated to the embeddings $e$ and $e'$, as well as a homotopy between $\Delta^Q$ and $\Delta^{Q'}$.  The conclusion follows.  
    \end{proof}
    The following lemma is clear from the definitions.
    \begin{lemma}\label{lem: diff inv}
        Let $f: M \to N$ be a diffeomorphism between compact manifolds with corners. Then there is a commutative diagram
        \begin{equation}
            \begin{tikzcd}
                \frac{\cL M^{-TM}}{\partial \cL M^{-TM}} \wedge S^1
                \arrow[r, "\Delta"] 
                \arrow[d, "f"]
                &
                \Sigma^\infty \frac{\cL M}M \wedge \frac{\cL M}M
                \arrow[d, "f \wedge f"]
                \\
                \frac{\cL N^{-TN}}{\partial \cL N^{-TN}} \wedge S^1
                \arrow[r, "\Delta"]
                &
                \Sigma^\infty \frac{\cL N}N \wedge \frac{\cL N}N
            \end{tikzcd}
        \end{equation}
    \end{lemma}
    
\subsection{Smoothing corners}\label{sec: smoth}
    In this subsection, we prove Proposition \ref{prop: smoth corn}, which says that smoothing the corners of a manifold with corners $M$ does not change its coproduct. Though this follows from (and is not used in the proof of) Theorem \ref{main theorem}, we record a more direct proof here.

    \begin{definition}\label{def: smoth}
        Let $M$ be a compact smooth manifold with corners. A compact manifold with boundary $N$ is a \emph{smoothing} of $M$ if $N \subseteq M$ is a codimension 0 submanifold, such that there is an isotopy $\{\omega_t\}_{t \in [0,1]}: \partial M \hookrightarrow M$ of locally flat topological embeddings, such that $\omega_0$ is the inclusion $\partial M \hookrightarrow M$ and $\omega_1(\partial M) = \partial N)$.
    \end{definition}
    Note that if this holds, then $M$ and $N$ are homeomorphic, by the topological isotopy extension theorem.
    \begin{proposition}\label{prop: smoth corn}
        
        Let $M$ be a compact manifold with corners and $N$ a smoothing of $M$. Then there is a commutative diagram: 
        \begin{equation}
            \begin{tikzcd}
                \frac{\cL M^{-TM}}{\partial \cL M^{-TM}} \wedge S^1
                \arrow[r, "\Delta"] 
                \arrow[d, "\simeq"]
                &
                \Sigma^\infty \frac{\cL M}M \wedge \frac{\cL M}M
                \arrow[d, "f \wedge f"]
                \\
                \frac{\cL N^{-TN}}{\partial \cL N^{-TN}} \wedge S^1
                \arrow[r, "\Delta"]
                &
                \Sigma^\infty \frac{\cL N}N \wedge \frac{\cL N}N
            \end{tikzcd}
        \end{equation}
        In particular, if $M$ and $M'$ are compact manifolds with corners which have diffeomorphic smoothings, then their coproducts agree.
    \end{proposition}
    \begin{proof}
        Let $\omega_t$ be the isotopy from Definition \ref{def: smoth}. For each $t \in [0,1]$, let $M_t$ be the component of $M \setminus \omega_t(\partial M)$ not touching $\partial M$. \par 

        Then letting $\bM = \cup_{t \in [0,1]} M_t \times \{t\}$, $\bM$ admits the structure of a (topological) fibre bundle over $[0,1]$, with fibres over $0$ and $1$ given by $M$ and $N$ respectively. \par 

        Each $M_t$ is not necessarily a smooth manifold with smooth boundary, but does admit a smooth atlas locally modelled on subsets of $\bR^n$ which are the closure of the a component of the complement of a separating (locally flat) topological hypersurface in $\bR^n$. Note that the condition for a loop in such a manifold to be of Sobolev class $H^1$ still makes sense. \par 

        The definition for the loop coproduct goes through word-for-word, though the smooth structure inherited by $M^{ext}$ is of this weaker form rather than a manifold with corners. \par 

        Now applying same argument as the end of the proof of Lemma \ref{lem: cop ind} to $\bM$ implies the result.   
    \end{proof}

\section{Spectral Chas-Sullivan modules}\label{sec: prod}
    Let $M$ be a compact $n$-manifold, possibly with corners. The purpose of this section is to construct a generalization of the Chas-Sullivan product to maps of spectra:
    \begin{equation}
        \mu_{{{l}}}: \frac{\cL M^{-TM}}{\partial \cL M^{-TM}} \wedge \Sigma^\infty_+ \cL M \to\Sigma^\infty_+ \cL M,
    \end{equation}
    and
    \begin{equation}
        \mu_{{{r}}}: \Sigma^\infty_+ \cL M \wedge \frac{\cL M^{-TM}}{\partial \cL M^{-TM}} \to \Sigma^\infty_+ \cL M.
    \end{equation}
    These maps, constructed in the spirit  of Cohen and Jones  \cite{cohen2002homotopy}, are adapted to the case that $M$ has boundary and are best suited for our purposes. 
    \begin{remark}
        In general, $\frac{\cL M^{-TM}}{\partial \cL M^{-TM}}$ is a unital ring spectrum, whose multiplication
        $$\frac{\cL M^{-TM}}{\partial \cL M^{-TM}} \wedge \frac{\cL M^{-TM}}{\partial \cL M^{-TM}} \to \frac{\cL M^{-TM}}{\partial \cL M^{-TM}}$$
        realises the Chas-Sullivan product on homology in the case $M$ is closed, see \cite{cohen2002homotopy}. Although we do not prove this here, $\mu_{{{r}}}$ and $\mu_{{{l}}}$ together equip $\Sigma^\infty_+ \cL M$ with the structure of a bimodule over this ring spectrum. In Section \ref{sec: hom com pro} we prove that our model for these module maps does recover the definition of the Chas-Sullivan product given in \cite{hingston2017product} after passing to homology, up to a sign.
    \end{remark}

    \begin{definition}
         Let $Q$ be a choice of {embedding data} for $M$. 
         The  \emph{unstable left product} is defined to be the map of spaces:

    \begin{equation}
       \mu_{{{l}},unst} = \mu^Q_{r, unst}: \frac{\cL M^{D\nu_e}}{\partial \cL M^{D\nu_e}} \wedge \cL M_+
        \to \Sigma^L_+ \cL M
    \end{equation}
   
    sending $((v, \gamma), \delta)$ to
    \begin{equation}
        \begin{cases}
            \begin{pmatrix}
                \lambda (v-\phi_1 \circ \delta(0)), \\
                \gamma(0) \overset{\gamma} \rightsquigarrow 
                \gamma(0) \overset \theta \rightsquigarrow 
                \phi_1 \circ \delta(0) \overset {\overline \phi} \rightsquigarrow 
                \delta(0) \overset \delta \rightsquigarrow 
                \delta(0) \overset \phi \rightsquigarrow 
                \phi_1 \circ \delta(0) \overset \theta \rightsquigarrow 
                \gamma(0)
            \end{pmatrix}
            & \textrm{ if } \lVert v - \phi_1 \circ \delta(0) \rVert \leq \eps \\
            * & \textrm{ otherwise.}
        \end{cases}
    \end{equation}
    The unstable \emph{right product} is defined to be the map of spaces:
    \begin{equation}
       \mu_{{{r}}, unst}=\mu_{{{r}}, unst}^Q: \cL M_+ \wedge \frac{\cL M^{D\nu_e}}{\partial \cL M^{D\nu_e}}
        \to \Sigma^L_+ \cL M
    \end{equation}
    sending $(\delta, (v, \gamma))$ to
    \begin{equation}
        \begin{cases}
            \begin{pmatrix}
                \lambda (v-\phi_1 \circ \delta(0)), \\
                \gamma(0) \overset \theta \rightsquigarrow 
                \phi_1 \circ \delta(0) \overset {\overline \phi} \rightsquigarrow 
                \delta(0) \overset \delta \rightsquigarrow 
                \delta(0) \overset \phi \rightsquigarrow
                \phi_1 \circ \delta(0) \overset \theta \rightsquigarrow
                \gamma(0) \overset \gamma \rightsquigarrow \gamma(0)
            \end{pmatrix}
            & \textrm{ if } \lVert v - \phi_1 \circ \delta(0) \rVert \leq \eps \\
            * & \textrm{ otherwise.}
        \end{cases}
    \end{equation}
The \emph{stable right module product} 
 $$\mu_{{{r}}}: \Sigma^\infty_+ \cL M \wedge \frac{\cL M^{-TM}}{\partial \cL M^{-TM}} \to \Sigma^\infty_+ \cL M$$
and the \emph{stable left product} 
   $$\mu_{{{l}}}: \frac{\cL M^{-TM}}{\partial \cL M^{-TM}} \wedge \Sigma^\infty_+ \cL M \to\Sigma^\infty_+ \cL M,$$ 
   are obtained from the unstable counterparts via Lemma  \ref{lem: spectra unst}.

    \end{definition}

Arguing exactly as in Lemmas \ref{lem: Delta well def} and \ref{lem: cop ind} we see that these are well-defined maps of spectra, independent of choices up to homotopy. 
    \begin{definition}

       Let
        \begin{equation} \label{eq: split}
            \Sigma^\infty_+ \cL M \simeq \Sigma^\infty_+ M \vee \Sigma^\infty \frac{\cL M}M
        \end{equation}
        be the cannonical splitting induced by the inclusion of constant loops. Then $\tilde \mu_{{{l}}, unst}$ is the  composition:
        \begin{equation}\label{eq: split prod}
            \tilde \mu_{{{l}}, unst}: \frac{\cL M^{-TM}}{\partial \cL M^{-TM}} \wedge \Sigma^\infty \frac{\cL M}M 
            \to 
            \frac{\cL M^{-TM}}{\partial \cL M^{-TM}} \wedge \Sigma^\infty_+ \cL M
            \xrightarrow{\mu_{{{l}}, unst}}
            \Sigma^\infty_+ \cL M 
            \to 
            \Sigma^\infty \frac{\cL M}M
        \end{equation}
        where the first and second arrows are the canonical inclusion and projection respectively, induced by (\ref{eq: split}).
        
    \end{definition}
    
\section{Stability}\label{sec: stab}

Let $M$ be a compact manifold, possibly with corners, and let $e \in \operatorname{Emb}(M, \bR^L)$. In this section we prove that the string topology operations from Sections \ref{sec: coproduct section} and \ref{sec: prod} are invariant under replacing $M$ with the total space of the disc bundle $D\nu$ of the normal bundle $\nu$ of $e$.

Let $\pi: \nu \to M$ be the projection, and $\iota: M \hookrightarrow \nu$ the inclusion of the zero section. In the following lemma we first identify the domains of the coproducts for $M$ and $D\nu$:
\begin{lemma}\label{lem: stab dom}
    There is a homotopy equivalence of spectra
    \begin{equation}
        \alpha: \frac{\cL M^{-TM}}{\partial \cL M^{-TM}} \to \frac{\cL D\nu^{-TD\nu}}{\partial \cL D\nu^{-TD\nu}}
    \end{equation}
\end{lemma}
\begin{proof}
    Choose {embedding data} $Q$ for $M$ extending $e$. We define a homotopy equivalence of spaces $\alpha: \frac{\cL M^{D\nu}}{\partial \cL M^{D\nu}} \to \frac{\cL D\nu}{\partial \cL D\nu}$ which induces a homotopy equivalence of spectra as desired, via Lemma \ref{lem: spectra unst}.\par 
    For $(v, \gamma) \in \frac{\cL M^{D\nu}}{\partial \cL M^{D\nu}}$, we define 
    \begin{equation} \label{eq: alpha}
        \alpha(v,\gamma) := (\gamma_v) \in \frac{\cL D\nu}{\partial \cL D\nu}
    \end{equation}
    where $\gamma_v$ is the loop
    \begin{equation} \label{eq: alpha 2} 
        v \overset{\theta
        } \rightsquigarrow
        \gamma(0) \overset\gamma \rightsquigarrow 
        \gamma(0) \overset {\theta
        } \rightsquigarrow 
        v
    \end{equation}
    
    A homotopy inverse to $\alpha$ is given by sending $\gamma$ to $(\gamma(0), \pi \circ \gamma)$. Also note that since the space of {embedding data} extending $e$ is connected, $\alpha$ is well-defined up to homotopy.
\end{proof}
\begin{remark}\label{rwgepinbspribndrip}
    By construction the map $\alpha$ in Lemma \ref{lem: stab dom} is compatible with fundamental classes (see Definition \ref{def: fund class}), in the sense that the following diagram commutes up to homotopy:
    \begin{equation}\label{eq: stab fund class}
        \begin{tikzcd}
            \bS \arrow[r, "{[}M{]}"] 
            \arrow[d, "="] &
            \frac{\cL M^{-TM}}{\partial \cL M^{-TM}} \arrow[r, "i^M"] 
            \arrow[d, "\simeq"] &
            \frac{\cL M^{-TM}}{\partial \cL M^{-TM}} \arrow[d, "\alpha"] \\
            \bS \arrow[r, "{[}D\nu{]}"] &
            \frac{D\nu^{-TD\nu}}{\partial D\nu^{-TD\nu}} \arrow[r, "i^{D\nu}"] &
            \frac{\cL D\nu^{-TD\nu}}{\partial \cL D\nu^{-TD\nu}} 
        \end{tikzcd}
    \end{equation}
    where the $i^M$ and $i^{D\nu}$ are induced by the inclusions of constant loops for $M$ and $D\nu$ respectively.
\end{remark}

\begin{remark}
    In the definition of the coproduct, we do not have to quotient by $\partial \cL M^{-TM}$; one would still arrive at a reasonable operation. However if we do not do this, then Lemma \ref{lem: stab dom} can't hold: the domains of the two coproducts wouldn't be homotopy equivalent.\par 
    For example, if $\nu$ is a trivial vector bundle of rank $r$ and $M$ has no boundary, the spectra $\cL D\nu^{-TD\nu}$ and $\frac{\cL D\nu^{-TD\nu}}{\partial \cL D\nu^{-TD\nu}}$ differ by a shift of degree $r$.
\end{remark}

\subsection{Coproduct}\label{sec: stability}

\begin{theoremaaa}\label{thm: stability}
    There is a homotopy commutative diagram of spectra:
    $$\begin{tikzcd}
        \frac{\cL M^{-TM}}{\partial \cL M^{-TM}} \wedge S^1 \arrow[r, "\Delta"]\arrow[d, "\alpha \wedge Id_{S^1}"] & 
        \Sigma^\infty \frac{\cL M}M \wedge \frac{\cL M}{M}\\
        \frac{\cL D\nu^{-TD\nu}}{\partial \cL D\nu^{-TD\nu}} \wedge S^1 \arrow[r, "\Delta"] & 
        \Sigma^\infty \frac{\cL D\nu}{D\nu} \wedge \frac{\cL D\nu}{D\nu}\arrow[u, "\pi \wedge \pi"]
    \end{tikzcd}$$
    where $\alpha$ and $\pi \wedge \pi$ are homotopy equivalences.
\end{theoremaaa}

\begin{proof}
    Choose $Q = (e, \rho^{ext}, \zeta, V, \eps, \lambda) \in ED^L(M)$ {embedding data} extending $e$. We define $$Q' = (e', \rho'^{ext}, \zeta', V', \eps', \lambda') \in ED^L(D\nu)$$ as follows. Let $e'=\rho$. Note that since this is a codimension 0 embedding, its normal bundle is trivial. We fix a diffeomorphism $(D\nu)^{ext} \cong D_2 \nu_e$, such that the natural map $r': (D\nu)^{ext} \to D\nu$ is given by projection to the sphere bundle on $D_2\nu \setminus D\nu$, and on $D\nu_e|_{M^{ext} \setminus M}$ is a horizontal lift of the map $M^{ext} \to M$. In particular, this implies $r \circ r' = r$. Let $\rho'^{ext} = \rho^{ext}$. \par 
    We set $\zeta' = \zeta$ and assume we have chosen $\zeta > 0$ small enough that (\ref{def: cop dat}.\ref{item: zeta small cop}) holds for $D\nu$.\par 
    Using the induced metrics on $M$ and $\nu_e|_M$, we let $\tilde{V}$ be the horizontal lift of $V$ to $D\nu$. Let $W$ be the tautological vector field on $D\nu$ (i.e. its value at a point $v$ is $v$). Now choose $\mu > 0$ and let $V' = V-\mu W$. This is an inwards-pointing vector field on $D\nu$, and for $\mu > 0$ small enough, (\ref{def: cop dat}.\ref{item: V small}) holds.\par 
    Let $\eps' = \eps$ and $\lambda' = \lambda$, and we may choose them so that $\eps$ is small enough and $\lambda$ is large enough that (\ref{def: cop dat}.\ref{item: eps small 1}, \ref{item: eps small 2}, \ref{item: eps small 3}, \ref{item: lambda big}) all hold.\par 
    We show that the following diagram commutes up to homotopy, with vertical arrows homotopy equivalences, which will imply the desired result, by Lemmas \ref{lem: spectra hom} and \ref{lem: spectra unst}.
    \begin{equation} \label{eq: instability homotopy}\begin{tikzcd}
        \frac{\cL M^{D\nu}}{\partial \cL M^{D\nu}} \wedge S^1 \arrow[r, "\Delta_{unst}^Q"]\arrow[d, "\alpha \wedge Id_{S^1}"] & 
        \Sigma^L \frac{\cL M}M \wedge \frac{\cL M}{M}\\
        \frac{\cL D\nu}{\partial \cL D\nu} \wedge S^1 \arrow[r, "\Delta_{unst}^{Q'}"] & 
        \Sigma^L \frac{\cL D\nu}{D\nu} \wedge \frac{\cL D\nu}{D\nu}\arrow[u, "\pi \wedge \pi"]
    \end{tikzcd}\end{equation}
    Now consider the incidence conditions for $\Delta^Q_{unst}$ and $\Delta^{Q'}_{unst} \circ (\alpha \wedge Id_{S^1})$ respectively, for $(v, \gamma, t) \in \frac{\cL M^{D\nu}}{\partial \cL M^{D\nu}} \wedge S^1$. These are the conditions $\lVert v-\phi_1 \circ \gamma(t) \rVert \leq \eps$, and $\lVert v-\phi'_1 \circ \gamma_v(t)\rVert \leq \eps$ respectively.\par 
    If the incidence conditions hold, the two ways around the diagram both have the same final two components.\par 
    We find a homotopy between these two ways around the diagram by linearly interpolating between $V$ and $V'$. Explicitly, this is the homotopy 
    $$H: [0,1]_u \times \frac{ \cL M^{D\nu}}{\partial \cL M^{D\nu}} \wedge S^1 \to \Sigma^L \frac{\cL M}M \wedge \frac{\cL M}M$$
    defined so that $H_u$ sends $(v, \gamma, t)$ to 
    \begin{equation}\label{eq: stability homotopy}
    \begin{cases}
        \begin{pmatrix}
            \lambda \left(v-\phi_1^u \circ \gamma_{uv}(t)\right),\\
            B\left(
            \gamma(0) \overset{\gamma|_{[0,t]}}\rightsquigarrow
            \gamma(t) \overset{\phi}\rightsquigarrow
            \phi_1 \circ \gamma(t) \overset{\theta}\rightsquigarrow
            \gamma(0)
             \right),\\
            B\left(
            \gamma(0) \overset{\theta}\rightsquigarrow
            \phi_1 \circ \gamma(t) \overset{\phi}\rightsquigarrow 
            \gamma(t) \overset{\gamma|_{[t,1]}}\rightsquigarrow
            \gamma(0)
             \right)\\
        \end{pmatrix}
        & \textrm{ if } \lVert v - \phi_1^u \circ \gamma_{uv}(t)\rVert \leq \varepsilon\\
        * & \textrm{ otherwise.}
    \end{cases}\end{equation}
    where $\phi_1^u$ is the time-one flow of the vector field $V-\mu u W$ (so in particular $\phi^1_1 = \phi'_1$). Note the only difference from (\ref{eq: Delta unst}) is that $\phi$ is replaced by $\phi^u$ (which agrees with $\phi$ on the zero section $M$).\par 
Arguing as in \cref{lem: Delta well def}, we see that  (\ref{eq: stability homotopy}) is well-defined. \par 
We assume $\eps > 0$ is small enough that $d(S\nu, D_{\frac 12} \nu) > \eps$ and $d(\phi_1^{\frac 12}(D_1\nu), S\nu) > \eps$. Then if $|v|=1$, the incidence condition can't hold: for $u \leq \frac 12$ this is because $\phi_1^u\circ \gamma_{uv} \subseteq D_{\frac 12}\nu$ so by the first condition the incidence condition can't hold, and for $u \geq \frac 12$ by the second condition the incidence condition can't hold.\par 

Inspection of (\ref{eq: stability homotopy}) and (\ref{eq: Delta unst}) shows that $H_0$ and $\Delta^Q_{unst}$ agree, and also that $H_1$ and $(\pi \wedge \pi_ \circ \Delta^{Q'}_{unst} \circ (\alpha \wedge Id_{S^1})$ agree. 

It is clear that $\pi \wedge \pi$ is a homotopy equivalence. 
\end{proof}

\begin{corollary}
    Let $M$ and $M'$ be closed manifolds which are simple homotopy equivalent. Then their string coproducts agree.\par 
    More precisely, there is a homotopy commutative diagram of spectra, with vertical arrows homotopy equivalences: 
    $$\begin{tikzcd}
        \frac{\cL M^{-TM}}{\partial \cL M^{-TM}} \wedge S^1 \arrow[r, "\Delta"] \arrow[d, "\simeq"] &
        \Sigma^\infty \frac{\cL M}{M} \wedge \frac{\cL M}{M} \arrow[d, "\simeq"]\\
        \frac{\cL M'^{-TM'}}{\partial \cL M'^{-TM'}} \wedge S^1 \arrow[r, "\Delta"] &
        \Sigma^\infty \frac{\cL M'}{M'} \wedge \frac{\cL M'}{M'}
    \end{tikzcd}$$
\end{corollary}
This in particular implies homeomorphism invariance of the string coproduct, though this could have been proved in a different way (for example, by giving a more general definition that did not make use of the smooth structure on $M$). 
\begin{proof}
    By \cite[Page 7]{Mazur}, for $L \gg 0$, there are embeddings $M,M' \hookrightarrow \R^L$ with diffeomorphic tubular neighbourhoods; the result then follows from Theorem \ref{thm: stability}.
\end{proof}
Alternatively, this corollary follows from Theorem \ref{main theorem}, which includes the case when $M$ and $M'$ have boundary, and further without assuming $M$ and $M'$ even have the same dimension.
\subsection{Thom isomorphism}\label{sec: thom}
    In this section we prove that the homology-level coproduct $H_*(\Delta)$ is compatible with the Thom isomorphism, in the appropriate sense.\par 

    Let $M$ be a compact oriented manifold with corners of dimension $n$, and $\pi_E: E \to M$ an oriented vector bundle of rank $n$ with disc bundle $DE$. We orient the total space $DE$ such that the canonical isomorphism $E \oplus TM \cong TDE$ is orientation-preserving.\par 

    Let $\tau_E \in H^r(DE, SE)$ be the Thom class. There is then a Thom isomorphism
    \begin{equation}
        \operatorname{Thom}: H_*\left(\frac{\cL DE}{\partial \cL DE}\right) \to H_*\left(\frac{\cL M}{\partial \cL M}\right)
    \end{equation}
    given by the composition $\pi_E \circ ((ev_0^*\tau_E) \cap \cdot)$, where $ev_0: \cL DE \to DE$ evaluates at 0.

    \begin{proposition}\label{prop: thom cop}
        There is a commutative diagram:
        \begin{equation}\label{eq: th}
            \begin{tikzcd}
                H_*\left(\frac{\cL M}{\partial \cL M}\right) 
                &
                H_*\left(\frac{\cL M^{-TM}}{\partial \cL M^{-TM}}\right) 
                \arrow[l, "\operatorname{Thom}"]
                \arrow[r, "H_*(\Delta)"]
                &
                H_*\left(\frac{\cL M}M \wedge \frac{\cL M}M \right)
                \\
                H_*\left(\frac{\cL DE}{\partial \cL DE}\right)
                \arrow[u, "\operatorname{Thom}"]
                &
                H_*\left(\frac{\cL DE^{-TDE}}{\partial \cL DE^{-TDE}}\right)
                \arrow[r, "H_*(\Delta)"]
                \arrow[l, "\operatorname{Thom}"]
                &
                H_*\left(\frac{\cL DE}{DE} \wedge \frac{\cL DE}{DE}\right)
                \arrow[u, "\pi_E \wedge \pi_E"]
            \end{tikzcd}
        \end{equation}
    \end{proposition}
    \begin{proof}
        Choose an embedding $e_E: DE \hookrightarrow \bR^L$ for $L \gg 0$, and let $\nu_E$ be its normal bundle. Let $e_M$ be the composition $M \xhookrightarrow{0\operatorname{-section}} DE \xhookrightarrow{e_E} \bR^L$, and let $\nu_M$ be its normal bundle. \par 
        There is a canonical isomorphism 
        \begin{equation} \label{ eq: M E norm}
            \nu_E \oplus E \cong \nu_M
        \end{equation}
        Orienting the two normal bundles as in Definition \ref{def: thom iso}, we find that (\ref{ eq: M E norm}) is orientation-preserving. Let $\tau_{\nu_M}$ and $\tau_{\nu_E}$ be the corresponding Thom classes; then 
        \begin{equation} \label{eq: M E or}
            \tau_E \cup \tau_{\nu_E} = \tau_{\nu_M}
        \end{equation}
        
        We extend these embeddings to embedding data on $M$ and $DE$, and note that from (\ref{ eq: M E norm}), $D\nu_M$ and $D\nu_E$ smooth to manifolds with boundary which are diffeomorphic to each other.

        Consider the following diagram:
        \begin{equation}\label{eq: big th}
            \begin{tikzcd}
                &
                H_{*-n}\left(\frac{\cL M^{-TM}}{\partial \cL M^{-TM}} \wedge S^1\right)
                \arrow{ddl}[swap]{\operatorname{Thom}}
                \arrow[d, "="]
                \arrow[rr, "\Delta"]
                &&
                H_{*-n}\left(\frac{\cL M}M \wedge \frac{\cL M}M\right)
                \arrow[dd, "\operatorname{Susp}"]
                \\
                &
                \tilde H_{*+L-n}\left(\frac{\cL M^{D\nu_M}}{\partial \cL M^{D\nu_M}}\wedge S^1\right)
                \arrow[d, "\alpha"]
                \arrow[dl, "( ev_0^*\tau_{\nu_M} \cap \cdot)"]
                &&
                \\
                H_*\left(\frac{\cL M}{\partial \cL M} \wedge S^1\right)
                &
                \tilde H_{*+L-n}\left(\frac{\cL D\nu_M}{\partial \cL D\nu_M}\wedge S^1\right)
                \arrow[d, "\cong"]
                \arrow[rr, "(\pi_{\nu_M} \wedge \pi_{\nu_M}) \circ \Delta"]
                &&
                \tilde H_{*+L-n} \left(\Sigma^L \frac{\cL M}M \wedge \frac{\cL M}M\right) 
                \arrow[d, "="]
                \\
                H_{*+r}\left(\frac{\cL DE}{\partial \cD DE}\wedge S^1\right)
                \arrow[u, "\operatorname{Thom}"]
                &
                \tilde H_{*+L-n}\left(\frac{\cL D\nu_E}{\partial \cL D\nu_E} \wedge S^1\right)
                \arrow[rr, "(\pi_{\nu_E} \wedge \pi_{\nu_E}) \circ \Delta"]
                &&
                \tilde H_{*+L-n}\left(\Sigma^L \frac{\cL M}M \wedge \frac{\cL M}M\right)
                \\
                &
                \tilde H_{*+L-n}\left(\frac{\cL DE^{D\nu_E}}{\partial \cL DE^{D\nu_E}}\wedge S^1 \right)
                \arrow[u, "\alpha"]
                \arrow{ul}[swap]{(ev_0^* \tau_{\nu_E} \cap \cdot)}
                &&
                \\
                &
                H_{*-n}\left(\frac{\cL DE^{-TDE}}{\partial \cL DE^{-TDE}}\wedge S^1\right)
                \arrow{uul}{\operatorname{Thom}}
                \arrow[u, "="]
                \arrow[rr, "\pi_E \circ \Delta"]
                &&
                H_{*-n}\left(\frac{\cL M}M \wedge \frac{\cL M}M\right)
                \arrow[uu, "\operatorname{Susp}"]
            \end{tikzcd}  
        \end{equation}
        Here $\alpha$ is as in Lemma \ref{lem: stab dom} and the middle vertical isomorphism is from Proposition \ref{prop: smoth corn}. Note that all vertical arrows are isomorphisms, and that the composition of the vertical isomorphisms in the right column gives the identity.\par 
        Note that the outside of this diagram is exactly (\ref{eq: th}), so it suffices to show that (\ref{eq: big th}) commutes.\par 

        The top right and bottom right pentagons commute by Theorem \ref{thm: stability}. The middle right square commutes by Proposition \ref{prop: thom cop}. The top left and bottom left triangles commute by Definition \ref{def: thom iso}. The middle left hexagon commutes by (\ref{eq: M E or}) along with the fact that $\pi_E \circ \pi_{\nu_E} = \pi_M$.
    \end{proof}
\subsection{Product}
    
    The following lemma is stated for $\mu_{{{l}}}$, but a similar one holds for $\mu_{{{r}}}$. 
        \begin{theoremaaa}\label{thm: prod stab}
            There is a homotopy commutative diagram of spectra:
            \begin{equation}
                \begin{tikzcd}
                    \frac{\cL M^{-TM}}{\partial \cL M^{-TM}} \wedge \Sigma^\infty_+ \cL M \arrow[r, "\mu_{{{l}}}"] \arrow[d, "\alpha \wedge \iota"] &
                    \Sigma^\infty_+ \cL M \\
                    \frac{\cL D\nu^{-TD\nu}}{\partial \cL D\nu^{-TD\nu}} \wedge \Sigma^\infty_+ \cL D\nu \arrow[r, "\mu_{{{l}}}"] &
                    \Sigma^\infty_+ \cL D\nu \arrow[u, "\pi"]
                \end{tikzcd}
            \end{equation}
            where $\alpha$ is as in Lemma \ref{lem: stab dom}.  
        \end{theoremaaa}
        \begin{proof}
            We choose {embedding data} $Q$ for $M$ extending the embedding $e$, and use this to define {embedding data} $Q'$ for $D\nu$ as in the proof of Theorem \ref{thm: stability}. We take $\alpha$ to be as in Theorem \ref{thm: stability}. Then the following diagram of spaces commutes up to homotopy:
            \begin{equation}\label{eq: prod stab unst}
                \begin{tikzcd}
                    \frac{\cL M^{D\nu}}{\partial \cL M^{D\nu}} \wedge \cL M_+ \arrow[r, "\mu_{{{l}}, unst}^Q"] \arrow[d, "\alpha \wedge \iota"] 
                    &
                    \Sigma^L_+ \cL M
                    \\
                    \frac{\cL D\nu}{\partial \cL D\nu} \wedge \cL D\nu_+ 
                    \arrow[r, "\mu_{{{l}}, unst}^{Q'}"] 
                    &
                    \Sigma^L_+ \cL D\nu \arrow[u, "\pi"].
                \end{tikzcd}
            \end{equation}
            via a homotopy constructed similarly to the one in Theorem \ref{thm: stability}, interpolating between the different incidence conditions (and first coordinates) obtained from going the two different ways around (\ref{eq: prod stab unst}).
        \end{proof}

\section{Homological comparisons: coproduct}\label{sec: cop hom comp}

    Let $M$ be a closed oriented manifold of dimension $n$. In this section we prove that by taking homology and applying the Thom isomorphism,  the spectral coproduct defined in Section \ref{sec: coproduct section} recovers the Goresky-Hingston coproduct as defined in \cite{hingston2017product} and \cite{Naef-Rivera-Wahl}.  Note that the homology coproduct currently existing in the literature only deals with the case that $M$ has no boundary, so that's the one we treat in this section. 
    
    In Section \ref{sec: GH cop} we recap the definition of the Goresky-Hingston coproduct from the literature. In \cref{sec: geo coproduct} we give a geometric model for the homology coproduct using transversality, which we use as an intermediate step for our comparisons. It follows the constructions in \cite{Chataur} which gives a similar description for the Chas-Sullivan product, and \cite{hingston2017product} which gives a similar description for the coproduct for some (but not all) homology classes. The precise statements of the comparisons are given in Propositions \ref{prop: geo GH comp} and \ref{prob: comp cop geo sp} and Corollary \ref{cor: GH spec cop}.

    \subsection{Goresky-Hingston coproduct}\label{sec: GH cop}
        In this section we recap the definition of the Goresky-Hingston coproduct, following \cite[Section 2.2]{Naef-Rivera-Wahl}.
        \begin{remark} 
            The definition we give here differs from \cite[Section 2.2]{Naef-Rivera-Wahl} only in that, corresponding to the conventions in Section \ref{sec: loops}, we restrict to working with constant speed loops in the domain and codomain. This is unproblematic since the inclusion of constant speed loops into all loops induces an isomorphism in homology. That said, it will still be convenient at one stage to consider the space of free loops of not necessarily constant speed, which we denote by $\widetilde{\cL M}$.
        \end{remark}
        
        Assume $M$ is equipped with a Riemannian metric. Let $\tau_M \in H^n(DTM, STM)$ be the Thom class determined by the given orientation on $M$. Let $\Delta: M \hookrightarrow M \times M$ be the diagonal embedding. We choose a tubular neighbourhood of the diagonal $\Delta(M)$ as follows: let $\sigma_\Delta: DTM \to M \times M$ send 
        \begin{equation} \label{eq: tub diag}
            v \in (DTM)_p \mapsto (p, \operatorname{exp}_p(v))
        \end{equation}
        This also identifies the normal bundle of the diagonal $\nu_\Delta$ with $TM$. Let $U_M = \operatorname{Im}(\sigma_\Delta)$. \par 
        We may push forward the Thom class $\tau_M$ along the diffeomorphism of pairs $\sigma_\Delta: (DTM, STM) \to (U_M, \partial U_M)$ to obtain a cohomology class that we also denote by $\tau_M \in H^n(U_M, \partial U_M)$. 
        Let $e_I: \widetilde{\cL M} \times [0,1] \to M \times M$ send $(\gamma, s)$ to $(\gamma(0), \gamma(s))$. Then let $\cF = e_I^{-1}(\Delta(M))$, which we note contains $\widetilde{\cL M} \times \{0,1\}$, and $U_{GH} = e_I^{-1} U_M$, a neighbourhood of $\cF$. Let $\partial U_{GH} = e_I^{-1} \partial U_M$. Let $\operatorname{cut}: \cF \to \cL M \times \cL M$ be the map which sends $(\gamma, s)$ to $(\gamma|_{[0,s]}, \gamma|_{[s, 1]})$ (reparametrised appropriately).\par 
        We pull back $\tau_M$ along the map of pairs $e_I: (U_{GH}, \partial U_{GH}) \to (U_M, \partial U_M)$ to obtain a class that we call $\tau_{GH} = e_I^*\tau_M \in H^n(U_{GH}, \partial U_{GH})$.\par 
        Let $R_{GH}: U_{GH} \to \cF$ be the retraction which sends $(\gamma, s)$ to the concatenation
        \begin{equation}
            \left(\gamma(0) \overset{\gamma|_{[0,s]}} \rightsquigarrow 
            \gamma(s) \overset{\theta}{\rightsquigarrow}
            \gamma(0) \overset \theta \rightsquigarrow 
            \gamma(s) \overset{\gamma|_{[s,1]}}\rightsquigarrow 
            \gamma(0), 
            s\right)
        \end{equation}
        We parametrise this loop so that it reaches the middle $\gamma(0)$ at time $s$ (this is unproblematic since if $s=0$, the first two paths are constant, and similar for $s=1$, and so that the loop has constant speed on both $[0,s]$ and $[s, 1]$ separately.
        \begin{remark}
            The paths $\theta$ are there to force a self-intersection at time $s$. Also note that here we parametrise loops differently to \cite{Naef-Rivera-Wahl}, though this is unproblematic since the space of orientation-preserving homeomorphisms of $S^1$ of Sobolev class $H^1$ preserving 0 is contractible. Similarly they also concatenate with geodesic paths rather than the $\theta$; again the resulting maps are homotopic. 
        \end{remark}
        \begin{definition}(\cite[Definition 2.2]{Naef-Rivera-Wahl})
            The \emph{Goresky-Hingston coproduct} $\Delta^{GH}$ (written $\vee_{TH}$ in \cite{hingston2017product}) is defined to be the following composition:
            \begin{multline}
                H_*(\cL M) \xrightarrow{\cdot \times [0,1]} 
                H_{*+1}\left(\cL M \times [0,1], \cL M \times \{0,1\}\right) \xrightarrow{\tau_{GH} \cap \cdot }
                H_{*+1-n}(U_{GH}, \cL M \times \{0,1\}) \\
                \xrightarrow{R_{GH}} 
                H_{*+1-n}(\cF, \cL M \times \{0,1\}) \xrightarrow{\operatorname{cut}}
                H_{*+1-n}\left(\cL M \times \cL M, (M \times \cL M) \cup (\cL M \times M)\right)
            \end{multline}
        \end{definition}
        See \cite[Figure 1(b)]{Naef-Rivera-Wahl} for a picture.
        \begin{remark}
            As in \cite{Naef-Rivera-Wahl}, we work with the definitions of the cup and cap products for (co)homology from \cite{Bredon}. 
        \end{remark}
    \subsection{Coproduct via geometric intersections}\label{sec: geo coproduct}
        In this section we give an alternative definition of the Goresky-Hingston coproduct using transverse intersections. \par 
        Let $X$ be a closed oriented manifold and $f: X \to \cL M$ a map. We define $Y=Y(f, X)$:
        \begin{equation}\label{eq: Y}
            Y = \overline{
                \left\{(x, t) \in X \times [0,1] \,|\, f(x)(t) = f(x)(0)\,\&\, t \neq 0\right\}
            }
        \end{equation}
        Here $\overline{\cdot}$ denotes the closure in $X \times [0,1]$.
        \begin{lemma}\label{lem: transverse}
            $f$ is homotopic to a map $f': X \to \cL M$ such that $Y(f',X)$ is a transversally cut out submanifold of $X \times [0,1]$, with boundary on $X \times \{0,1\}$ and intersecting it transversally.
        \end{lemma}
        \begin{remark}
            Similarly to (\ref{eq: Y}), let $Y' = \{(x, t) \in X \in [0,1]\,|\, f(x)(t) = f(x)(0)\}$ be the space of self-intersections of loops in the family $f$; note that $Y$ is obtained from $Y'$ by removing elements of the form $(x, 0)$ and then taking the closure.\par 
            The reason we do not work with $Y'$ is that the analogue of Lemma \ref{lem: transverse} can never hold for $Y'$, since $Y'$ always contains $X \times \{0\}$. We remove $X \times \{0\}$; we then take the closure since $Y' \setminus (X \times \{0\})$ is not necessarily compact otherwise.
        \end{remark}
        \begin{proof}[Proof of Lemma \ref{lem: transverse}]
            We first show that the intersection of $Y$ with $X \times [0, \eta)$ can be made smooth, for some small $\eta > 0$.
            
            \par 
            Choose a Riemannian metric on $M$; this induces one on $M \times M$ along with a decomposition $T(M \times M)|_{\Delta(M)} \cong T\Delta(M) \oplus \nu_\Delta$, where $\nu_\Delta$ is the normal bundle of the diagonal. Then for $\eta>0$ small, there are time-dependent sections
            \begin{equation}
                \{\alpha_t\}_{t \in [0, \eta)} \subseteq \Gamma\left((X\times \{0\}, (ev_0) \circ f)^*T\Delta(M)\right) 
                \textrm{ and }
                \{\beta_t\}_{t \in [0, \eta)} \subseteq \Gamma\left(X \times \{0\}, (ev_0) \circ f)^* \nu \right)
            \end{equation}
            such that both are identically 0 for $t=0$, and such that for $(x, t) \in X \times [0, \eta)$, 
            \begin{equation}
                f(x)(t) = \operatorname{exp}_{f(x)(0)}(\alpha_t(x) + \beta_t(x))
            \end{equation}
            The intersection of $Y$ with $X \times [0, \eta)$ is then $\{(x, t)\,|\, \beta_t(x) = 0\}$; this may not be smooth.\par 
            Now let $\beta' \in \Gamma\left(X \times \{0\}, (ev_0) \circ f)^* \nu \right)$ be a generic section, so its zero set $S$ is transversally cut out. \par 
            Then we may homotope $f$ in $X \times [0, \eta)$, without changing $ev_0 \circ f$, so that for $(x,t) \in X \times [0, \eta)$, we have that
            $$f(x)(t) = \operatorname{exp}_{f(x)(0)}\left(t\beta'(x)\right)$$
            Then the intersection of $Y$ with $X \times [0, \eta)$ is $S \times [0, \eta)$, which is smooth.
            We may do the same thing on $(1-\eta, 1]$, so that $Y \cap (X \times ([0, \eta) \cup (1-\eta, 1]))$ is smooth. We may make a $C^0$ perturbation to $f$ away from $X \times ([0, \eta) \cup (1-\eta, 0])$ so that the image of $f$ lies in the subspace of $C^\infty$ loops, and then generically perturb $f$ away from $X \times ([0, \eta) \cup (1-\eta, 0])$ so that $Y$ is smooth everywhere.
        \end{proof}
        We may assume the conclusion of Lemma \ref{lem: transverse} holds. Then the normal bundle $\nu_{Y \subseteq X\times [0,1]}$ of $Y$ in $X \times [0,1]$ is canonically identified with the pullback $(ev_I \circ f)^*\nu_\Delta \cong (ev_0 \circ f)^*TM$; 
        this is oriented and so we obtain a Thom class 
        \begin{equation} \label{eg: Thom GH YX}
            \tau_{Y \subseteq X \times [0,1]} := (f \times Id_{[0,1]})^*\tau_{GH} = (ev_0 \circ f)^*\tau_M
        \end{equation}
        for $\nu_{Y \subseteq X \times [0,1]}$.\par 
        We orient $Y$ so that the natural isomorphism
        \begin{equation}\label{eq: or Y}
            T(X \times [0,1])|_Y \cong \nu_{Y \subseteq X \times [0,1]} \oplus TY
        \end{equation}
        is orientation-preserving (similarly to \cite[Proposition 3.7]{hingston2017product}).
        We use the following result of Jakob \cite{Jakob}:
        \begin{proposition}
            Let $B$ be a space and $A \subseteq B$ a subspace, such that the pair $(B, A)$ is homotopy equivalent to a CW pair. Let $x \in H_*(B, A)$.\par 
            Then $x = f_* (\alpha \cap [X])$ for some $X, f, \alpha$, where
            \begin{itemize}
                \item $X$ is a compact oriented $i$-manifold, for some $i$.
                \item $f: X \to B$ is some map sending $\partial X$ to $A$.
                \item $\alpha \in H^{i-p}(X)$.
            \end{itemize}
            We call such a triple $(X^i,f, \alpha)$ a \emph{geometric representative} for $x$.\par 
        \end{proposition}
        \begin{definition}
            We define the \emph{geometric coproduct} to be the map 
            \begin{equation}
                \Delta^{geo}: H_*(\cL M) \to H_{*+1-n}(\cL M \times \cL M, (M \times \cL M) \cup (\cL M\times M)) 
            \end{equation}
            defined as follows.\par 
            Let $x \in H_p(\cL M)$, and let $(X^i, f, \alpha)$ be a geometric representative for $x$.\par 
            Assume that $Y=Y(f, X)$ satisfies the conclusion of Lemma \ref{lem: transverse}. Let $g = \operatorname{cut} \circ (f \times Id_{[0,1]}): Y \to \cL M \times \cL M$; this sends $\partial Y$ to $(\cL M \times M) \cup (M \times \cL M)$.\par 
            We define 
            \begin{equation}
                \Delta^{geo}(x) = (-1)^{n(i-p)}g_*\left(\alpha|_Y \cap [Y]\right).
            \end{equation}
        \end{definition}
        \begin{remark}
            It is not immediate that the definition for $\Delta$ is independent of choices, since the representation $x=f_*(\alpha \cap [M])$ is not unique. However its failure to be unique is completely classified by Jakob \cite{Jakob}. Using this, one could show independence of choices directly.\par 
            We do not carry this out. Instead, it follows from Proposition \ref{prop: geo GH comp} or Proposition \ref{prob: comp cop geo sp} that $\Delta^{geo}$ is well-defined.
        \end{remark}
    \subsection{From the Goresky-Hingston to the geometric coproduct}
        In this section, we prove:
        \begin{proposition}\label{prop: geo GH comp}
            $\Delta^{geo}(x) = \Delta^{GH}(x)$ for all $x \in H_*(\cL M)$.
        \end{proposition}
        This extends \cite[Proposition 3.7]{hingston2017product} in the case $x= f_*[X]$ for $f: X \to \cL M$ a map from a closed oriented manifold, and is proved similarly.
        \begin{lemma}\label{lem: geo prod int}
            Let $x \in H_p(\cL M)$, and assume $x$ has geometric representative $(X^i, f, \alpha)$. Then 
            \begin{equation}
                x \times [0,1] = \left(f \times Id_{[0,1]}\right)_* \left(\alpha \cap [X \times [0,1]]\right) \in H_{p+1}\left(\cL M\times [0,1], \cL M \times \{0,1\}\right)
            \end{equation}
        \end{lemma}
        \begin{proof}
            \begin{align*}
                \left(f \times Id_{[0,1]}\right)_*\left(\alpha \cap [X \times [0,1]]\right) 
                &= \left(f \times Id_{[0,1]}\right)_* \left(\alpha \cap ([X] \times [0,1]) \right) \\
                &= \left(f \times Id_{[0,1]}\right)_* \left( (\alpha \cap [X]) \times (1 \cap [0,1])\right) \\
                &= f_*(\alpha \cap [X]) \times (Id_{[0,1]})_*[0,1]\\
                &= x \times [0,1]
            \end{align*}
        \end{proof}
        \begin{lemma}\label{lem: cop comp tau int}
            Let $x \in H_p(\cL M)$, and assume $x$ has geometric representative $(X^i, f, \alpha)$, such that $Y = Y(f, X)$ satisfies the conclusion of Lemma \ref{lem: transverse}. Then 
            \begin{equation}
                \tau_{GH} \cap (x \times [0,1]) = (-1)^{n(i-p)} \left(f \times Id_{[0,1]}\right)_* (\alpha|_Y \cap [Y])
            \end{equation}
            noting that $f \times Id_{[0,1]}$ sends $Y$ to $\cF$ and sends $\partial Y$ to $\cL M \times \{0,1\}$.
        \end{lemma}
        \begin{proof}
            \begin{align*}
                \tau_{GH} \cap (x \times [0,1]) 
                &= \tau_{GH} \cap \left(f \times Id_{[0,1]}\right)_* \left(\alpha \cap [X \times [0,1]]\right) \\
                &= \left(f \times Id_{[0,1]}\right)_*
                \left(\left(f \times Id_{[0,1]}\right)^* \tau_{GH} \cap \left(\alpha \cap[X \times [0,1]]\right)\right)\\
                &= \left(f \times Id_{[0,1]}\right)_*
                \left(\left(\left(f \times Id_{[0,1]}\right)^* \tau_{GH} \cup \alpha \right) \cap[X \times [0,1]]\right)\\ 
                &= (f \times Id_{[0,1]})_* \left((\tau_{Y \subseteq X \times [0,1]} \cup \alpha) \cap [X \times [0,1]]\right)\\
                &= (-1)^{n(i-p)} (f\times Id_{[0,1]})_* \left(\alpha \cap (\tau_{Y \subseteq X \times [0,1]} \cap [X \times [0,1]]\right) \\
                &= (-1)^{n(i-p)} (f \times Id_{[0,1]})_* (\alpha|_Y \cap [Y])
            \end{align*}
            The first equality is by Lemma \ref{lem: geo prod int}, the second is by \cite[(A.1)]{hingston2017product}, the third by \cite[Proposition VI.5.1.iv]{Bredon}, the fourth by (\ref{eg: Thom GH YX}), the fifth by \cite[(A.3)]{hingston2017product} and the sixth by Poincar\'e duality (see e.g. \cite[Proof of Proposition 3.7]{hingston2017product}).
        \end{proof}
        \begin{proof}[Proof of Proposition \ref{prop: geo GH comp}]
            Let $x \in H_*(\cL M)$, and $(X^i,f, \alpha)$ a geometric representative for $x$. Note that $f \times Id_{[0,1]}$ sends $Y$ to $\cF \subseteq U_{GH}$, so $R_{GH}$ acts on it by the identity. 
            \begin{align*}
                \Delta^{GH}(x) &= (\operatorname{cut} \circ R_{GH})_* (\tau_{GH} \cap [x \times [0,1]])\\
                &= (-1)^{n(i-p)} (\operatorname{cut}\circ (f \times Id_{[0,1]}))_*(\alpha|_Y \cap [Y])\\
                &= \Delta^{geo}(x)
            \end{align*}
            where the second equality is by Lemma \ref{lem: cop comp tau int}, and the others are by definition.
        \end{proof}
    \subsection{From the geometric to the spectral coproduct}
        In this section, we prove that taking homology and applying the Thom isomorphism, the spectral coproduct from Section \ref{sec: coproduct section} agrees with the geometric coproduct, up to sign. More precisely:
        \begin{proposition}\label{prob: comp cop geo sp}
            The following diagram commutes up to a sign of $(-1)^n$:
            \begin{equation}
                \begin{tikzcd}
                    H_*(\cL M^{-TM} \wedge S^1) \arrow[r, "\Delta_*"] \arrow[d, "\operatorname{Thom} \wedge Id_{S^1}"] &
                    H_*\left(\Sigma^\infty \frac{\cL M}M \wedge \frac{\cL M}M\right) \arrow[d, "="]\\
                    H_{*+n}(\cL M_+ \wedge S^1) & 
                    \tilde H_*\left(\frac{\cL M}M \wedge \frac{\cL M}M\right)\\ 
                    H_{*+n-1}(\cL M) 
                    \arrow[u, "\cdot \times {[}0{,}1{]}"] 
                    \arrow[ur, "\Delta^{geo}"] &
                \end{tikzcd}
            \end{equation}
        \end{proposition}
        \begin{corollary}\label{cor: GH spec cop}
            By Proposition \ref{prop: geo GH comp}, it follows that Proposition \ref{prob: comp cop geo sp} also holds with $\Delta^{geo}$ replaced with $\Delta^{GH}$.
        \end{corollary}
        Choose an embedding $e: M \hookrightarrow \bR^L$ for some $L \gg 0$ and {embedding data} for $M$ extending $e$. Using the identifications from Definitions \ref{def: hom sp}, \ref{def: thom iso}, we see that it suffices to show that the following diagram commutes:
        \begin{equation}\label{eq: big comp diag}
            \begin{tikzcd}
                \tilde H_*\left(\cL M^{D\nu_e} \wedge S^1 \right) 
                \arrow[r, "(\Delta_{unst})_*"] 
                \arrow[d, "\tau_{\nu_e} \cap \cdot"] & \tilde H_*\left(\Sigma^L \frac{\cL M}M \wedge \frac{\cL M}M \right) 
                \arrow[d, "\Phi", shift left=5]\\
                \tilde H_{*+n-L}\left(\cL M_+ \wedge S^1\right) 
                \arrow[u, "\Theta", shift left=5] &
                \tilde H_{*-L}\left(\frac{\cL M}M \wedge \frac{\cL M}M \right) 
                \arrow[u, "{[}-1{,}1{]}^L \times \cdot"] \\
                H_{*+n-L-1}(\cL M) 
                \arrow[u, "\cdot \times {[}0{,}1{]}"] 
                \arrow[ur, "(-1)^n \cdot \Delta^{geo}"] &
            \end{tikzcd}
        \end{equation}
        Here $\Theta$ and $\Phi$ are maps which we define shortly, inverse to the corresponding maps in the reverse direction. Note all vertical maps in (\ref{eq: big comp diag}) are isomorphisms.
        \begin{lemma}\label{lem: prod disc sign}
            Let $x \in \tilde H_{*-L}\left(\frac{\cL M}M \wedge \frac{\cL M}M\right)$ have geometric representative $(X^i, f, \alpha)$, with $f: X \to \cL M \times \cL M$ sending $\partial X$ to $(\cL M \times M) \cup (M \times \cL M)$. Then
            \begin{equation}
                [-1,1]^L \times x = (-1)^{L(i-p)} (Id_{[-1,1]^L} \times f)_* \left(\alpha \cap [[-1,1]^L \times X]\right)
            \end{equation}
        \end{lemma}
        \begin{proof}
            \begin{align*}
                [-1,1]^L \times x &=
                [-1,1]^L \times f_*(\alpha \cap [X]) \\
                &= (Id_{[-1,1]^L} \times f)_*\left([-1,1]^L \times (\alpha \cap [X])\right) \\
                &= (-1)^{L(i-p)} (Id_{[-1,1]^L} \times f)_*\left(\alpha \cap [[-1,1]^L \times X]\right)
            \end{align*}
            where the final equality is by \cite[(A.3)]{hingston2017product}.
        \end{proof}
        We now define the map $\Phi$ from (\ref{eq: big comp diag}). Let $x \in \tilde H_p\left(\Sigma^L \frac{\cL M}M \wedge \frac{\cL M}M\right)$, and let $(X^i, f, \alpha)$ be a geometric representative for $x$, where $f: X \to [-1,1]^L \times \cL M \times \cL M$ sends $\partial X$ to 
        \begin{equation}
            (\partial [-1,1]^L \times \cL M \times \cL M) \cup [-1,1]^L \times ((\cL M \times M) \cup (M \times \cL M))
        \end{equation}
        Generically perturbing $f$ if necessary, we may assume that $f$ is transverse to $\{0\} \times \cL M \times \cL M$. Let $Z = f^{-1}(\{0\} \times \cL M \times \cL M)$.\par 
        $Z$ is a smooth submanifold of $X$ with normal bundle $\nu_{Z \subseteq X}$ canonically identified with $\bR^L$. We orient $Z$ so that the canonical identification
        \begin{equation}\label{eq: 7.14}
            TX|_Z \cong \bR^L \oplus TZ
        \end{equation}
        is orientation-preserving. \par 
        Note that $f|_Z$ sends $Z$ to $\cL M \times \cL M$ and $\partial Z$ to $(\cL M \times M) \cup (M \times \cL M)$. We now define
        \begin{equation}\label{eq: def Phi}
            \Phi(x) := (-1)^{L(i-p)} (f|_Z)_*\left(\alpha|_Z \cap [Z]\right)
        \end{equation}
        It follows from the following lemma that the definition for $\Phi(x)$ is independent of the choice of geometric representative of $x$.
        \begin{lemma}
            $\Phi$ is an inverse to $[-1,1]^L \times \cdot$.
        \end{lemma}
        \begin{proof}
            Let $x \in \tilde H_p\left(\frac{\cL M}M \wedge \frac{\cL M}M\right)$, and let $(X^i, f, \alpha)$ be a geometric representative, where $f: X \to \cL M \times M$ sends $\partial X$ to $(\cL M \times M) \cup (M \times \cL M)$. By Lemma \ref{lem: prod disc sign}, we have that
            \begin{equation}
                [-1,1]^L \times x = (-1)^{L(i-p)} (Id_{[-1,1]^L} \times f)_*\left(\alpha \cap \left[[-1,1]^L \times X\right]\right)
            \end{equation}
            Applying $\Phi$ to the right hand side gives a geometric representative with $Z=\{0\} \times X \cong X$ equipped with the same orientation, so we find that $\Phi([-1,1]^L\times x) = x$.
        \end{proof}
        We now define the map $\Theta$ from (\ref{eq: big comp diag}). Let $x \in \tilde H_p(\cL M_+ \wedge S^1)$ and let $(X^i, f, \alpha)$ be a geometric representative, with $f: X \to \cL M \times [0,1]$ sending $\partial X$ to $\cL M \times \{0,1\}$.\par 
        Let $\tilde X = \operatorname{Tot}(f^*D\nu_e \to X)$, and let $\tilde f: \tilde X \to \operatorname{Tot}(D\nu_e \to \cL M) \times [0,1]$ be the map induced by $f$.\par 
        $\tilde X$ is naturally a smooth manifold of dimension $i+L-n$, and there is a canonical identification
        \begin{equation}
            T\tilde X \cong f^*\nu_e \oplus TX
        \end{equation}
        We orient $\tilde X$ so that this is orientation-preserving. We now define
        \begin{equation}\label{eq: def Theta}
            \Theta(x) := (-1)^{(L-n)(i-p)}\tilde f_*(\alpha \cap [\tilde X])
        \end{equation}
        It follows from the following lemma that the definition for $\Theta(x)$ is independent of the choice of geometric representative of $x$.
        \begin{lemma}
            $\Theta$ is an inverse to $\tau_{\nu_e} \cap \cdot$.
        \end{lemma}
        \begin{proof} 
            Let $x$, as well as a geometric representative $(X^i, f, \alpha)$ for $x$, be as above. Then
            \begin{align*}
                \tau_{\nu_e} \cap \Theta(x) &=
                (-1)^{(L-n)(i-p)}\tau_{\nu_e} \cap \tilde f_* (\alpha \cap [\tilde X]\\
                &= (-1)^{(L-n)(i-p)}\tilde f_* \left( (\tilde f^* \tau_{\nu_e} \cup \alpha) \cap [\tilde X]\right) \\
                &= \tilde f_*\left(\alpha \cap (\tilde f^*\tau_{\nu_e} \cap [\tilde X])\right) \\
                &= (\tilde f|_X)_*(\alpha \cap [X])\\
                &= x
            \end{align*}
            noting that the intersection of $\tilde X$ with the zero section is exactly $X$, with the same orientation.
        \end{proof}
        \begin{proof}[Proof of Proposition \ref{prob: comp cop geo sp}]
            Let $x \in H_{p+n-L-1}(\cL M)$. We show that the result of going both ways around (\ref{eq: big comp diag}) to the bottom right give the same result when applied to $x$. Let $(X^i, f, \alpha)$ be a geometric representative for $x$; we may assume the conclusion of Lemma \ref{lem: transverse} holds. Let $Y= Y(f, X)$, oriented as in (\ref{eq: or Y}). Then by definition,
            \begin{equation}\label{eq: comp cop geo}
                \Delta^{geo}(x) = (-1)^{n(i-p-n+L+1)} g_*(\alpha|_Y \cap [Y])
            \end{equation}
            where $g= \operatorname{cut} \circ (f \times Id_{[0,1]}$.\par 
            By Lemma \ref{lem: geo prod int}, 
            \begin{equation}
                x \times [0,1] = (f \times Id_{[0,1]})_*\left(\alpha \cap [X \times [0,1]]\right)
            \end{equation}
            Let $\tilde X = \operatorname{Tot}(f^* D\nu_e \to X)$, and $\tilde f: \tilde X \to \operatorname{Tot}(D\nu_e \to \cL M)$ the natural map. We orient $\tilde X$ so that the natural identification 
            \begin{equation}
                T\tilde X \cong f^*\nu_e \oplus TX
            \end{equation}
            is orientation-preserving. Then
            \begin{equation}
                \Theta(x \times [0,1]) = (-1)^{(L-n)(i+1-p-n+L)} (\tilde f \times Id_{[0,1]})_* \left(\alpha \cap [\tilde X \times [0,1]]\right)
            \end{equation}
            and so 
            \begin{equation}\label{eq: aaaaaa}
                (\Delta_{unst})_*(\Theta(x \times [0,1])) = (-1)^{(L-n)(i+1-p-n+L)} (\Delta_{unst} \circ (\tilde f \times Id_{[0,1]}))_* \left(\alpha \cap [\tilde X \times [0,1]]\right)
            \end{equation}
            We next compute $\Phi(\ref{eq: aaaaaa})$. Define
            \begin{equation}
                Y' := \left(\Delta_{unst} \circ (\tilde f \times Id_{[0,1]})\right)^{-1} \left(\{0\}\times \cL M \times \cL M\right) \subseteq \tilde X \times [0,1]
            \end{equation}
            Opening up (\ref{eq:closed case}), we see that
            \begin{equation}
                Y' = \left\{(v,x, t)\,|\, x\in X,\,v \in (D\nu_e)_{f(x)},\, t \in [0,1]\, v=0\, f(x)(t) = f(x)(0)\right\}
            \end{equation}
            which is canonically identified with $Y$ as smooth manifolds. Examining the two maps $Y, Y' \to \cL M \times \cL M$, we see that
            \begin{align}\label{eq: simplifying signs Phi}
                \Phi((\Delta_{unst})_*(\Theta(x \times [0,1]))) & = (-1)^{(L-n)(i+1-p-n+L)}(-1)^{L(L-n+i+1-p)} g_*\left(\alpha|_{Y'} \cap [Y']\right)\\
                & = (-1)^{n(i-p-n+L+1)} g_*\left(\alpha|_{Y'} \cap [Y']\right)
            \end{align}
            Note the sign here agrees with that of (\ref{eq: comp cop geo}). It remains to compare the orientations on $Y'$ and $Y$.\par 
            Consider the following diagram of isomorphisms of vector bundles over $Y' \cong Y$ (all pulled back appropriately):
            \begin{equation}\label{eq: wrfewrf}
                \begin{tikzcd} 
                    \nu_e \oplus TM \oplus TY' 
                    \arrow[r, "Y' \cong Y"] 
                    \arrow[d, "- \oplus Id_{TY'}"] &
                    \nu_e \oplus TM \oplus TY
                    \arrow[d, "(\ref{eq: or Y}){,}(\ref{eg: Thom GH YX})"] \\
                    \bR^L \oplus TY' 
                    \arrow[d, "(\ref{eq: 7.14})"] &
                    \nu_e \oplus T(X \times [0,1]) \arrow[d, "="] \\
                    T(\tilde X \times [0,1]) \arrow[r, "="] &
                    \nu_e \oplus TX \oplus \bR 
                \end{tikzcd}
            \end{equation}
            where the isomorphism $-: \nu_e \oplus TM \to \bR^L$ sends $(u,v)$ to $u-v$. Inspecting (\ref{eq:closed case}) and (\ref{eq: tub diag}) shows that the diagram commutes. All isomorphisms except possibly the top horizontal and top left vertical ones are orientation-preserving; the top left verical one preserves orientation up to $(-1)^n$ (since $+: \nu_e \oplus TM \to \bR^L$ is orientation-preserving and $TM$ has rank $n$) so the diffeomorphism $Y' \cong Y$ is orientation-preserving up to $(-1)^n$. Therefore
            \begin{equation}
                [Y] = (-1)^n [Y']
            \end{equation}
            Comparing this with (\ref{eq: comp cop geo}) and (\ref{eq: simplifying signs Phi}), the result follows.
        \end{proof}
    \section{Homological comparisons: product}\label{sec: hom com pro}
            In this section we prove the spectral product we work with in Section \ref{sec: prod} recovers the Chas-Sullivan product by taking homology and applying the Thom isomorphism, using the same strategy as in Section \ref{sec: cop hom comp}. A similar result is shown in \cite[Theorem 1(3)]{cohen2002homotopy}, however here we work with different sign conventions/twists.\par  Let $M$ be a closed oriented manifold of dimension $n$. As in Section \ref{sec: cop hom comp}, similar methods can be applied to the case where $M$ has boundary.
            \subsection{Chas-Sullivan product}
                In this section we recap the definition of the Chas-Sullivan product, following \cite[Section 2.2]{Naef-Rivera-Wahl}. Once again we work implicitly with constant-speed loops, but this does not affect the homology-level product operation.\par 
                Assume $M$ is equipped with a Riemannian metric, and let $\tau_M, \Delta, \sigma_M, U_M$ all be as in Section \ref{sec: GH cop}.\par 
                We define $U_{CS} = (ev_0 \times ev_0)^{-1}U_M \subseteq \cL M \times \cL M$, and $U_{CS} = (ev_0 \times ev_0)^{-1}\partial U_M$. We pull back $\tau_M$ along the map of pairs $ev_0 \times ev_0: (U_{CS}, \partial U_{CS}) \to (U_M, \partial U_M)$ to obtain a class $\tau_{CS} = (ev_0 \times ev_0)^*\tau_M \in H^n(U_{CS}, \partial U_{CS})$.\par 
                Let $R_{CS}: U_{CS} \to \cL M \times_M \cL M$ be the retraction which sends $(\gamma, \delta)$ to 
                \begin{equation}
                    \left(\gamma, \gamma(0) \overset \theta \rightsquigarrow 
                    \delta(0) \overset \delta \rightsquigarrow 
                    \delta(0) \overset \theta \rightsquigarrow 
                    \gamma(0)\right)
                \end{equation}
                and let $\operatorname{concat}: \cL M \times_M \cL M \to \cL M$ send $(\gamma, \delta)$ to the concatenation $(\gamma(0) \overset \gamma \rightsquigarrow \gamma(0) = \delta(0) \overset \delta \rightsquigarrow \delta(0))$.
                \begin{definition}(\cite[Definition 2.1]{Naef-Rivera-Wahl})
                    The \emph{Chas-Sullivan product} $\mu^{CS}$ (written $\wedge_{TH}$ in \cite{hingston2017product}) is defined to be the following composition:
                    \begin{equation}
                        H_*(\cL M) \otimes H_*(\cL M) \xrightarrow{\times} H_*(\cL M \times \cL M) \xrightarrow{\tau_{CS} \cap \cdot}
                        H_{*-n}(U_{CS}) \xrightarrow{\operatorname{concat}} H_{*-n}(\cL M)
                    \end{equation}
                \end{definition}
            \subsection{Product via geometric intersections}
                In this section we recap an alternative definition of the Chas-Sullivan product, using transverse intersections, following \cite{Chataur} (though with slightly different sign conventions). 
                \begin{definition}
                    We define the \emph{geometric product} to be the map 
                    \begin{equation}
                        \mu^{geo}: H_*(\cL M) \otimes H_*(\cL M) \to H_{*-n}(\cL M)
                    \end{equation}
                    defined as follows.\par 
                    Let $x \in H_p(\cL M)$ and $y \in H_q(\cL M)$. Let $(X^i, f, \alpha)$ and $(Y^j, g, \beta)$ be geometric representatives for $x$ and $y$ respectively. Generically perturbing if necessary, we may assume that the maps $ev_0 \circ f: X \to M$ and $ev_0 \circ g: Y \to M$ are transverse. We define $Z$ to be the space
                    \begin{equation}
                        \left\{(a,b) \in X \times Y\,|\, f(a)(0)=g(b)(0)\right\}
                    \end{equation}
                    which is a smooth manifold of dimension $i+j-n$ by assumption. We orient $Z$ so that the natural isomorphism
                    \begin{equation}
                        \nu_M \oplus TZ \cong TX \oplus TY
                    \end{equation}
                    is orientation-preserving. Let $h: Z \to \cL M$ send $(a, b)$ to $\operatorname{concat}(f(a), g(b))$.\par 
                    We define
                    \begin{equation}\label{eq: def mu geo}
                        \mu^{geo}(x,y) = (-1)^{i(j-q) + n(i+j-p-q)} h_*\left((\alpha \cup \beta) \cap [Z]\right)
                    \end{equation}
                    where we pull $\alpha$ and $\beta$ back to $Z$ in the natural way.
                \end{definition}
            \subsection{From the Chas-Sullivan to the geometric product}
                In this section, we prove:
                \begin{proposition}\label{prop: CS geo prod comp}
                    $\mu^{CS}(x, y) = \mu^{geo}(x,y)$ for all $x \in H_p(\cL M), y \in H_q(\cL M)$.
                \end{proposition}
                This extends \cite[Proposition 3.1]{hingston2017product} as well as \cite{Chataur}, with a similar proof.
                \begin{proof}
                    Let $(X^i, f, \alpha)$ and $(Y^j, g, \beta)$ be geometric representatives for $x$ and $y$ respectively. Then
                    \begin{align*}
                        \tau_{CS} \cap (x \times y) &=
                        \tau_{CS} \cap \left(f_*(\alpha \cap [X]) \times g_*(\beta \cap [Y])\right) \\
                        &= (-1)^{i(j-q)} \tau_{CS} \cap \left( (f\times g)_* \left((\alpha \cup \beta) \cap [X \times Y]\right) \right) \\
                        &= (-1)^{i(j-q) + n(i+j-p-q)} (f\times g)_*\left( (\alpha \cup \beta) \cap \left( (f\times g)^* \tau_{CS} \cap [X \times Y]\right)\right) \\
                        &= (-1)^{i(j-q)+n(i+j-p-q)} (f \times g)_* \left( (\alpha \cup \beta) \cap [Z]\right)\\
                        &= \mu^{geo}(x,y)
                    \end{align*}
                \end{proof}
            \subsection{From the geometric to the spectral product}
                In this section, we prove that taking homology and applying the Thom isomorphism, the spectral products (on the left or right) from Section \ref{sec: prod} agree with the geometric product, up to sign. More precisely:
                \begin{proposition}\label{prop: hom comp prod}
                    The following diagrams commute up to a sign of $(-1)^n$:
                    \begin{equation}
                        \begin{tikzcd}
                            H_*\left(\cL M^{-TM} \wedge \Sigma^\infty_+ \cL M\right) 
                            \arrow[d, "\operatorname{Thom}"] 
                            \arrow[r, "(\mu_{{{l}}})_*"] &
                            H_*\left(\Sigma^\infty_+\cL M\right) \arrow[d, "="] &
                            H_*\left(\Sigma^\infty_+ \cL M \wedge \cL M^{-TM}\right) 
                            \arrow[d, "\operatorname{Thom}"] 
                            \arrow[r, "(\mu_{{{r}}})_*"] &
                            H_*\left(\Sigma^\infty_+ \cL M\right)
                            \arrow[d, "="]\\
                            H_{*+n}\left(\cL M \times \cL M\right) &
                            H_*\left(\cL M\right) &
                            H_{*+n}\left(\cL M \times \cL M\right) &
                            H_*\left(\cL M\right) \\
                            H_*\left(\cL M \right) \otimes H_{*+n} \left(\cL M\right) 
                            \arrow[u, "\times"] 
                            \arrow[ur, "\mu^{geo}"] 
                            &
                            &
                            H_*\left(\cL M \right) \otimes H_{*+n} \left(\cL M\right) 
                            \arrow[u, "\times"] 
                            \arrow[ur, "\mu^{geo}"] 
                            &
                        \end{tikzcd}
                    \end{equation}
                \end{proposition}
                \begin{corollary}\label{cor: prod CS spec}
                    By Proposition \ref{prop: CS geo prod comp}, it follows that Proposition \ref{prop: hom comp prod} also holds with $\mu^{geo}$ replaced with $\mu^{CS}$.
                \end{corollary}
                We give the proof for the right-hand diagram; the left-hand case is identical. \par 
                Choose an embedding $e: M \hookrightarrow \bR^L$ and {embedding data} for $M$ extending $e$; since $M$ is closed, we may assume the isotopy $\{\phi_s\}_s$ is constant. Using the identifications from Definitions \ref{def: hom sp}, \ref{def: thom iso} and \ref{def: smash} (choosing sequences $(u_i)_i$ and $(v_i)_i$ with $u_L=L$ and $v_L=0$), we see that it suffices to show that the following diagram commutes:
                \begin{equation}
                    \begin{tikzcd}
                        \tilde H_r\left(\cL M^{D\nu_e} \wedge \cL M_+\right) 
                        \arrow[d, "\tau_{\nu_e} \cap \cdot"] 
                        \arrow[r, "(\mu_{{{l}}{,}unst})_*"] 
                        &
                        \tilde H_r \left(\Sigma^L_+ \cL M\right) 
                        \arrow[d, "\Phi", shift left=5] 
                        \\
                        H_{r+n-L}\left(\cL M \times \cL M\right) 
                        \arrow[u, "\Theta'", shift left=5] 
                        &
                        H_{r-L}\left(\cL M\right)
                        \arrow[u, "{[}-1{,}1{]}^L \times \cdot"] 
                        \\
                        H_p\left(\cL M \right) \otimes H_q\left(\cL M \right) 
                        \arrow[u, "\times"]
                        \arrow[ur, "(-1)^n\mu^{geo}"] 
                        &
                    \end{tikzcd}
                \end{equation}
                where $p+q=r+n-L$, $\Phi$ is as in (\ref{eq: def Phi}) and $\Theta'$ is defined analogously to (\ref{eq: def Theta}).
                \begin{proof}[Proof of Proposition \ref{prop: hom comp prod}]
                    Let $x \in H_p(\cL M)$ and $y \in H_q(\cL M)$; let $(X^i, f, \alpha)$ and $(Y^j, g, \beta)$ be geometric representatives for $x,y$ respectively. 
                    \begin{lemma}\label{lem: cross geo}
                        $x\times y = (-1)^{i(j-q)} (f \times g)_*\left((\alpha \cup \beta) \cap [X \times Y]\right)$
                    \end{lemma}
                    \begin{proof}[Proof of lemma.]
                        \begin{align*}
                            x \times y 
                            &= f_*(\alpha \cap [X]) \times g_*(\beta \cap [Y]) \\
                            &= (f \times g)_*\left((\alpha \cap [X]) \times (\beta \cap [Y])\right) \\
                            &= (-1)^{i(j-q)} (f \times g)_*\left( (\alpha \cup \beta) \cap [X \times Y]\right)
                        \end{align*}
                        where the final equality is by \cite[(A.3)]{hingston2017product}.
                    \end{proof}
                    We first compute $\Theta'(x \times y)$:
                    \begin{align*}
                        \Theta'(x \times y) 
                        &= (-1)^{i(j-q)}\Theta'\left( (f \times g)_* \left( (\alpha \cup \beta) \cap [X \times Y]\right) \right) \\
                        &= (-1)^{i(j-q)} (-1)^{(L-n)(i+j-p-q)} (\tilde f \times g)_*\left((\alpha \cup \beta) \cap [\tilde X \times Y]\right) \\
                    \end{align*}
                where we define $\tilde X = \operatorname{Tot}(f^* D\nu_e \to X)$ and $\tilde f: \tilde X \to \operatorname{Tot}(D\nu_e \to \cL M)$ is the natural map. The first equality is by Lemma \ref{lem: cross geo} and the second by definition of $\Theta'$. Therefore 
                \begin{equation}\label{eq: mu r unst theta}
                    (\mu_{{{l}}}, unst)_*(\Theta'(x \times y)) = (-1)^{i(j-q)+(L-n)(i+j-p-q)} \left(\mu_{{{l}},unst} \circ (\tilde f \times g) \right)_* \left((\alpha \cup \beta) \cap [\tilde X \times Y]\right)
                \end{equation}
                Similarly to the proof of Proposition \ref{prob: comp cop geo sp}, we see that 
                \begin{align*}
                    \Phi\left((\mu_{{{l}}, unst})_*(\Theta'(x \times y))\right) 
                    &= (-1)^{L(i+j+L-n-r) + i(j-q) + (L-n)(i+j-p-q)} h_*\left((\alpha \cup \beta) \cap [Z']\right)\\
                    &=(-1)^{i(j-q) + n(i+j-p-q) } h_*\left((\alpha \cup \beta) \cap [Z']\right)
                \end{align*}
                where $Z' = \left(\mu_{{{l}},unst}\circ (\tilde f \times g)\right)^{-1}(\{0\} \times \cL M)$. $Z'$ is transversally cut out by assumption, and we have a canonical identification $Z \cong Z'$ as smooth manifolds. Since the sign here agrees with that of (\ref{eq: def mu geo}), it suffices to compare the orientations on $Z$ and $Z'$; by the same argument as in the proof of Proposition \ref{prob: comp cop geo sp}, their orientations differ by a factor of $(-1)^n$. Therefore $[Z]= (-1)^n[Z']$; the result follows.
            \end{proof}

\section{Traces and torsion}\label{sec: torsion}
Given a homotopy equivalence $f: N \to Z$ one could ask whether $f$ is a simple homotopy equivalence. A related question arises when classifying diffeomorphism classes of higher dimensional $h$-cobordisms. Namely, one could ask whether an $h$-cobordism is smoothly trivial. \par
In order to prove the main results of this paper we convert the first question into the second. That is, to $f$ we associate a codimension $0$ embedding of manifolds with boundary, $P \subset Q$, so that the complement of $P$ in $Q$ is an $h$-cobordism. We then study the failure of $f$ to be a simple homotopy equivalence by considering instead the triviality of $W$. In particular, we will study the Whitehead torsion, $\tau (W)$, and its image under various trace maps.

So let $W$ be a smooth $h$-cobordism of dimension $n \geq 6$; we assume its boundary is partitioned into two components $M$ and $N$. 
We allow $W$ to have corners, so $M$ and $N$ may not necessarily be smooth everywhere, but this does not affect any of the constructions in this section, nor the ones in \cite{GN}, where  PL manifolds are considered.
In \cite{GN} Geoghegan and Nicas study the obstruction to deforming $W$ to $M$ in a fixed point free manner. They do so by considering the fixed point set of a strong deformation retraction $F: W\times I \to W$. 
To such a deformation retraction they associate an algebraic $1$-parameter Reidemeister trace:

 $$R(W) \in HH_1(\bZ [\pi_1 M])/ \bZ[\pi_1 M],$$
and prove the following:
\begin{theoremaaa}[\cite{GN}, Theorem 7.2] \label{thm: GN 1}
    Let $M$ be a smooth compact manifold of dimension $n \geq 5$, and $\H(M)$  the space of $h$-cobordisms on $M$.  Suppose $\pi_2(M) = 0$. Then the following diagram commutes:\\
    \[\begin{tikzcd}
    K_1 (\bZ[\pi_1(M)]) \arrow[r] \arrow[d, "tr"] & Wh(\pi_1(M)) \cong \pi_0\H(M) \arrow [d, "-R"]\\ 
    HH_1(\bZ [\pi_1 M])\arrow[r]
    & HH_1(\bZ [\pi_1 M])/ \bZ[\pi_1 M].
    \end{tikzcd}\]
    Here the equivalence $Wh(\pi_1(M)) \cong \pi_0\H(M)$ is given by the $s$-cobordism theorem;  $tr$ is the Dennis trace map,  and the horizontal maps are the natural quotient maps.  
    
    \end{theoremaaa}

In order to prove (\ref{eq: pi20 case}) we need to consider other geometric incarnations of the invariant $R(W)$. In \cite{GN} Geoghegan and Nicas further define a geometric $1 $ parameter Reidemeister trace, $\Theta(W) \in H_1(E_F)$, where  $E_F$ is the \emph{twisted free loop space} defined by:  
\begin{equation}\label{eq: tw loop space}
    E_F := \left\{\gamma: I \to W \times I \times W \,|\, \gamma(0) = (x,t,x) \text{ and } \gamma(1) = (y, s, F_s(y)) \text{ for some } x,y,s,t\right\}.
\end{equation}
They construct a map: $$\Psi: H_1(E_F) \to HH_1(\bZ [\pi_1 M]) $$ and prove: 

\begin{theoremaaa}[\cite{GN}, Theorem 1.10]\label{thm: NG 2}

    $\Psi(\Theta(W)) =-R(W)$. Moreover, when $\pi_2(M) = 0$, $\Theta(W)$ vanishes if and only if $R(W)$ vanishes. 
\end{theoremaaa}

In this section we construct two other variations of the $1$ parameter Reidemeister trace. In \S \ref{sec: framed bordism invariant} we define a framed bordism class $[T] \in \Omega^{fr}_1(\LL W, W)$, which is used in the statement of our main \cref{main theorem}. Composing with the homotopy equivalence $r := F_1: W \to M$, this construction gives a well defined map: $$T_*: \pi_0 \H (M) \to \Omega^{fr}_1(\LL M, M).$$

Combining \cref{lem: T=theta}, \cref{thm: GN 1} and \cref{thm: NG 2} we obtain:

\begin{proposition}\label{lem: comp}
    Suppose $\pi_2(M) = 0$. Then the following diagram commutes:
    \[\begin{tikzcd}
    K_1 (\bZ[\pi_1(M)]) \arrow[r] \arrow[d, "tr"] & \pi_0\H(M) \arrow [d, "h_* \circ T_*"]\\ 
    HH_1(\bZ [\pi_1 M])\arrow[r]
    & H_1(\LL M, M).
    \end{tikzcd}\]
 Here $h_*: \Omega^{fr}_1(\LL M, M ) \to H_1(\LL M, M )$ is the Hurewicz homomorphism. The bottom horizontal arrow is the composition:
    
    $$HH_1(\bZ [\pi_1 M]) \xrightarrow[]{\Psi^{-1}} H_1(E_F) \xrightarrow[]{(r \circ \mu)_*} H_1(\LL M)
    \xrightarrow{q}H_1(\LL M, M),$$ 
    $\mu$ is given in  \cref{lem: tw loop eq reg loop}, $r: W \to M$ is the retraction $F_1$, $\Psi$ is the isomorphism of \cite{GN}[\S 6A]\footnote{\textit{loc. cit.} uses a twisted version of $HH_1(\bZ[\pi_1 M])$, taking into account an automorphism $\phi$ of $\pi_1 M$. In our case, the twisting is trivial, because $F$ acts as the identity on $\pi_1$.}, and $q$ is the projection map. 
\end{proposition}

In \cref{sec: Reid Tr} we construct the $1$ parameter Reidemeister trace:
$$Tr(W): \Sigma \mathbb{S}\to \Sigma^{\infty}\frac{\cL W}W$$ on spectra. This definition adapts a homotopical construction of the  Reidemeister trace to the 1-parameter and relative settings, see for example \cite{malkiewichparametrized}.

The invariant $Tr(W)$ is shown to correspond to $[T]$ under the Pontrjagin-Thom isomorphism in \S \ref{sec: T Tr}. It is also used as a prototype for the definition of the operations: 
$$\Xi_{{{r}}}, \Xi_{{{l}}}: \Sigma \frac{\cL W^{-TW}}{\partial \cL W^{-TW}} \to \Sigma^\infty \frac{\cL W}W \wedge \frac{\cL W}W $$
constructed in Section \ref{sec:traces on loop space}. In \cref{thm: T Xi} in Section \ref{sec: cop defect}, we show that the maps $\Xi_{{{r}}}$ and $\Xi_{{{l}}}$ can be described by taking the Chas-Sullivan product with the class $[T]$.

\subsection{The framed bordism invariant}

\label{sec: framed bordism invariant}
\subsubsection{The definition of $[T]$}

For the rest of this section, we assume that $W$ is embedded as a codimension $0$ submanifold of $\R^L$.

Define subsets $\tilde{T}, T^\circ, T$ of $W \times [0,1]$ as follows.
$$\tilde{T} := \left\{ (x,t) \in W\times [0,1] \,|\, F_t(x) = x\right\}$$
$$T^\circ := \left\{(x,t) \in \tilde{T} \,|\, t \neq 0 \textrm{ and } x \notin  M\right\}, $$
and let \begin{equation}\label{eq: T}
    T = \bar{T}^\circ
\end{equation} 
be the closure of $T^\circ$ in $W \times [0,1]$, which we note is compact.  

Though $\tilde T$ is the most natural of these to define, it can never be transversally cut out: if it was, it would have dimension 1, but $\tilde T$ always contains $W \times \{0\}$ and $M \times I$. 

Note that if $T$ is transversally cut out and hence a 1-manifold, its boundary can lie on $W \times \{1\}$, $M$, $N$ and $W \times \{0\}$. We prove next that $T$ can be assumed to be transversally cut out with no boundary terms on the first three types, and later (in Lemma \ref{lem: no bound T}) that we can also choose $F$ so it has no boundary of the fourth type either.
\begin{lemma}\label{lem: gen pert F}
    There is a $C^0$-small perturbation of $F$ through strong deformation retractions, such that $T$ is a smooth $1$-dimensional submanifold of $[0,1]\times W$, possibly with boundary which must lie on $\{0\} \times W^\circ$.
\end{lemma}
Note this lemma cannot hold for $\tilde T$ instead of $T$, since $\tilde T$ always contains $(W \times \{0\}) \cup (M \times [0,1])$.
\begin{proof}
    
    If we could perturb $F$ arbitrarily, standard transversality results would imply the lemma. Instead, $F$ is constrained along $M \times [0,1]$, $W \times \{1\}$ and $W \times \{0\}$. We first argue that the lemma holds in some neighbourhood of this region.\par

    $\tilde T$ does not intersect $W \times \{1\}$ except along $M \times \{1\}$. We may perturb $F$ such that for all $x$ sufficiently close to $M$, the path $\{F_t(x)\}_{t \in [0,1]}$ is the embedded geodesic to the closest point in $M$. Now any point in $(x, t) \in \tilde T$ such that $x$ is near to $M$ must have $x \in M$.\par 
    It follows that now $T$ can only intersect $(W \times \{0, 1\}) \cup (M \times [0,1])$ along $W \times \{0\}$.\par 
    To ensure $T$ is smooth near $W \times \{0\}$, we consider the vector field $Z$ on $W$, whose value at $p \in W$ is $\frac d{ds}|_{s=0} F_s(p)$. This is constrained so that it points inwards along $N$ and vanishes along $M$. We may generically perturb $F$ such that $Z$ intersects the zero section transversally away from $M$. We may further perturb $F$ so that for $\eta > 0$ small, for all $p \in W$, the path $\{F_t\}_{t \in [0,\eta)}$ is a geodesic. Now the intersection of $T$ with $W \times [0,\eta)$ agrees with $S \times [0,\eta)$.\par 
    Therefore $T$ is smooth near $W \times \{0\}$; perturbing generically away from the region on which $F$ is constrained allows us to obtain the lemma.
\end{proof}
Let 
\begin{equation}
    i: T \hookrightarrow W \times [0,1]
\end{equation}
be the natural inclusion, and denote the normal bundle by $\nu_i$. 

Let $\psi: \nu_i \to \R^L$ be the isomorphism of vector bundles sending $(v,t)$ in the fibre of $\nu_i$ over $(x,s)$ to 
\begin{equation}   \label{eq: psi}
    \psi(v,t) := v- dF_{(x,s)}(v,t).
\end{equation} 
That $\psi$ is an isomorphism is exactly the condition that $T$ is transversally cut out. 

We consider the natural map $f: T \to \cL W$ sending $(x,t)$ to the loop $F|_{[0,t]}$ from $x$ to itself.
Note that $\psi$ equips $T$ with a stable framing which  therefore defines a class $[T]$ in $\Omega^{fr}_1(\cL W,W)$. 
\begin{lemma}
    The space of strong deformation retractions is contractible.
\end{lemma}
\begin{lemma}\label{lem: T ind}
    The class $[T] \in \Omega_1^{fr}(\cL W,W)$ is independent of the choice of $F$.
\end{lemma}
\begin{proof}
   Let $F'$ be another choice of strong deformation retraction as above. Since the space of such deformation retractions is contractible, there is a $1$-parameter family of strong deformation retractions $\left\{F^\tau: W \times I \to W\right\}_{\tau \in [0,1]}$ such that $F^0 = F$ and $F^1 = F'$. Generically perturbing $\{F^\tau\}$ relative to $\{\tau \in \{0,1\}\}$ similarly to Lemma \ref{lem: gen pert F}, and letting $S$ be the closure of 
    \begin{equation*}
        S^\circ := \left\{(x, t, \tau) \in W \times [0,1]^2 \,|\, F^\tau_t(x) = x,\,t\neq 0,\,x \notin M\right\}
    \end{equation*}
    provides the desired bordism; this can be equipped with a stable framing similarly to (\ref{eq: psi}).
\end{proof}

\subsubsection{Definition of $\Theta(W)$ }
  
In this subsection we recall the definition of $\Theta (W) \in H_1(E_F)$ appearing in \cite[Section 6]{GN}.
\par  

Let $(x,t),(y,s) \in W \times [0,1]$ be two fixed points of $F$ (i.e. $F_t(x) =x$, $F_s(y)=y$). We say that $(x,t)$ and $(y,s)$ are \emph{in the same fixed point set} if there is some path $\gamma$ in $W \times I$ from $x$ to $y$, such that the loop $(pr_1\circ \gamma) \star (F\circ \gamma)^{-1}$ is homotopically trivial (where $pr_1$ projects to the first factor of $W \times [0,1]$). This defines an equivalence relation on the set of fixed points.\par  

The manifold $T$,  constructed in \cref{eq: T}, consists of a union of circles and arcs. Note that fixed points in the same path component of $T$ are in the same fixed point class. A geometric intersection invariant in \cite{GN} is defined using the submanifold $A \subset T$ consisting only of the union of those circles of intersections not in the same fixed point class as the fixed points of $F_0$ and $F_1$.  

\par In \cite[Page 432]{GN} an orientation of $A$ is defined as follows: to an isolated fixed point $x$ of $F_t$, one associates an index $i(F_t, x)$, which is the degree of the map: $$id -F_t: B_\epsilon(x) \setminus\{x\} \to \R^L \setminus\{0\}.$$ Here $B_\epsilon$ is a small neighborhood of $x$ in $W \times \{t\}$ not containing any other fixed point of $F_t$. The transversality hypothesis implies that generically $i(F_t,x) = \pm 1$, and both values occur on each loop. The orientation on each circle of fixed points, $S$, is given by picking any $(x,t) $ for which $i(F_t, x) =1$, and orientating $S$ near $(x,t)$ in the direction of increasing time. 

Let $E_F$ be the twisted loop space defined in \cref{eq: tw loop space}.
Then $A$ is a closed oriented $1$-manifold which includes into $E_F$ by constant loops and hence defines a class which we define $\Theta(W) \in  H_1(E_F)$ to be. 

\subsubsection{Relating $[T]$ and $\Theta(W)$}

To compare $\Theta(W)$ and $[T]$ we need to consider the following.  Firstly, we need to relate the target of the invariants; the definition of $[T]$ involves the free loop space $\LL W$ while $\Theta(W)$ concerns the twisted loop space $E_F$. Moreover, $\Theta (W)$ consists of a choice of orientation and defines a class in $H_1(E_F)$, while $[T]$ consists of a choice of framing, and defines a class in $\Omega_1^{fr}(\LL W,W)$. Secondly, $\Theta(W)$ is defined by manually discarding circles of intersections in the fixed point class of $F_0$ and $F_1$. The analogous procedure in the definition of $[T]$ corresponds to modding out $\LL W$ by constant loops. \par

We show that if $\pi_2(W) = 0$, after passing to homology, the two invariants agree. For this to make sense, we must first relate the groups in which these invariants live.

\begin{lemma} \label{lem: tw loop eq reg loop}
There exists a homotopy equivalence   $\mu: E_F \to \LL W$.
\end{lemma}
\begin{proof}
We will construct $\mu$ as the composition of several homotopy equivalences. Let $\tilde E_F$ be the pullback in the diagram:

\begin{center}
    \begin{tikzcd}
        \tilde E_F \arrow[r] \arrow[d] & \P( W)\times \P(I) \arrow[d]\\
        W\times I \times I \arrow[r] & W\times I \times I \times W\\
    \end{tikzcd}
\end{center} where the bottom horizontal map is given by $(w,t,s) \mapsto (w, t,s, F_s(w))$, the right vertical map is given by $(\alpha, \beta) \mapsto (\alpha(0), \beta(0), \beta(1), \alpha(1))$, and $\cP\cdot$ denotes the path space. Then $\tilde E_F$ consists of pairs $(\alpha, \beta) \in \P(W )\times \P(I)$ 
satisfying $F_{\beta(1)}(\alpha(0)) = \alpha(1)$ 

Let $\gamma$  be a path in $E_F$, so $\gamma(0) = (x,t,x) $ and $\gamma(1) = (y, s, F_s(y))$. We can decompose $\gamma$ into components $(\gamma_1, \gamma_I, \gamma_2)$ by projecting into the first, second, and third factors in $W \times I \times W$. So that $\gamma_1$ is a path from $x$ to $y$, $\gamma_2$ is a path from from $x $ to $F_s(y)$, and $\gamma_I$ is a path in $I$ from $t$ to $s$. 

Define $\Gamma: E_F \to \tilde E_F $ by  sending $\gamma$ to 
$$ (y \overset{\overline\gamma_1} \rightsquigarrow x \overset{\gamma_2} \rightsquigarrow F_s(y), \gamma_I)$$
where we choose the concatenation of $y \overset{\overline \gamma_1} \rightsquigarrow x \overset{\gamma_2} \rightsquigarrow F_s(y)$ to happen at time equals to $1 /2$. Then $\Gamma$ is a homotopy equivalence admitting a homotopy inverse sending $(\alpha, \beta)$ to $(\overline\alpha_{[0,1/2]}, \beta, \alpha_{[1/2, 1]})$ (and appropriately rescaling). 

Note that since $\P(I)$ is contractible, $\tilde E_F$ is further homotopy equivalent to $\bar E_F$, the pullback of the diagram:

\begin{center}
    \begin{tikzcd}
        \bar E_F \arrow[r] \arrow[d] & \P(W) \arrow[d]\\
        W\times I \arrow[r] & W \times W\\
    \end{tikzcd}
\end{center}
where the right vertical map is given by $\gamma \to (
\gamma(0), \gamma(1))$, and the bottom horizontal map is given by $(w,s) 
\to (w, F_s(w))$. 

Then $\bar E_F$ consists of pairs $(\alpha, s)$ where $\alpha: [0,1] \to W $ is such that  $\alpha(1) =F_s(\alpha(0))$. The homotopy equivalence is given by the forgetful map sending $(\alpha, \beta) \to (\alpha, \beta(1))$. 

We further define 
$$\bar \Gamma: \bar E_F \to \LL W \times I $$ 
by sending $(\alpha, s)$ to: 
$$(\alpha(0) \overset{\alpha} \rightsquigarrow F_s(\alpha(0)) \overset{\overline F|_{[0,s]}}\rightsquigarrow \alpha(0),s).$$
Then $\bar \Gamma$ is a homotopy equivalence with inverse given by $$(\delta, s)  \mapsto (\delta(0) \overset{\delta} \rightsquigarrow \delta(0) \overset{F|_{[0,s]}}\rightsquigarrow F_s(\delta(0)), s).$$
Lastly, note that the forgetful map $ \LL W \times I \to \LL W $ is a homotopy equivalence. The homotopy equivalence $\mu$ is given by the composition of $\Gamma$, $\bar \Gamma$ and the forgetful map. 
\end{proof}

\begin{remark}
    We provide a more direct description of the map $\mu$. $\mu$ sends a path $(\gamma_1, \gamma_I, \gamma_2): I \to W \times I \times W$ from $(x,t,x)$ to $(y,s,F_s(y))$ to the following loop in $W$:
    \begin{equation}
        \left(y \overset{\overline \gamma_1}\rightsquigarrow x \overset{\gamma_2}\rightsquigarrow F_s(y) \overset{\overline F|_{[0,s]}}\rightsquigarrow y\right)
    \end{equation}
\end{remark}

The homotopy equivalence $\mu$ from \cref{lem: tw loop eq reg loop} induces a map: $$\mu_*: H_1(E_F) \to H_1(\LL W),$$
which we can compose with the quotient map: $$\pi: H_1(\LL W) \to H_1(\LL W,W).$$ To complete the comparison of $[T]$ and $\Theta (W)$, we will need to consider the Hurewicz map 
$$h_*: \Omega^{fr}_1(\LL W, W ) \to H_1(\LL W, W ).$$

In order to define $h_*$, we must fix conventions for how a stable framing on a manifold induces an orientation.\par 

Given a stably framed manifold, one consistent choice of orientation is given as follows. Let  $[Y] \in \Omega^{fr}_1(\LL W) $ be represented by $f:Y \to \LL W $; choose an embedding $e: Y \to \R^{L+1}$, with normal bundle $\nu_Y$, and framing $\phi: Y\times \R^L \to \nu_Y $ representing the stable framing on $Y$. Let $\{v_0, v_1, ... , v_L\}$ be the standard basis of $\R^{L+1}$
 and $\{ v_1, ... , v_L\}$ a basis for $\R^L$. For $y \in Y$, there exists a unique vector $v_y \in T_y Y \subset \R^{L+1}$ such that the matrix $(\phi(y,v_1) , ..., \phi(y, v_L), v_y)$ has determinant $1$. We orient $Y$ so that the positive orientation points in the direction of $v_y$. \par

\begin{lemma}\label{lem: T=theta}
    Suppose $\pi_2(W) = 0$. Then $\pi \circ \mu_*(\Theta(W)) =  h_*([T])$.
\end{lemma}

\begin{proof}

Both invariants are defined starting with the manifold $T$. Since in the definition of $\Theta (W)$ we discard the arcs and circles in $T\setminus A$, we need to consider their contribution to $h_*[T]$.  Note that for $(x,t) \in  T\setminus A$, the loop $F|_{[0,t]}(x)$ is contractible. Let $\LL_0 W$ be the path component of $\LL W$ consisting of contractible loops. When $\pi_2(W) =0$, $\pi_1(
\LL_0 W)$ is isomorphic to $ \pi_1(W)$ (by the long exact sequence associated to the fibration $\Omega_0 W \to \LL_0 W \to W$) and is generated by constant loops. Hence  $\pi_1(\LL_0 W, W)=0$, and the contributions of $T\setminus A$ die in $H_1(\LL W, W)$.  

By chasing the homotopy equivalence $\mu$ we see that $\mu$ sends the constant loop at $(y, s, F_s(y))$, associated to a fixed point $(y,s)$, to the loop $F_{[0,s]}$ based at $y$. Hence,  up to a question of orientation, we have the equivalence $\pi \circ \mu_*(\Theta(W)) =  h_*([T])$. So the last thing to consider is the equivalence of orientations. \par 
Let $x$ be a fixed point of $F_t$, such that $i(F_t,x) =1$.  Let $(v_1, ..., v_L)$ be the standard basis for $\R^{L}$, and $(v_0, v_1, ..., v_L)$ be the standard basis for $\R^{L} \oplus \bR$. This choice of basis induces a trivialization of $T(W\times [0,1]) \cong \R^L \oplus \bR$.

Recall the map 
$$ Id -F_t: B_\epsilon(x) \setminus \{x\} \to \R^L \setminus\{0\}$$ 
defining the index $i(F_t,x)$. Note that $Id-F_t$ extends to $B_\epsilon$ and we denote its differential at $x$ by $\phi$. For generic $(x, t)$, $\phi$ is a linear isomorphism; we may assume this holds. Note if the degree of $Id - F_t$ equals to one, then $\phi$ is orientation preserving, and hence has positive determinant.

Let $\psi: \R^L \oplus \bR \to  \R^L$ be the map sending $(v,s) \in T(W \times [0,1])$ in the fibre over $(x,t)$

to 
\begin{equation}   
    \psi(v,s) := v- dF_{(x,t)}(v,s).
\end{equation}  
Note that $\ker \psi \cong TT$. Let $\tilde \psi :\R^L \oplus \bR \to \R^L \oplus \bR$ be the map sending $(v,s)$ in the fibre over $(x,t)$ to 
\begin{equation}   
    \tilde\psi(v,s) := (s,v- dF_{(x,t)}(v,s)).
\end{equation}
Then $\tilde \psi^{-1}$ defines an isomorphism $\R^L \oplus \bR \to \R^L \oplus \bR$ sending the final $\R$ factor to $TT$ (by the implicit function theorem). The matrix of $\tilde \psi$ is given by 
\begin{equation*}
    \begin{pmatrix}
        \phi & * \\
        0 & 1
    \end{pmatrix}
\end{equation*}
and hence has positive determinant, and the matrix of $\tilde \psi^{-1}$ is given by:
\begin{equation*}
    \begin{pmatrix}
        \phi^{-1} & * \\
        0 & 1
    \end{pmatrix}
\end{equation*}
where the vector \begin{equation}
  \tau := \begin{pmatrix}
        *\\
        1
    \end{pmatrix}
    = \tilde \psi^{-1}\begin{pmatrix}
        0\\
        1
    \end{pmatrix} \in TT_{x,t}
\end{equation} 
is oriented in the direction of increasing time (because its first coordinate is positive). 
The first $L$ columns of $\tilde \psi^{-1}$ don't necessarily give a framing of $\nu_T$, but by performing column operations (specifically those which don't change the sign of the determinant), i.e. projecting off of the subspace spanned by $\tau$, we arrive at a matrix $(\chi, \tau) $ which has positive determinant, and is such that: 
$$\R^L \xrightarrow[]{\chi} \R^L \oplus \bR\xrightarrow{\psi} \R^L  $$ 
is the identity
and hence induces our choice of framing of $T$. Note that after possibly rescaling by a positive number, $\tau$ defines an orientation of $TT$, consistent with the Hurewicz isomorphism defined above. Since $\tau$ is oriented in the direction of increasing time, it follows that the two conventions for orienting $T$ agree.
\end{proof}

\subsection{The Reidemeister trace of an $h$-cobordism}\label{sec: Reid Tr}
Let $W$ be a smooth $h$-cobordism of dimension $n$. $\partial W$ consists of two boundary components, which we call $M$ and $N$. In this section we define the Reidemeister trace of $W$ as a map of spectra:

$$Tr: \Sigma^\infty S^1\to \Sigma^{\infty}\frac{\cL W}W$$
and show that it is related to the framed bordism invariant $[T]$ by the Pontrjagin-Thom isomorphism in Section \ref{sec: T Tr}.\par 
We will need to make some choices, analogously to the definition of the coproduct. 

\subsubsection{Choices}
We choose an extension $$W^{ext}:= M\times [0,1]\cup_M W \cup_N \times N \times [0,1]$$ of $W$ as in Section \ref{sec: CD}.
\begin{definition}\label{def: trace data hcob}
     We define \emph{Trace data} for $W$ to be a tuple $\bar {R} = ( e, \rho^{ext}, \zeta, V, \epsilon, \lambda,  F)$ consisting of:
    \begin{enumerate}[(i).]
        \item \label{it: tr embedding}A smooth embedding $e: W^{ext} \hookrightarrow \R^L$. We write $\nu_e$ for the normal bundle of this embedding, defined to be the orthogonal complement of $TW^{ext}$. Note that $e$ canonically equips both $TW^{ext}$ and $\nu_e$ with metrics, by pulling back the Euclidean metric on $\R^L$. Let $\pi_e: \nu_e \to W^{ext}$ be the projection map.
        \item \label{it: tr tubular ngh} A tubular neighbourhood $\rho^{ext}: D_2 \nu_e \hookrightarrow \bR^L$. More precisely, $\rho^{ext}$ is a smooth embedding, restricting to $e$ on the zero-section. We let $\tilde U$ be the image of $\rho^{ext}$. We let $\rho$ be the restriction of $\rho^{ext}$ to the unit disc bundle of $\nu_e$ over $W$, and $U$ the image of $\rho$; this lies in the interior of $\tilde U$. In symbols: $\rho :=\rho^{ext}|_{D_1 \nu_e|_W}$, $U:= \operatorname{Im}(\rho)$ and $\tilde U = \operatorname{Im}(\rho^{ext})$. From the choices above we obtain a retraction $r: \tilde U \to W$ defined to be the composition of $(\rho^{ext})^{-1}$, the projection to $W^{ext}$, and the natural map $W^{ext} \to W$.

        \noindent We require that along the zero section, the derivative of $\rho^{ext}$ in the fibre direction agrees with the canonical inclusion of vector bundles $\nu_e \to \bR^L$.
        \item \label{item: zeta small} A real number $\zeta > 0$. 
        
        \noindent We require that \textrm{$\zeta$} is small enough that whenever $x, y \in M$ satisfy $\lVert x-y\rVert \leq \zeta$, the straight-line path between them $[x,y]$ lies inside $\tilde U$.
        
    \item \label{it: tr v.f.} A  vector field $V$ on $W^{ext}$ such that:

    \begin{enumerate}[(a).]
        \item $V|_W$ points strictly inwards at $N$ and strictly outwards at $M$. For simplicity, we require that for $(m,t) \in M \times (0,1]$, $V_{(m,t)}$ is a non zero rescaling of $V_{(m,0)}$, and similarly for $(n,t) \in N \times [0,1]$. 
        
        We denote the flow of $V$ by $\{\phi_s(x)\}_{s \geq 0}$. A priori this isn't defined for all time since the flow can leave along one of the components of $\partial W^{ext}$; we define the flow to be constant in $s$ as soon as it hits this component of $\partial W^{ext}$. 
        \item Let $\pi: W^{ext} \to W$ be the natural projection. For $x \in W^{ext}$, the length of the path $\pi(\{\phi_s(x)\}_{s \in [0,1]})$ is $\leq \zeta/4$.
    \end{enumerate} 
      
        \item A real number $\eps > 0$ sufficiently small such that: 
        \begin{enumerate}[(a).]
            \item $\eps < \zeta/8$.
            \item  $\tilde U$ contains an $\eps$-neighbourhood of $W$ (with respect to the Euclidean distance in $\bR^L$). 
            \item \label{it: tr eps small 1} If $x\in U$, $y \in e(W^{ext}) $ and  $\lVert x-y \rVert \leq \eps$  then the straight line path $[x,y]$ 
            lies in $\tilde U$, and $r([x,y])$ has length $\leq \zeta/4$.
            \item \label{it: tr eps small 2 }
            The Euclidean distance: $d\left(\phi_1(M), \rho(D\nu|_{ W}))\right) \geq 2\eps$ 
            \item\label{it: tr eps small 3} The Euclidean distance: $d\left(\phi_1(N), \rho(D\nu|_{ N}))\right) \geq 2\eps$
         
        \end{enumerate}
        
        \item \label{it: tr lambda big}$\lambda > 0$, large enough such that:
        
        \[\lambda \cdot d(\rho(S\nu_e|_W), e(W^{ext})) \geq \sqrt L\]  where $S\nu_e$ is the unit sphere bundle of $\nu_e$; note that this distance on the left hand side is at least $\eps$.
        \item A strong deformation retraction $F: W \times [0,1] \to W$ onto $M$.
    \end{enumerate}
  We write $TD^L(W)$ for the simplicial set whose $k$-simplices consist of the set of continuously-varying families of tuples of trace data, parametrised by the standard $k$-simplex. 
    
    \end{definition}

\begin{lemma}
    The forgetful map $TD^L(W) \to \operatorname{Emb}(M^{ext}, \R^L)$ which forgets all the data except the embedding $e$ is a trivial Kan fibration and hence a weak equivalence.
\end{lemma}

\begin{proof}
  This lemma is the same as that of \cref{lem: kan}, also using the fact that the space of deformation retractions is contractible. 
\end{proof}

\subsubsection{The definition of the trace}
\begin{definition}\label{def: Tr}
    Fix  trace data 
    $$ \bar {R} = (e, \rho^{ext}, \zeta, V, \epsilon, \lambda,  F). $$
    Let $(v, w, t) \in \frac{W^{D_{\nu_e}}}{\partial W^{D_{\nu_e}}} \wedge S^1$.  So $t \in [0,1], w \in W$, and $v \in (D\nu_e)_w$. The \emph{ unstable Trace, $Tr_{unst}$,} is the composition of the Thom collapse map: 
    $$\Sigma^L S^1 \to \frac{W^{D_{\nu_e}}}{\partial W^{D_{\nu_e}}} \wedge S^1 $$  
    and the map
    $$\frac{W^{D_{\nu_e}}}{\partial W^{D_{\nu_e}}} \wedge S^1 \to \Sigma^L \frac{\cL W}W $$ 
    defined by:
    \begin{equation}\label{eq: tr unst}
    (v, w, t) \mapsto \begin{cases}
    \begin{pmatrix}
        \lambda \left(v-\phi_1 \circ F_t(w)\right),\\
        B\left(
        w \overset{F|_{[0,t]}}\rightsquigarrow
        F_t(w) \overset{\phi}\rightsquigarrow
        \phi_1 \circ F_t(w) \overset{\theta}\rightsquigarrow
        w
         \right)
    \end{pmatrix}
    & \textrm{ if } \lVert v - \phi_1 \circ F_t(w)\rVert \leq \varepsilon\\
    * & \textrm{ otherwise.}
    \end{cases}\end{equation}
Note that we have used convention (\ref{def: sus 2}) for a model of the target. 
\end{definition}
\begin{remark}
    Unlike the case of the coproduct, the target of $\phi_1$ is $W^{ext}$, hence in order to end up with loops in $W$ we need to use the natural projection $W^{ext} \to W$.  Therefore, in \ref{eq: tr unst} the path \[F_t(w) \overset{\phi}\rightsquigarrow
        \phi_1 \circ F_t(w)\] is understood to be its projection to $W$, and the path $$\phi_1 \circ F_t(w) \overset{\theta}\rightsquigarrow
        w$$
        is the retraction of the straight line path $[v, \pi\circ F_t(w)]$ to $W$. 
\end{remark}
\begin{lemma}\label{lem: tr well def}
    $Tr_{unst}$ is a well-defined continuous map.
\end{lemma}

\begin{proof}
Clearly the collapse map is well defined.  We must check that (\ref{eq: tr unst}) sends $(v, w, t)$ to the basepoint whenever $t \in \{0,1\}$, $|v|=1$ or $w\in \partial W$.  \par 

    Indeed, if $t =0$ and the incidence condition  holds then the second component simplifies to $$B(w \overset{\phi}\rightsquigarrow \phi_1(w) \overset{\theta} \rightsquigarrow w)$$ which is a constant loop since each of the paths has length less than $\frac{\zeta}{4}$ by (\ref{def: trace data hcob}.\ref{it: tr v.f.}) and (\ref{def: trace data hcob}.\ref{it: tr eps small 1}).\par
    
    When $t=1$, $F_1(w) $ is in $M$, and by (\ref{def: trace data hcob}.\ref{it: tr eps small 2 }) the incidence condition can not hold so (\ref{eq: tr unst}) represents the basepoint.\par 

    Similarly, if $w \in \partial W$, then by (\ref{def: trace data hcob}.\ref{it: tr eps small 2 }) and (\ref{def: trace data hcob}.\ref{it: tr eps small 3}) the incidence condition can not hold so (\ref{eq: tr unst}) represents the basepoint. \par 

    Lastly, if $|v|=1$, the first entry in (\ref{eq: tr unst}) lies outside of the cube, by (\ref{def: trace data hcob}.\ref{it: tr lambda big}), so (\ref{eq: tr unst}) represents the basepoint.\par    
    
\end{proof}

\begin{definition}
    The \emph{(stable) Trace}:
    $$Tr: \Sigma^\infty S^1\to \Sigma^{\infty}\frac{\cL W}W$$
    is defined to be the $L$-times desuspension of $$Tr_{unst}:  \Sigma^L S^1 \to \Sigma^L \frac{\LL W}{W}$$ for some trace data $\bar R$. 
\end{definition}

The proof of \cref{lem: cop ind} carries over word by word to give: 
\begin{lemma}
    The stable Trace is well defined and is independent of choices up to homotopy.
\end{lemma}

\subsection{The operations $\Xi_{{{r}}}$ and $\Xi_{{{l}}}$} \label{sec:traces on loop space}

In the previous section we defined the trace map:

 $$Tr: \Sigma^\infty S^1 \to \Sigma^{\infty}\frac{\cL W}W.$$
In this section we will upgrade the construction and define maps: 
$$\Xi_{{{r}}}, \Xi_{{{l}}}: \frac{\cL W^{-TW}}{\partial \cL W^{-TW}} \wedge S^1 \to \Sigma^\infty \frac{\cL W}W \wedge \frac{\cL W}W.$$
It will prove more useful for the following sections to consider the situation of a cobordism with a filling. This is, let $M \subseteq P$ be a codimension 0 submanifold with corners, with $j: M\hookrightarrow P$ an embedding which is a homotopy equivalence, and such that $\partial M$ and $\partial P$ are disjoint. Let $M^\circ = M \setminus \partial M$ be the interior of $M$ and $W = P \setminus M^\circ$, a cobordism from $\partial M$ to $\partial P$. We assume that $W$ is an $h$-cobordism.

Precomposing $j$ by the diffeomorphism $\Phi$ of \cref{def: ext mnfld}, we obtain an embedding $M^{ext} \hookrightarrow P$. 
Note that this  defines a collar neighborhood $ \partial M \times [0,1] \to P$ (in the sense of Remark \ref{remark: emb corn}) by restricting this embedding to $M^{ext} \setminus M^\circ$, and a smooth structure on: 
$$W^{ext} := W\cup_{\partial M}\partial M \times [0,1] \cup_{\partial N} \partial N \times [0,1].$$ 
\begin{definition}\label{def: tr dara loop}   
A choice of \emph{trace data} for $(M,P,j)$ is a pair $(Q,F)$, where $Q \in ED^L(P)$ is {embedding data} for $P$ and $F:P\times I \to P$ is a strong deformation retraction onto $M$. We require that: 
    $$(Q,F)|_W:= (e|_{W^{ext}}, \rho^{ext}|_{D_2\nu|_{W^{ext}}}, \zeta, V|_{W^{ext}}, \eps, \lambda, F)$$
    consists of trace data for $W$.\par 
    We write $TD^L(M \xhookrightarrow{j} P)$ for the simplicial set whose $k$-simplices consist of the set of continuously-varying families of tuples of trace data, parametrised by the standard $k$-simplex. 
 
\end{definition}

\begin{definition}
   Fix a choice of trace data $\Bar R \in TD^L(M \xhookrightarrow{j} P)$. We define  
   $$\Xi_{{{r}},unst}: \frac{\cL P^{D\nu_e}}{\partial \cL P^{D\nu_e}} \wedge S^1 \to \Sigma^L \frac{\cL P}P \wedge \frac{\cL P}P$$ 
   to send $(v, \gamma, s)$ to: 
    
    \begin{equation}\label{eq: tr  XiL}
        \begin{cases}
            \begin{pmatrix}
                \lambda\left(v-\phi_1 \circ F_s \circ \gamma(0)\right), \\
                B\left(
                    \gamma(0) \overset{F|_{[0,s]}}{\rightsquigarrow}
                    F_s \circ \gamma(0) \overset{\phi}{\rightsquigarrow}
                    \phi_1 \circ F_s \circ \gamma(0) \overset{\theta}{\rightsquigarrow}
                    \gamma(0)
                \right),\\
                B\left(
                    \gamma(0) \overset{\theta}{\rightsquigarrow}
                    \phi_1 \circ F_s \circ \gamma(0) \overset{\bar \phi}{\rightsquigarrow}
                    F_s \circ \gamma(0) \overset{\bar F|_{[0,s]}}{\rightsquigarrow}
                    \gamma(0) \overset{\gamma}{\rightsquigarrow}
                    \gamma(0)
                \right)
            \end{pmatrix} &
            \textrm{ if } \lVert v- \phi_1 \circ F_s \circ \gamma(0) \rVert \leq \eps \\
            * & 
            \textrm{otherwise.}
        \end{cases}
    \end{equation}
    and similarly, $$\Xi_{{{l}},unst}: \frac{\cL P^{D\nu_e}}{\partial \cL P^{D\nu_e}} \wedge S^1 \to \Sigma^L \frac{\cL P}P \wedge \frac{\cL P}P$$  sends $(v, \gamma, s)$ to
    \begin{equation}\label{eq: tr  XiR}
        \begin{cases}
            \begin{pmatrix}
                \lambda\left(v-\phi_1 \circ F_s \circ \gamma(0)\right), \\
                B\left(
                    \gamma(0) \overset{\gamma}{\rightsquigarrow} 
                    \gamma(0) \overset{F|_{[0,s]}}{\rightsquigarrow}
                    F_s \circ \gamma(0) \overset{\phi}{\rightsquigarrow}
                    \phi_1 \circ F_s \circ \gamma(0) \overset{\theta}{\rightsquigarrow}
                    \gamma(0)
                \right),\\
                B\left(
                    \gamma(0) \overset{\theta}{\rightsquigarrow}
                    \phi_1 \circ F_s \circ \gamma(0) \overset{\overline \phi}{\rightsquigarrow} 
                    F_s \circ \gamma(0) \overset{\overline F|_{[0,s]}}{\rightsquigarrow}
                    \gamma(0)
                \right)
            \end{pmatrix} &
            \textrm{ if } \lVert v- \phi_1 \circ F_s \circ \gamma(0) \rVert \leq \eps \\
            * & 
            \textrm{otherwise.}
        \end{cases}
    \end{equation}
\end{definition}
\begin{lemma}\label{lem: Xi cont}
    $\Xi_{{{r}},unst}$ and $\Xi_{{{l}},unst}$ are well-defined continuous maps.
\end{lemma}
\begin{proof}
    We prove that (\ref{eq: tr XiL}) sends $(v, \gamma, s)$ to the basepoint if $s \in \{0,1\}$, $\gamma(0) \in \partial P$ or $|v|=1$; the case of (\ref{eq: tr XiR}) is identical.\par 
    If $s = 0$, the second entry in (\ref{eq: tr XiL}) is constant, and so (\ref{eq: tr XiL}) represents the basepoint. \par 
    
    If $s=1$ and $\gamma(0) \in W$ then since $F_1(\gamma(0)) \in M$ by (\ref{def: trace data hcob}.\ref{it: tr eps small 2 }) the incidence condition can not hold. If $\gamma(0) \in M, $ then the second entry of (\ref{eq: tr XiL}) represents the basepoint. 
\par 
    The case of $|v|=1$ and $\gamma(0) \in \partial P$ is the same as in Lemma \ref{lem: tr well def}.
\end{proof}

\begin{definition}
The stable operations:
    $$\Xi_{{{r}}}, \Xi_{{{l}}}: \frac{\LL P^{-TP}}{\partial \LL P^{-TP}} \wedge S^1\to \Sigma^{\infty}\frac{\cL P}P$$
    are defined to be the $L$-times desuspension of $\Xi_{{{r}},unst}$ and $\Xi_{{{l}},unst}$.     
\end{definition}

 By a proof similar to that of Lemma \ref{lem: cop ind}, $\Xi_{{{r}}}$ and $\Xi_{{{l}}} $ are independent of choices.

\subsection{$T$ and $Tr$}
    \label{sec: T Tr}
    
    In this section, we prove Lemma \ref{lem: T vs Tr}, which says that the invariant $[T] $ corresponds to $Tr$ under the Pontrjagin-Thom correspondence. \par 
     Let $j: M \hookrightarrow P$ and $W$ be as in Section \ref{sec:traces on loop space}. We begin by choosing convenient trace data and collars for $W$.
    \subsubsection{Convenient choices}
        We first choose convenient collars for $M$ and $P$. Let:  
        $$\cC_P: \partial P \times [0,1] \to W$$ 
        be a collar neighbourhood of $\partial P$, sending $\partial P \times \{1\}$ to $\partial P$. We write $\cC_P$ also for its image, and $\cC_P^{in}$ for the smaller collar neighbourhood $\cC_P(\partial P \times [\frac 1 2, 1])$. \par
        Similarly, let $\cC_M: \partial M \times [0,1] \to W$ be a collar neighbourhood of $\partial M$, sending $\partial M \times \{0\}$ to $\partial M$. We write $\cC_M$ also for its image; we assume this is disjoint from $\cC_P$.

        \begin{figure}[h] 
            \centering
            \begin{minipage}{.42 \textwidth}
                \begin{tikzpicture}[scale=.8,
                v/.style={draw,shape=circle, fill=black, minimum size=1.3mm, inner sep=0pt, outer sep=0pt},
                vred/.style={draw,shape=circle, fill=red, minimum size=1mm, inner sep=0pt, outer sep=0pt},
                vsmall/.style={draw,shape=circle, fill=black, minimum size=1mm, inner sep=0pt, outer sep=0pt}]
        
                    \draw[color=black] (0,-1) to (4,-1);
                    \draw[color=black] (0,1) to (4,1);
                    \draw[color=black] (0,-1) to (-2, -1);
                    \draw[color=black] (0,1) to (-2, 1);
                    \draw[color=black] (-2, -1) .. controls (-4, -1) and (-4, 1) ..  (-2, 1);
        
                    \draw[color=black] (-2,0.1) to[out=330, in=210] (-1,0.1);
                    \draw[color=black] (-1.75, 0) to [out=30, in=150] (-1.25, 0);
        
                    \draw[color=black] (0,-1) to[in=240, out=120] (0,1);
                    \draw[color=black, dashed] (0,-1) to [in=300, out=60] (0,1);
                    \draw[color=black] (4,-1) to[in=240, out=120] (4,1);
                    \draw[color=black] (4,-1) to [in=300, out=60] (4,1);
        
                    \draw[color=red] (1, -1) to[in=240, out=120] (1,1);
                    \draw[color=red, dashed] (1,-1) to[in=300, out=60] (1,1);
        
                    \draw[color=red] (3, -1) to[in=240, out=120] (3,1);
                    \draw[color=red, dashed] (3,-1) to[in=300, out=60] (3,1);
        
                    \draw[color=blue] (3.5, -1) to[in=240, out=120] (3.5,1);
                    \draw[color=blue, dashed] (3.5,-1) to[in=300, out=60] (3.5,1);
        
                    \draw[color=red, <->] (-0.3, 0) to (0.7, 0);
                    \draw[color=red, <->] (2.75, -0.3) to (3.75, -0.3);
        
                    \draw[color=blue, <->] (3.25, 0.3) to (3.75, 0.3);
        
                    \draw[color=black, <->] (-3.5, -1.2) to (0, -1.2);
                    \draw[color=black, <->] (0,-1.2) to (4, -1.2);
                    \draw[color=black, <->] (-3.5, 1.2) to (4, 1.2);
        
                    \node at (-1.5, -1.5) {\tiny $M$};
                    \node at (2, -1.5) {\tiny $W$};
                    \node at (0, 1.5) {\tiny $P$};
        
                    \node at (0.2, -0.3) {\tiny ${\color{red}\cC_M}$};
                    \node at (3.2, -0.6) {\tiny ${\color{red}\cC_P}$};
        
                    \node at (3.6, 0.6) {\tiny ${\color{blue} \cC_P^{in}}$};

                \end{tikzpicture}
            \end{minipage}
            \caption{Collars.}
            \label{fig:collar-pic}
        \end{figure}
        We also choose trace data $(Q, F) \in TD^L(M \xhookrightarrow{j} P)$, such that $F$ satisfies the conclusion of Lemma \ref{lem: gen pert F}.

        By choosing the collars to be sufficiently small, we may assume
        \begin{equation}\label{eq: coll small}
            i(T) \textrm{ is disjoint from } (\cC_M \cup \cC_P) \times [0,1]
        \end{equation}
        
        By using an appropriate cut-off function, we may also assume
        \begin{equation}\label{eq: V cond}
            V=0 \textrm{ on } P \setminus (M \cup \cC_M \cup \cC_P^{in}).
        \end{equation}

    \subsubsection{Pontrjagin-Thom data}
        We consider Pontrjagin-Thom data (see Appendix \ref{sec: PT}) for the manifolds $P$ and $T$, (defined in \cref{eq: T}), as follows, using the trace data chosen above.\par 
        For $P$, we take the embedding $e: P \hookrightarrow \bR^L$. By rescaling if necessary, we may assume the image of $e$ lies in $(-1, 1)^L$. Since this is a codimension 0 embedding, no extra data is required.\par 
        For $T$, we take 
        \begin{itemize} 
            \item The embedding to be  
            \begin{equation} 
                T \xhookrightarrow{i} P \times [0,1) \xhookrightarrow{e \times Id} (-1,1)^L \times (-1,1)
            \end{equation}
            noting that by Lemma \ref{lem: gen pert F}, the image of $i$ does not hit $P \times \{1\}$.
            \item  $\psi_\mu: \nu_i \to \R^L$ is the isomorphism of vector bundles sending $(v,t)$ in the fibre of $\nu_i$ over $(x,s)$ to 
            \begin{equation}   \label{eq: psimu}
                \psi(v,t) :=\mu( v- dF_{(x,s)}(v,t)),
            \end{equation} 
            where $\mu>0$ is large.
            
            \item $\sigma_\chi: D\nu_i \to P \times [0,1]$ to send $(v,t)$, lying in the fibre of $D\nu_i$ over $(x,s) \in T$, to $(x,s) + \chi \cdot (v,t)$, where $\chi > 0$ is a small fixed number. 
        \end{itemize}
        
        \begin{lemma}\label{lem: chi mu}
            For $\chi > 0$ sufficiently small, $\sigma_\chi$ is an embedding, with image lying outside of $(\cC_M \cup \cC_P) \times [0,1]$.\par 
            For $\chi > 0$ fixed and $\mu > 0$ sufficiently large, $\psi_\mu$ satisfies (\ref{eq: PT fram cond}).
        \end{lemma}
        \begin{proof}
        
            The first statement follows from the inverse function theorem and the fact that $i(T)$ lies outside $(\cC_M \cup \cC_P) \times [0,1]$, by (\ref{eq: coll small}). The second statement is clear. 
        \end{proof}
    \subsubsection{Proof of Lemma \ref{lem: T vs Tr}}
        For the rest of the section, we fix $\chi, \mu>0$ as in Lemma \ref{lem: chi mu}. We assume the map $[T]_{unst}$ is taken with respect to this choice of data.
        \begin{lemma}\label{lem: change inc cond}
            For $\lambda > 0$ large enough, if $Tr(\gamma,s)$ is not the basepoint, then $(\gamma, s) \in \sigma_\chi(D\nu_i)$.
        \end{lemma}
        \begin{proof}
            Let $S = \{(x, s) \in P \times [0,1]\,|\, \lVert x-F_s(x)\rVert \leq \eps \} \setminus \sigma_\chi(D\nu_i^\circ)$. Since $S$ is compact, for $\lambda > 0$ large enough, whenever $(\gamma(0), s)$ doesn't lie in $S$, the first term of (\ref{eq: tr unst}) has large norm.
        \end{proof}
        Choosing $\lambda >0$ large enough that Lemma \ref{lem: change inc cond} holds and using \cref{eq: V cond} and Lemma \ref{lem: chi mu}, we have:
        \begin{equation}\label{eq: simp tr}
            Tr_{unst}(x, s) = 
            \begin{cases}
                \begin{pmatrix}
                    \lambda(x-F_s(x)) \\
                    B\left(x \overset{F|_{[0,s]}} \rightsquigarrow
                    F_s(x) \overset \theta \rightsquigarrow
                    x\right)
                \end{pmatrix}
                & \textrm{ if } (x, s) \in \sigma_\chi(D\nu_i)\\
                * & \textrm{ otherwise.}
            \end{cases}
        \end{equation}
        Using the chosen Pontrjagin-Thom data for $T$ (and assuming that $\lambda = \mu/\chi$, which we can do by increasing $\lambda$ or $\mu$ as necessary) and opening up the definition of $\psi_\mu$, we have that
        \begin{equation}\label{eq: simp T}
            [T]_{unst}(x, s) = \begin{cases}
                \begin{pmatrix}
                    \lambda ((x-y) - dF_{(y,t)} (x-y,s-t)) \\
                    y \overset{F|_{[0,s]}} \rightsquigarrow y
                \end{pmatrix}
                & \textrm{ if } (x, s) \in \sigma_\chi(D\nu_i)\\
                * & \textrm{ otherwise.}
            \end{cases}
        \end{equation}
        here $(y,t)$ is in $  T$, and is such that $\sigma^{-1}_\chi(x,s)$ lives in the fibre of $D\nu_i \to T$ over $(y,t)$. This is well defined when the incidence condition holds.

        \begin{lemma}\label{lem: T vs Tr}
            $Tr$ and $[T]$ are homotopic.
        \end{lemma}
        \begin{proof}
         We construct a homotopy between (\ref{eq: simp T}) and (\ref{eq: simp tr}) as follows.

            The first entries in (\ref{eq: simp T}) and (\ref{eq: simp tr}) agree when $(x, s) \in T$, or when $(x, s) \notin \sigma_\chi(D\nu_i)$. As $(x, s)$ leaves $T$, the derivative of the first entry of (\ref{eq: simp tr}) in the normal direction is exactly the first entry of (\ref{eq: simp T}). Therefore we may define a homotopy by linearly interpolating between them. By the implicit function theorem, this is well-defined for $(x, s)$ in some small neighbourhood of $T$; by picking $\lambda$ sufficiently large, this neighbourhood will include the region where the two expressions are not sent to the base-point, so this homotopy will be defined everywhere.

            In the second entry we can take a homotopy of the form $\{z_\tau(x,y) \overset{F|_{[0,s]}}\rightsquigarrow F_s(z_\tau(x,y)) \overset \theta \rightsquigarrow z_\tau(x,y)\}_\tau$, where $\{z_\tau(x,y)\}_\tau$ follows the straight line between $x$ and $y$, and also applying Lemma \ref{lem: B}.

        \end{proof}

\section{Codimension 0 coproduct defect}\label{sec: cop defect}

Let $j: M \subseteq P$ be a codimension $0$ embedding such that the complement $W := P \setminus M^\circ$ is an $h$-cobordism (in particular, $j$ is a homotopy equivalence), possibly with corners. Let $F: W \times I \to W$ be a strong deformation retraction onto $\partial M$, which we extend by the identity on $M$ to a strong deformation retraction $F: P \times I \to P$.

Let $[T]$ be the associated framed bordism invariant, defined as in Section \ref{sec: framed bordism invariant}. 
In this section we compare the coproducts on $M$ and $P$ and relate the difference to the diagonal Chas-Sullivan product with $[T]$. We do so by first relating the difference to the operations $\Xi_{{{r}}}$ and $\Xi_{{{l}}}$ in Section \ref{sec: cop def xi} (\cref{thm: cop Xi}), and then relating $\Xi_{{{r}}}$ and $\Xi_{{{l}}}$ to the diagonal Chas-Sullivan product with $[T]$ in Section \ref{sec: tr is product} (\cref{thm: T Xi}).
\subsection {Coproduct defect is given by $\Xi_{{{l}}}-\Xi_{{{r}}}$} \label{sec: cop def xi}

For the rest of this section fix trace data $ (Q,F) \in TD^L(M \xhookrightarrow{j} P)$ as in Definition \ref{def: trace data hcob}. We assume that $j$ extends to an embedding  $j^{ext}: M^{ext}\hookrightarrow P$  such that $j^{ext}(M^{ext})$ and $\partial P$ are disjoint.  We require that $Q \in ED^L(P)$ is a choice of {embedding data}, such that
        \begin{equation}\label{eq: Q restr Q}
            \textrm{$Q|_M := (e|_{M^{ext}}, \rho^{ext}|_{D_2\nu|_{M^{ext}}}, \zeta, V|_M, \eps, \lambda)$ consists of {embedding data} for $M$}.
        \end{equation} For convenience, we write $\nu_P$ for $\nu_e|_P$, and similarly for $\nu_M$.
Let $F:P\times I \to P$ be the strong deformation retraction. Then $F_1$ induces a map of spaces: 
$$\overline F_1: \frac{\cL P^{D\nu_P}}{\partial \cL P^{D\nu_P}} \to \frac{\cL M^{D\nu_M}}{\partial \cL M^{D\nu_M}}$$
by sending 
\begin{equation}\label{eq: overline F1}
   \overline F_1(v, \gamma)= 
\begin{cases}
    (v, F_1 \circ \gamma) & \textit{ if  } \gamma(0) \in M\\
    * & \textit{ otherwise}
\end{cases}
\end{equation} 
By passing to spectra, we get a map that we also call $\overline F_1$: 
 $$\overline F_1:\frac{\cL P^{-TP}}{\partial \cL P^{-TP}}\to \frac{\cL M^{-TM}}{\partial \cL M^{-TM}}$$
\begin{lemma}
    $\overline F_1$ is an equivalence of spectra.
\end{lemma}
\begin{proof}
    We prove this at the level of spaces. We define an explicit homotopy inverse
    \begin{equation*}
        G: \frac{\cL M^{D\nu_M}}{\partial \cL M^{D\nu_M}} \to \frac{ \cL P^{D\nu_P}}{\partial \cL P^{D\nu_P}}
    \end{equation*}
    as follows. Choose a collar neighbourhood $C: \partial M \times I \to M$ sending $\partial M \times \{1\}$ to $\partial M$; by abuse of notation we write $C$ also for the image of the map. Choose a map $g_1: \partial M \times I \to W \cup C$ which is given by $C|_{\partial M\times \{0\}}$ on $\partial M \times \{0\}$ and sends $\partial M \times \{1\}$ to $\partial P$, along with a homotopy $\{g_t\}_{t\in[0,1]}$ from $g_0=C$ to $g_1$ relative to $\partial M \times \{0\}$. This exists since $P \setminus M^\circ$ is an $h$-cobordism: $g_1$ is a homotopy inverse (of pairs) to $F_1$, and $\{g_t\}$ is a homotopy between $g_1$ and the inclusion of the collar.
    Now define
    \begin{equation*}
        G(v, \gamma) = \begin{cases}
            (v, \gamma) & \textrm{ if } \gamma(0) \in M \setminus C\\
            \left( \tilde v, g_1(x,t) \overset{\overline{g(x,t)}}\rightsquigarrow
            g_0(x,t)=\gamma(0) \overset\gamma\rightsquigarrow
            \gamma(0) \overset{g(x,t)}\rightsquigarrow
            g_1(x,t)\right)&
            \textrm{ if }\gamma(0) = C(x,t)
        \end{cases}
    \end{equation*}
    where $\tilde v$ is given by parallel transporting $v$ along the path $\{g_\tau(x,t)\}_{\tau \in [0,1]}$.\par 
    
    We show by explicit construction of a homotopy that $G \circ \overline F_1 \simeq Id_P$.
    We do this by concatenating two homotopies
    \begin{equation*}
        H,H': \frac{\cL P^{D\nu_P}}{\partial \cL P^{D\nu_P}} \times [0,1] \to \frac{\cL P^{D\nu_P}}{\partial \cL P^{D\nu_P}}
    \end{equation*}
    For $s \in [0,1]$, we define $H_s(v, \gamma)$ to be
    \begin{equation*}
        \begin{cases} 
            (v, F_s \circ \gamma) & \textrm{ if } \gamma(0) \in M \setminus C\\
            \left(\tilde v, g_1(x, t) \overset{\overline g}{\rightsquigarrow} 
            g_0(x, t) \overset{F_s \circ \gamma}{\rightsquigarrow}
            g_0(x, t) \overset{g}{\rightsquigarrow} 
            g_1(x,t)\right) &
            \textrm{ if } \gamma(0) \in W \cup C
        \end{cases}
    \end{equation*} 
    where in the second case, $(x, t) \in \partial M \times I$ is such that $F_1(\gamma(0)) = C(x, t)$, and $\tilde v$ is defined as above.
    
    We choose a map $\delta: (W \cup C) \times [0,1]_\tau \times [0,1]_t \to W \cup C$, which we think of as a family of paths $\{\delta^y_\tau\}_{\tau \in [0,1], y \in W\cup C}$, such that:
    \begin{itemize}
        \item $\delta(y, \tau, 0) = y$ for all $y, \tau$.
        \item $\delta(y, 1, t) = y$ for all $y, t$.
        \item $\delta(y, 0, \cdot)$ is the path
        $$y \overset{F}\rightsquigarrow F_1(y) = C(x, t) \overset g \rightsquigarrow g_1(x, t)$$
        for all $y$, where $(x, t) \in \partial M \times [0,1]$ is determined by $F_1(y) = C(x, t)$.
        \item $\delta(y, \tau, t) = y$ for all $y \in C(\partial M \times \{0\})$.
        \item 
        $\delta(y, \tau, 1) \in \partial P$ for all $y \in \partial P$ and all $\tau$.
    \end{itemize}
    
    Together, all the constraints except the last one specify $\delta$ totally over a subspace which $W \times [0,1]^2$ deformation retracts to: explicitly, this subspace is $\left[(W \cup C) \times \sqcup\right] \cup \left[C(\partial M \times \{0\}) \times [0,1]^2\right]$, where $\sqcup := \{t=0\} \sqcup \{\tau \in \{0,1\}\} \subseteq [0,1]^2$ consists of three sides of the boundary of the square. The final constraint says that $\delta$ sends $(W \cup C) \times \{ t=1\}$ to $\partial P$. Since the codomain $W \cup C$ deformation retracts to $\partial P$, we may choose $\delta$ satisfying the first constraints, and then homotope it (relative to this subspace) so that it satisfies the final constraint too.
    
    We define $H'_s(v, \gamma)$ to be
    \begin{equation}
        \begin{cases}
            (v, \gamma) & \textrm{ if } \gamma(0) \in M \setminus C\\
            \left(\tilde v_s, \delta^{\gamma(0)}_s(1) \overset{\overline \delta^{\gamma(0)}_s}\rightsquigarrow 
            \gamma(0) \overset \gamma \rightsquigarrow
            \gamma(0) \overset{\delta^{\gamma(0)}_s} \rightsquigarrow
            \delta^{\gamma(0)}_s(1)\right) & \textrm{ otherwise.}
        \end{cases}
    \end{equation}
    where $\tilde v_s$ denotes $v$ parallel transported along the path $\delta^{\gamma(0)}_s$.\par 
    Then $H_1 = G\circ \overline F_1$, $H_0 = H'_1$ and $H'_0$ is the identity.

    The other direction (that $\overline F_1 \circ G$ is homotopic to the identity) is similar but simpler, as it does not require the $\delta$-paths. We may choose $\{g_t\}$ such that for $t \leq 1/2$, it follows a straight line, hitting $\partial P$ when $t=1/2$: explicitly, $g_t(x, s) = (x, st + (1-t)2s)$ for $t \leq 1/2$, and $g_1(x,t) \in W$ whenever $t \geq 1/2$. Then $\overline F_1 \circ G$ sends $(v, \gamma)$ to:
    \begin{equation}
        \begin{cases}
            (v, \gamma)
            &
            \textrm{ if } \gamma(0) \in M \setminus C
            \\
            \left(\tilde v, g_1(x, t) \overset{\overline g}{\rightsquigarrow} 
            g_0(x, t) \overset{\gamma}{\rightsquigarrow}
            g_0(x, t) \overset{g}{\rightsquigarrow} 
            g_1(x,t)\right)
            &
            \textrm{ if } \gamma(0) = C(x, t) \textrm{ and } t \leq \frac12 \\
            *
            &
            \textrm{ otherwise.}
        \end{cases}
    \end{equation}
    One can construct an explicit homotopy to the identity by retracting the $g$-paths.
\end{proof}

The main result of this section is that $\Xi_{{{l}}}$ and $\Xi_{{{r}}}$ together determine the failure for the coproducts for $M$ and $P$ to agree:
\begin{theoremaaa}\label{thm: cop Xi}
    There is a homotopy
    \begin{equation}
        \Delta^P - (j\wedge j) \circ \Delta^M \circ \overline F_1 \simeq  \Xi_{{{l}}} -  \Xi_{{{r}}}
    \end{equation}
    between maps of spectra $\frac{\cL P^{-TP}}{\partial \cL P^{-TP}} \wedge S^1 \to \Sigma^\infty \frac{\cL P}P \wedge \frac{\cL P}P$. \par

\end{theoremaaa}

To prove \cref{thm: cop Xi}, we start by defining a map $\Lambda$ whose boundary will give rise to the required homotopy. 
\begin{definition}
  For the fixed choice of $(Q,F)$, we define a map of spaces:
    $$\Lambda: \frac{\cL P^{D\nu}}{\partial \cL P^{D\nu}} \times [0,1]^2_{s,t} \to \Sigma^L \frac{\cL P}P \wedge \frac{\cL P}P$$
  which sends $(v, \gamma, s, t)$ to
    \begin{equation}\label{eq: Lambda}
        \begin{cases}
            \begin{pmatrix}
                \lambda\left(v-\phi_1\circ F_s \circ \gamma(t)\right),\\
                B\left( \gamma(0) \overset{F|_{[0,s]}}\rightsquigarrow
                F_s \circ \gamma(0) \overset{F_s \circ \gamma|_{[0,t]}}\rightsquigarrow
                F_s \circ \gamma(t) \overset{\phi}\rightsquigarrow
                \phi_1 \circ F_s \circ \gamma(t) \overset{\theta}\rightsquigarrow
                \gamma(0)
                \right),\\
                B\left(\gamma(0) \overset{\theta}\rightsquigarrow
                \phi_1 \circ F_s \circ \gamma(t) \overset{\overline \phi}\rightsquigarrow
                F_s \circ \gamma(t) \overset{F_s \circ \gamma|_{[t,1]}}\rightsquigarrow
                F_s \circ \gamma(1) \overset{\overline F|_{[0,s]}}\rightsquigarrow
                \gamma(1) \right)
            \end{pmatrix}
            & \textrm{ if } \lVert v-\phi_1\circ F_s \circ \gamma(t) \rVert \leq \varepsilon\\
            * & \textrm{ otherwise.}
        \end{cases}
    \end{equation}
\end{definition}

\begin{figure}[h] 
    \centering
    \begin{minipage}{.25 \textwidth}
        \begin{tikzpicture}[scale=.8,
        v/.style={draw,shape=circle, fill=black, minimum size=1.3mm, inner sep=0pt, outer sep=0pt},
        vred/.style={draw,shape=circle, fill=red, minimum size=1mm, inner sep=0pt, outer sep=0pt},
        vsmall/.style={draw,shape=circle, fill=black, minimum size=1mm, inner sep=0pt, outer sep=0pt}]
            \draw[color=black, ->] (0,0) to[out=90, in=180] (1,2);
            \draw[color=black] (1,2) to[out=0, in=105] (1.8, 1);
            \draw[color=black] (1.8, 1) to[out=285, in=90] (2, 0);

            \draw[color=black] (2, 0) to[out=270, in=0] (1, -2);
            \draw[color=black] (1, -2) to[out=180, in=270] (0,0);

            \draw[thin, color=blue, ->] (1.8, 1) to[out=135, in=0] (1.4, 1.4);
            \draw[thin, color=blue] (1.4, 1.4) to[out=180, in=90] (1,-0.6);
            
            \draw[thin, color=orange] (1, -0.6) to (0.5,0.5);
            \draw[thin, color=orange, ->] (1, -0.6) to (0.75,-0.05);

            \draw[thin, color=purple] (0.5, 0.5) to (0,0);
            \draw[thin, color=purple, ->] (0.5, 0.5) to (0.25, 0.25);

            \draw[thick, teal, dashed] (0.9,0.4) arc (0:360: 1.3);

            \draw[thin, color=cyan, ->] (0, 0) to (-0.2, 0.2);
            \draw[thin, color=cyan] (-0.2, 0.2) to (-0.4, 0.4);

            \node[v] at (1.8,1) {};
            \node[v] at (1, -0.6) {};
            \node[v] at (0.5, 0.5) {};
            \node[v] at (0,0) {};
            \node[v] at (-0.4, 0.4) {};
            \node[vred] at (0,0) {};

            \node at (-1.7, 1.7) {\tiny ${\color{teal} B_\eps(v)}$};
            \node at (0.2, 0.5) {\tiny ${\color{purple} \theta}$};
            \node at (0.5, -0.2) {\tiny ${\color{orange} \phi}$};
            \node at (0.9, 1.6) {\tiny ${\color{blue} F|_{[0,s]}}$};
            \node at (0.3, 1.9) {\tiny $\gamma$};
            \node at (-0.2, -0.2) {\tiny $\gamma(0)$};
            \node at (-0.6, 0.6) {\tiny ${\color{cyan} v}$};
            \node at (2.2, 1) {\tiny $\gamma(t)$};

            \node at (1, 2.5) {$(v,\gamma, s,t)$};

        \end{tikzpicture}
    \end{minipage}
    $\mapsto B\left(
    \begin{minipage}{.4 \textwidth}
        \begin{tikzpicture}[scale=.8,
        v/.style={draw,shape=circle, fill=black, minimum size=1.3mm, inner sep=0pt, outer sep=0pt},
        vred/.style={draw,shape=circle, fill=red, minimum size=1mm, inner sep=0pt, outer sep=0pt},
        vsmall/.style={draw,shape=circle, fill=black, minimum size=1mm, inner sep=0pt, outer sep=0pt}] 
            \begin{scope}[shift={(0,0)}]
                \node at (0,0) {$\lambda \cdot ($};
                \node at (1.4, 0) {$)$};
                \node at (1.8, -0.2) {,};
                    
                \draw[thin, color=purple] (1.3, 0.25) to (0.8,-0.25);
                \draw[thin, color=purple, ->] (1.3, 0.25) to (1.05, 0);
                \draw[thin, color=cyan] (0.8, -0.25) to (0.4, 0.15);
                \draw[thin, color=cyan, ->] (0.8, -0.25) to (0.6, -0.05);
            \end{scope}
            \begin{scope}[shift={(2.2,0)}]
                \draw[color=black, ->] (0,0) to[out=90, in=180] (1,2);
                \draw[color=black] (1,2) to[out=0, in=105] (1.8, 1);
    
                \draw[thin, color=blue, ->] (1.8, 1) to[out=135, in=0] (1.4, 1.4);
                \draw[thin, color=blue] (1.4, 1.4) to[out=180, in=90] (1,-0.6);
    
                \draw[thin, color=orange] (1, -0.6) to (0.5,0.5);
                \draw[thin, color=orange, ->] (1, -0.6) to (0.75,-0.05);
    
                \draw[thin, color=purple] (0.5, 0.5) to (0,0);
                \draw[thin, color=purple, ->] (0.5, 0.5) to (0.25, 0.25);
    
                \node[v] at (1.8,1) {};
                \node[v] at (1, -0.6) {};
                \node[v] at (0.5, 0.5) {};
                \node[v] at (0,0) {};
                \node[vred] at (0,0) {};

                \node at (1.8, -0.2) {,};
            \end{scope}

            \begin{scope}[shift={(4.5,0)}]
                \draw[color=black] (1.8, 1) to[out=285, in=90] (2, 0);

                \draw[color=black, ->] (2, 0) to[out=270, in=0] (1, -2);
                \draw[color=black] (1, -2) to[out=180, in=270] (0,0);

                \draw[thin, color=orange] (1, -0.6) to (0.5,0.5);
                \draw[thin, color=orange, -<] (1, -0.6) to (0.75,-0.05);
    
                \draw[thin, color=purple] (0.5, 0.5) to (0,0);
                \draw[thin, color=purple, -<] (0.5, 0.5) to (0.25, 0.25);

                \draw[thin, color=blue, -<] (1.8, 1) to[out=135, in=0] (1.4, 1.4);
                \draw[thin, color=blue] (1.4, 1.4) to[out=180, in=90] (1,-0.6);

                \node[v] at (1.8,1) {};
                \node[v] at (1, -0.6) {};
                \node[v] at (0.5, 0.5) {};
                \node[v] at (0,0) {};
                \node[vred] at (0,0) {};
            \end{scope}
            
        \end{tikzpicture}
    \end{minipage}
    \right)$
    \caption{The operation $\Lambda$: the figure on the left shows a tuple $(v,\gamma, s,t)$ in the domain of $\Lambda$, the one on the right shows the output.}
    \label{fig:Lambda}
\end{figure}

\begin{lemma}\label{lem: Lambda well def}
    $\Lambda$ is well-defined. 
    Furthermore if both $s, t \in \{0,1\}$, then $\Lambda$ sends $(v, \gamma, s,t)$ to the basepoint. 
\end{lemma}
\begin{proof}
   
    For (\ref{eq: Lambda}) to be well-defined, it must send $(v,\gamma,s,t)$ to the basepoint whenever $|v|=1$ or $\gamma(0) \in \partial P$; this holds by the same argument as in Lemma \ref{lem: Delta well def}.\par 
    Suppose $s=0$ and $t=0$ (or $1$). Then if the incidence condition holds, the second (or third, respectively) entry in (\ref{eq: Lambda}) must be constant, by (\ref{def: cop dat}.\ref{item: V small}) and (\ref{def: cop dat}.\ref{item: eps small 3}).\par 
    Suppose $s=1$. Then $F_s\circ \gamma(t) \in M$. If $\gamma(0) \in W$, then by  (\ref{def: trace data hcob}.\ref{it: tr eps small 3}) the incidence condition can not hold.  If $\gamma(0) \in M$ then since $F|_M$ is the identity, the paths $F|_{[0,s]}$ and $\overline F|_{[0,s]}$ appearing in (\ref{eq: Lambda}) are constant. Then if $t=0$, the second entry of (\ref{eq: Lambda}) is constant by the same argument as in Lemma \ref{lem: Delta well def}; similarly if $t=1$ the third entry of (\ref{eq: Lambda}) is constant.
\end{proof}

We next analyse the restriction of $\Lambda$ to each of the four sides of the square $[0,1]^2_{s,t}$. The restriction of $\Lambda$ to the subspace $s=0$ is denoted by:   
$$\Lambda|_{\{s=0\}} := \Lambda|_{(v, \gamma, 0,t)} : \frac{\cL P^{D\nu}}{\partial \cL P^{D\nu}} \times [0,1]_{t} \to \Sigma^L \frac{\cL P}P \wedge \frac{\cL P}P $$
The other sides of the square are denoted in a similar manner.  By Lemma \ref{lem: Lambda well def}, $\Lambda|_{\{s=0\}}$, as well as the restriction of $\Lambda$ to the other sides of the square, descend to maps from $\frac{\cL P^{D\nu}}{\partial \cL P^{D\nu}} \wedge S^1$.

\begin{lemma}\label{lem: Lambda s=0}
    $\Lambda|_{\{s=0\}} = \Delta^P$.
\end{lemma}
\begin{proof}
    Since $F_0$ is the identity on $P$, this follows by comparing (\ref{eq: Delta unst}) and (\ref{eq: Lambda}).
\end{proof}

\begin{lemma}\label{lem: Lambda t=0,1}
    There is a homotopy $\Lambda|_{\{t=0\}} \simeq \Xi_{{{l}},unst}$, relative to the subspace $\{s \in \{0,1\}, t=0\}$.\par 
    Similarly there is a homotopy $\Lambda|_{\{t=1\}} \simeq \Xi_{{{r}},unst}$, relative to the subspace $\{s \in \{0,1\}, t=1\}$.
\end{lemma}
\begin{proof}
    We first construct the homotopy $\Lambda|_{\{t=0\}} \simeq \Xi_{{{r}}}$. We define a homotopy 
    $$H: [0,1]_\tau\times  \frac{\cL P^{D\nu}}{\partial \cL P^{D\nu}} \wedge S^1 \to \Sigma^L \frac{\cL P} P \wedge \frac{\cL P} P$$
    by 
    \begin{equation} \label{eq: Lambda s0t0 hom}
    H_\tau(\gamma,s)=
    \begin{cases}
        (\lambda\left (v-\phi_1 \circ F_s \circ \gamma(0)\right), \alpha_{s,\gamma,\tau}, \beta_{s,\gamma,\tau})
        & \textrm{ if } \lVert v-\phi_1 \circ F_s \circ \gamma(0)\rVert \leq \varepsilon\\
        * & \textrm{ otherwise.}
    \end{cases}
    \end{equation}
    where \begin{align*}
        & \alpha_{s,\gamma,\tau} = B\left(
            \gamma(0) \overset{F|_{[0, \tau s]}}\rightsquigarrow
            F_{\tau s} \circ \gamma(0) \overset{F_{\tau s} \circ \gamma}\rightsquigarrow
            F_{\tau s} \circ \gamma(0) \overset{F|_{[\tau s, s]}}\rightsquigarrow
            F_s \circ \gamma(0) \overset{\phi}\rightsquigarrow
            \phi_1 \circ F_s \circ \gamma(0) \overset{\theta}\rightsquigarrow
            \gamma(0)
            \right) \\
            &\beta_{s,\gamma,\tau}=B\left(
            \gamma(0) \overset{\theta}\rightsquigarrow
            \phi_1 \circ F_s \circ \gamma(0) \overset{\overline \phi}\rightsquigarrow
            F_s \circ \gamma(0) \overset{\overline F|_{[0,s]}}\rightsquigarrow
            \gamma(0)
            \right)\\
            \end{align*}
    This is well-defined by the same argument as in Lemma \ref{lem: Lambda well def}. Inspection of (\ref{eq: Lambda}), (\ref{eq: tr  XiL}) and (\ref{eq: Lambda s0t0 hom}) shows that $H_0 = \Xi_{{{l}},unst}$ and $H_1 = \Lambda|_{\{t=0\}}$, so $H$ is the required homotopy.\par 
    
    The other case is similar; explicitly, a homotopy   
    $$H': [0,1]_{\tau} \times  \frac{\cL P^{D\nu}}{\partial \cL P^{D\nu}} \wedge S^1 \to \Sigma^L \frac{\cL P} P \wedge \frac{\cL P} P$$
    between $\Lambda|_{\{t=1\}}$ and $\Xi_{{{r}},unst}$ is given by
    \begin{equation} 
    H'_\tau(\gamma,s)=
    \begin{cases}
        (\lambda\left (v-\phi_1 \circ F_s \circ \gamma(0)\right), \tilde\alpha_{s,\gamma,\tau}, \tilde\beta_{s,\gamma,\tau})
        & \textrm{ if } \lVert v-\phi_1 \circ F_s \circ \gamma(0)\rVert \leq \varepsilon\\
        * & \textrm{ otherwise.}
    \end{cases}
    \end{equation}
    where \begin{align*}
        &\tilde \alpha_{s,\gamma,\tau} =  B\left(
            \gamma(0) \overset{F|_{[0, s]}}\rightsquigarrow
            F_s \circ \gamma(0) \overset{\phi}\rightsquigarrow
            \phi_1 \circ F_s \circ \gamma(0) \overset{\theta}\rightsquigarrow
            \gamma(0)
            \right)\\
            &\tilde \beta_{s,\gamma,\tau} = B\left(
            \gamma(0) \overset{\theta}\rightsquigarrow
            \phi_1 \circ F_s \circ \gamma(0) \overset{\overline \phi}\rightsquigarrow
            F_s \circ \gamma(0) \overset{\overline F|_{[\tau s,s]}}\rightsquigarrow
            F_{\tau s}\circ \gamma(0) \overset{F_{\tau s} \circ \gamma }\rightsquigarrow
            F_{\tau s} \circ \gamma(0) \overset{ \overline F|_{[0,\tau s]}} \rightsquigarrow
            \gamma(0)
            \right)
    \end{align*}
  \end{proof}

Lastly, we prove the following:

\begin{lemma}\label{lem: Lambda s=1}
 $\Lambda|_{\{s=1\}} = (j \wedge j) \circ \Delta^M \circ \overline F_1$.
\end{lemma}

\begin{proof}
    Note that $F_1 \circ \gamma(t) \in M$. Hence, if $\gamma(0) \in W$, by (\ref{def: trace data hcob}.\ref{it: tr eps small 2 }), the incidence condition can not hold. If $\gamma(0)\in M$ then $\overline F_1(v, \gamma) = (v, \gamma)$, and by our choice \cref{eq: Q restr Q}, the equality holds on the nose. 
\end{proof}
    
\begin{proof}[Proof of Theorem \ref{thm: cop Xi}]
    Passing to suspension spectra (and desuspending $L$ times), Theorem \ref{thm: cop Xi} follows from Lemmas \ref{lem: Lambda s=0}, \ref{lem: Lambda s=1} and \ref{lem: Lambda t=0,1}, by using the homotopy $\Lambda$.
\end{proof}

\subsection{Characterizing $\Xi_{{{r}}}$ and $\Xi_{{{l}}}$ }\label{sec: tr is product}
In this section we characterize the operations $\Xi_{{{r}}}$ and $\Xi_{{{l}}}$ in terms of known invariants and operations. In particular, we will show in \cref{thm: T Xi} that $\Xi_{{{r}}}$ and $\Xi_{{{l}}}$ are homotopic to certain operations, denoted $\mu_{{{l}}}\left((\cdot \times [P]), [T_{diag}]\right)$ and $\mu_{{{r}}}\left([\overline T_{diag}], [P] \times \cdot\right)$, which are determined by $[T]$, the fundamental class of $P$, and the Chas-Sullivan product.

We start by introducing slight variations of $[T]$ and $Tr$ and recalling the definition of the fundamental class:

\begin{definition}\label{def: Tdiag}
    Let \[[T_{diag}], [\overline T_{diag}] \in \Omega^{fr}_1(\cL (W \times W), (W \times W)) \] be the images of $[T]$ under the antidiagonal maps sending $\gamma$ to $(\gamma, \overline \gamma)$ and $(\overline \gamma, \gamma)$ respectively.
\end{definition}

Similarly:
\begin{definition}\label{def: Tr diag}
    Let  \[Tr_{diag}, \overline {Tr}_{diag}: \Sigma^\infty S^1 \to \Sigma^\infty \frac{\cL (W \times W)}{W\times W}\]  be given by the map $Tr$, (see \cref{def: Tr}), composed with the antidiagonals $\frac{\cL W}W \to \frac{\cL (W \times W)}{W\times W}$ sending $\gamma$ to $(\gamma, \overline \gamma)$ and $(\overline \gamma, \gamma)$ respectively.
\end{definition}
The definition of the fundamental class of $P$ represented by constant loops in $\frac{\cL P^{-TP}}{\partial \cL P^{-TP}}$ is then: 
\begin{definition}
   Let $e: P \to \R^L$ be a codimension $0$ embedding. Then

   $$[P]: \bS \to \frac{\cL P^{-TP}}{\partial \cL P^{-TP}}$$ is given by the composition: \begin{equation} 
            \mathbb{S} \to \frac{P^{-TP}}{\partial P^{-TP}} \to \frac{\cL P^{-TP}}{\partial \cL P^{-TP}}
        \end{equation}
        where the first arrow is the fundamental class (as in Appendix \ref{sec: PT}), and the last arrow is given by inclusion of constant loops. 
\end{definition}

We will now use $[P]$, $Tr_{diag}$, and the version of the Chas Sullivan products $\tilde \mu^{P \times P}_{{{r}}}$ and $\tilde \mu^{P \times P}_{{{l}}}$ on $P \times P$,  considered in \cref{eq: split prod}, to define:

\begin{definition}\label{def: compact not}
   Let  $\mu_{{{l}}}\left((\cdot \times [P]), [T_{diag}]\right)$ denote the composition: 

   \begin{multline*}
        \frac{\cL P^{-TP}}{\partial \cL P^{-TP}} \wedge S^1 
        \xrightarrow{\simeq}
        \frac{\cL P^{-TP}}{\partial \cL P^{-TP}} \wedge \bS \wedge \Sigma^\infty S^1 
        \xrightarrow{Id \wedge [P] \wedge Tr_{diag}}
        \frac{\cL P^{-TP}}{\partial \cL P^{-TP}} \wedge \frac{\cL P^{-TP}}{\partial \cL P^{-TP}} \wedge \Sigma^\infty \frac{\cL (P \times P)}{P \times P} 
        \\[8pt]
        \xrightarrow{\simeq} 
        \frac{\cL (P \times P)^{-T(P \times P)}}{\partial \cL (P \times P)^{-T(P \times P)}} \wedge \Sigma^\infty\frac{\cL (P \times P)}{P \times P} 
        \xrightarrow{\tilde \mu^{P\times P}_{{{l}}}}
        \Sigma^\infty \frac{\cL (P \times P)}{P \times P}
        \xrightarrow{} \Sigma^\infty \frac{\cL P}P \wedge \frac{\cL P}P.
    \end{multline*}  Where the third map is the natural identification and the final map is the natural quotient map. 
    
    Similarly, let $\mu_{{{r}}}\left([\overline T_{diag}], [P] \times \cdot\right)$ denote the composition: 

 \begin{multline*}
        \frac{\cL P^{-TP}}{\partial \cL P^{-TP}} \wedge S^1 
        \xrightarrow{\simeq} \bS \wedge \frac{\cL P^{-TP}}{\partial \cL P^{-TP}} \wedge \Sigma^\infty S^1
        \xrightarrow{\operatorname{swap}} 
        \Sigma^\infty S^1 \wedge \bS \wedge \frac{\cL P^{-TP}}{\partial \cL P^{-TP}} 
        \xrightarrow{\overline{Tr}_{diag} \wedge [P] \wedge Id}
        \\[8pt]
        \Sigma^\infty \frac{\cL (P \times P)}{P\times P} \wedge \frac{\cL (P \times P)^{-T(P \times P)}}{\partial \cL (P \times P)^{-T(P \times P)}} 
        \xrightarrow{\tilde \mu_{{{r}}}^{P\times P}} 
        \Sigma^\infty \frac{\cL (P \times P)}{P \times P} 
        \xrightarrow{}
        \Sigma^\infty \frac{\cL P}P \wedge \frac{\cL P}P.
    \end{multline*}
\end{definition}

\begin{remark}
    As suggested in the notation, the compositions appearing in Definition \ref{def: compact not} are the appropriate spectral-level analogues of taking the cross product with the fundamental class $[P]$ and then taking the Chas-Sullivan product in $P \times P$ with the classes $Tr_{diag}$ and $\overline{Tr}_{diag}$ in $\pi_1^{st}$, and indeed this is exactly what these maps do on any generalised homology theory.
\end{remark}

The main theorem of the section is: 
\begin{theoremaaa}\label{thm: T Xi}
    Let $M \subseteq P$ be a codimension 0 submanifold with corners, such that the complement $W := P \setminus M^\circ$ is an $h$-cobordism. Assume that there exists a codimension $0$ embedding $e: P \to \R^L$.\par 
    
    Then there are homotopies$$\Xi_{{{l}}}\simeq  \mu_{{{l}}}\left((\cdot \times [P]), [T_{diag}]\right)$$ 
and 
   $$\Xi_{{{r}}}\simeq \mu_{{{r}}}\left([\overline T_{diag}], [P] \times \cdot\right) .$$

\end{theoremaaa}

\begin{remark}
    The assumption that $P$ embeds as a codimension $0$ submaniold  of $\R^L$ is not necessary, but is sufficient to prove Theorem \ref{main theorem}, and simplifies some of the arguments.
\end{remark}
Combining Theorems \ref{thm: cop Xi} and \ref{thm: T Xi}, we have:
\begin{corollary}\label{cor: 9+10}
    There is a homotopy:
    \begin{equation}
        \Delta^P-(j \wedge j) \circ \Delta^M \circ \overline F_1 \simeq \mu_l(\cdot \times [P], [T_{diag}]) - \mu_r([\overline T_{diag}],[P] \times \cdot)
    \end{equation}
\end{corollary}

The proof of Theorem \ref{thm: T Xi} constitutes the rest of this subsection. We show the statement for the left product; the right case is identical. We first make convenient choices of trace data.

\subsubsection{Convenient data}

\begin{lemma}\label{lem: cop def choices}
    We can choose trace data $(Q,F) \in TD^L(M \xhookrightarrow{j} P)$, as well as collars $\cC_P$ and $\cC_M$ as in Section \ref{sec: T Tr}, so that the following conditions hold:

    \begin{enumerate}[(i).]
            \item \label{eq: cCP small} If $x \in \cC_P$, there is a (necessarily unique) $s^+=s^+(x) \in [0,1]$ such that $F_{[0,s^+]}(x) \subseteq \cC_P$ is a straight line in the collar direction, and $F|_{(s^+, 1]}(x) \subseteq P \setminus \cC_P$.

            \item \label{eq: cCM small}
            Whenever $x \in \cC_M$, the path $F(x)$ lies in $\cC_M$ and is a straight line in the collar direction.

            \item \label{eq: F sminus small}
            For all $x \in \cC_M$, the path $F(x)$ has length $\leq \frac \zeta 4$.
            
            \item \label{eq: F splus small} For all $x \in \cC_P$, $F|_{[0,s^+]}(x)$ has length $\leq \frac \zeta 4$
             \item  \label{eq: V 0 away col}
            $V=0$ \textrm{ on } $P \setminus (M \cup \cC_M \cup \cC_P^{in})$

       \item \label{eq: def eps small P}
            $d(P\setminus \cC_P, \cC_P^{in}) > \eps$
       \end{enumerate}
\end{lemma}
\begin{proof}
    We first choose $e$ and $\rho^{ext}$ any embeddings as in Definition \ref{def: cop dat}, and then $\zeta > 0$ sufficiently small. Next, choose disjoint collar neighbourhoods of the boundaries $\cC_M$ and $\cC_P$, which are small enough that the straight lines in each collar neighbourhood all have length $\leq \zeta/4$; this ensures (\ref{eq: F splus small}) and (\ref{eq: F sminus small}) hold.\par 
    Choose a vector field $V$ on $P$ which points into $M$ along $\partial M$ and into $P$ on $\partial P$, and which satisfies (\ref{eq: V 0 away col}), and scale $V$ down to be sufficiently small.\par 
    Specifying a smooth strong deformation retraction $F: P \times [0,1] \to P$ is the same as a smoothly-varying family of paths $\{F_{t}(x)\}_{t \in [0,1]}$ for $x \in P$. We first choose any smooth strong deformation retraction $F$, then modify $F$ by preconcatenating (and reparametrising appropriately) the paths $\{F_{t}(x)\}_{t \in [0,1]}$ with a straight line in the collar direction for all $x \in \cC_P$ and postcomposing similarly for all $x \in \cC_P$; this ensures that (\ref{eq: cCP small}) and (\ref{eq: cCM small}) hold.\par 
    We now choose $\eps > 0$ sufficiently small that (\ref{eq: def eps small P}) holds.
\end{proof}

Given $F$ satisfying the conditions in \cref{lem: cop def choices}, let $T=T(F)$ be the framed manifold defined as in \cref{sec: framed bordism invariant} and $f: T \to \cL P$ the natural map sending $(x,t)$ to the loop $F|_{[0,t]}$ from $x$ to itself. Let $[T] \in \Omega^{fr}_1(\cL P/P)$ be the associated framed bordism class.

\begin{lemma}\label{lem: no bound T}
    We can choose $(Q, F) \in TD^L(M \xhookrightarrow{j} P)$ such that the conditions in Lemma \ref{lem: cop def choices} hold, and additionally $T$ has no boundary.
\end{lemma}
\begin{proof}
    Consider the vector field $V'$ on $W$, where $V'(p) = \frac{d}{ds}|_{s=0} F_s(p)$. Zeroes of this vector field in $W \setminus \partial M$ biject with points in $\partial T$. Since the relative Euler characteristic $\chi(W, \partial M)$ vanishes, we can choose $F$ so that this vector field has no zeros; furthermore this is compatible with the proofs of Lemmas \ref{lem: cop def choices} and \ref{lem: gen pert F} (by making $V'$ agree with the appropriate collar vector field near $\partial M$).
\end{proof}
We assume we have chosen $(Q, F)$ so that the conclusion of Lemma \ref{lem: no bound T} also holds. We consider the following composition, which is the composition appearing in \cref{def: compact not} on $(3L)^{th}$ spaces (see Appendix \ref{sec: smash}): 
\begin{multline}\label{eq: unst xi comp}
    \frac{\cL P}{\partial \cL P} \wedge \Sigma^L S^0 \wedge \Sigma^L S^1 
    \xrightarrow{Id \wedge[P]_{unst} \wedge (Tr_{diag})_{unst}}
    \frac{\cL P}{\partial \cL P} \wedge \frac{\cL P}{\partial \cL P} \wedge \Sigma^L_+ \cL (P \times P) 
    \\
    \xrightarrow{\mu^{P \times P}_{r, unst}}
    \Sigma^{3L}_+ \cL P \times \cL P
    \to 
    \Sigma^{3L} \frac{\cL P}P \wedge \frac{\cL P}P
\end{multline}
where $[P]_{unst}$ and $(Tr_{diag})_{unst}$ are maps of spaces representing the maps of spectra $[P]$ and $Tr_{diag}$ respectively, as in Appendix \ref{sec: smash}.\par 
To prove Theorem \ref{thm: T Xi}, it suffices to show that (\ref{eq: unst xi comp}) is homotopic to the map sending $(\gamma, u, v, t)$ (so $\gamma \in \cL P,u,v \in [-1,1]^L$ and $t \in S^1$) to 
\begin{equation}\label{eq: ot prod}
    (u, v, \Xi_{{{l}},unst}(\gamma,t))
\end{equation}
\begin{remark}
    Though the first map in (\ref{eq: unst xi comp}) may depend on the choice of vector field in the proof of Lemma \ref{lem: no bound T} (which isn't necessarily unique up to homotopy), the total composition does not.
\end{remark}

\subsubsection{Simplifying $\Xi_{{{l}}}$}

\begin{lemma}\label{lem: V 0}
    Let $(\gamma,s) \in \frac{\cL P}{\partial \cL P} \wedge S^1$. If $\gamma(0)$ lies in $M$, $\cC_M$ or $\cC_P^{in}$, then $\Xi_{{{l}},unst}(\gamma,s)$ is given by the basepoint.\par 
    In particular, if $\Xi_{{{l}},unst}(\gamma,s)$ isn't the basepoint, then by (\ref{lem: cop def choices}.\ref{eq: V 0 away col}) and (\ref{lem: cop def choices}.\ref{eq: cCP small}), $V$ vanishes at $F_s \circ \gamma(0)$.
\end{lemma}
\begin{proof}
    If $\gamma(0) \in M$, the final term in (\ref{eq: tr XiR}) is constant.\par 
    If $\gamma(0) \in \cC_M$, then by (\ref{lem: cop def choices}.\ref{eq: cCM small}) and (\ref{lem: cop def choices}.\ref{eq: F sminus small}), the final term of (\ref{eq: tr XiR}) is again constant.\par 
    Now suppose $\gamma(0) \in \cC_P^{in}$. If $s \leq s^+(\gamma(0))$, then by (\ref{lem: cop def choices}.\ref{eq: F splus small}), the final term of (\ref{eq: tr XiR}) is constant. If instead $s \geq s^+(\gamma(0))$, by (\ref{lem: cop def choices}.\ref{eq: def eps small P}), the incidence condition for (\ref{eq: tr XiR}) can't hold.
\end{proof}
\begin{lemma}\label{lem: compare IC for Xi and T}
    For $\lambda > 0$ large enough, for any $(\gamma,s) \in \frac{\cL P}{\partial \cL P} \wedge S^1$, if $\Xi_{{{l}},unst}(\gamma, s)$ is not equal to the basepoint, then $(\gamma(0), s) \in \sigma_\chi(D\nu_i)$.\par 
\end{lemma}
\begin{proof}
    Same as Lemma \ref{lem: change inc cond}.
\end{proof}
We now assume we have made choices such that $\lambda$ satisfies the hypothesis of Lemma \ref{lem: compare IC for Xi and T}. By Lemmas \ref{lem: V 0} and \ref{lem: compare IC for Xi and T}, we can write an alternative formula for $\Xi_{{{l}},unst}$ with respect to these choices of data:
\begin{corollary}\label{cor: XiR IC}
    For $(\gamma, s) \in 
    \frac{\cL P}{\partial \cL P} \wedge S^1$, we have that $\Xi_{{{l}},unst}( \gamma, s)$ is equal to
    \begin{equation}\label{eq: XiR IC}
        \begin{cases}
            \begin{pmatrix}
                \lambda\left(\gamma(0)-F_s \circ \gamma(0)\right), \\
                B\left(
                    \gamma(0) \overset{\gamma}{\rightsquigarrow} 
                    \gamma(0) \overset{F|_{[0,s]}}{\rightsquigarrow}
                    F_s \circ \gamma(0) \overset{\theta}{\rightsquigarrow}
                    \gamma(0)
                \right),\\
                B\left(
                    \gamma(0) \overset{\theta}{\rightsquigarrow}
                    F_s \circ \gamma(0)  \overset{\overline F|_{[0,s]}}{\rightsquigarrow}
                    \gamma(0)
                \right)
            \end{pmatrix} &
            \textrm{ if } (\gamma(0), s) \in \sigma_\chi(D\nu_i)\\
            * & 
            \textrm{otherwise.}
        \end{cases}
    \end{equation}
\end{corollary}
Note that (\ref{eq: XiR IC}) is the equation (\ref{eq: tr XiR}), with the incidence condition replaced by that of (\ref{eq: simp tr}), and with all instances of $\phi$ removed.

\subsubsection{Proof} 
\begin{proof}[Proof of Theorem \ref{thm: T Xi}]
    Using Lemma \ref{lem: T vs Tr}, Lemma \ref{lem: B}, Lemma \ref{lem: V 0} to remove instances of $\phi$ and then plugging in the definitions, we see that (\ref{eq: unst xi comp}) is homotopic to the map which sends $(\gamma, u, x, s)$ (so $\gamma \in \cL P$, $u, x \in [-1, 1]^L$ and $s \in [0,1]$) to:
    \begin{equation}\label{eq: simplified product}
        \begin{cases}
            \begin{pmatrix}
                \lambda(\gamma(0) - x),\\
                \lambda( u-x),\\
                \lambda(x-F_s(x)),\\
                B\left(\gamma(0) \overset \gamma \rightsquigarrow 
                \gamma(0) \overset \theta \rightsquigarrow
                x \overset {F|_{[0,s]}} \rightsquigarrow 
                F_1(x) \overset \theta \rightsquigarrow 
                x \overset \theta \rightsquigarrow
                \gamma(0)\right),\\
                B\left( u \overset \theta \rightsquigarrow 
                x \overset \theta \rightsquigarrow 
                F_s(x) \overset {\overline F|_{[0,s]}} \rightsquigarrow 
                x \overset \theta \rightsquigarrow u\right)
            \end{pmatrix}
            & 
            \begin{matrix}
                \textrm{ if } u \in P,\, x \in P,\, \lVert x-F_s(x)\rVert \leq \eps \\
                \textrm{ and } \lVert (\gamma(0), u) - (x,x)\rVert \leq \eps
            \end{matrix}
            \\
            * & \textrm{ otherwise.}
        \end{cases}
    \end{equation}
    Note that the first two conditions of the incidence condition of (\ref{eq: simplified product}) are implied by the final two, implying they are redundant and we may therefore drop them.\par 
    We argue that this map is homotopic to (\ref{eq: ot prod}). The final terms are homotopic via a homotopy similar to the one between the final terms described in the proof of Lemma \ref{lem: T vs Tr}.\par 
    Then the second entry may be replaced with $\lambda u$, by a homotopy which replaces $(u-x)$ with $(u-\tau x)$ at time $\tau \in [0,1]$, both in the second entry and in the incidence condition.\par    
    The the third entry can be replaced by $\lambda(\gamma(0)- F_s\circ \gamma(0))$, by a homotopy which at time $\tau$ replaces $(x-F_s(x))$ with $z_\tau(x,y)-F_s(z_\tau(x,y))$ where $\{z_\tau(x,y)\}_\tau$ is a straight-line path between $x$ and $y$, both in the third entry and in the incidence condition.\par 
    Then the first entry can be replaced with $-\lambda x$ via a similar argument to the second entry. The resulting map then differs from (\ref{eq: ot prod}) only by applying the linear transformation $\begin{pmatrix} 0& Id_L \\ -Id_L & 0\end{pmatrix}$ to the first two entries; this matrix has positive determinant so is homotopic to the identity in $O(2L)$.
\end{proof}

\section{Proof of \cref{main theorem}}\label{proof of main thm}
    In this section we prove Theorem \ref{main theorem} using the results of the previous sections. We first reduce to the case where the homotopy equivalence is a codimension $0$ embedding of manifolds with corners, and then appeal to results of Section \ref{sec: cop defect}.\par 
    Let $f: N \to Z$ be a homotopy equivalence of compact manifolds as in Theorem \ref{main theorem}. Embed $Z$ into $\bR^L$ for some large $L$. Let $P$ be the unit disc bundle of the normal bundle, which we embed as a submanifold of $\bR^L$ extending the embedding of $Z$. Composing $f$ with the inclusion of the zero section $Z \hookrightarrow P$ gives a map $N \to P$. This is not an embedding, but we can choose a generic perturbation to an embedding $N \hookrightarrow P \subset \R^L$.  Let $M$ be the unit disc bundle of this embedding $N \hookrightarrow P$, which we can assume embeds as a submanifold of $P$ extending the embedding of $N$. Let $j: M \hookrightarrow P$ be the inclusion. Note $j$ is a codimension $0$ embedding. Then there is a homotopy commutative diagram:
    \begin{equation}\label{diag: emb comm}
        \begin{tikzcd}
            N \arrow[r, "f"] \arrow[d, "\iota^N"] &
            Z \arrow[d, "\iota^Z"]\\
            M \arrow[r, "j"]&
            P
        \end{tikzcd}
    \end{equation}
    where the vertical arrows, $\iota^N$ and $\iota^Z$, are the inclusions of the zero sections, and in particular are simple homotopy equivalences.\par 
    Let $\nu_N$ and $\nu_Z$ be the normal bundles of the embeddings $N, Z \hookrightarrow \bR^L$ respectively, so $M \cong \operatorname{Tot}(D\nu_N)$ and $P \cong \operatorname{Tot}(D\nu_Z)$.

    \begin{lemma}
        For $L$ sufficiently large, the complement $W := P \setminus M^\circ$ is an $h$-cobordism.
    \end{lemma}
    \begin{proof}
        The argument is similar to \cite[Remark 1.1.4]{WJR}.\par 
        We first argue that the inclusions $\partial M, \partial P \hookrightarrow W$ induce isomorphisms on $\pi_1$. We use the following decomposition:
       
        \begin{equation}
            \partial P \cong \operatorname{Tot}(S\nu_Z)\cup_{\operatorname{Tot}(S\nu_Z|_{\partial Z})} \operatorname{Tot}(D\nu_Z|_{\partial Z})
        \end{equation}
        Since the fibres of the sphere bundle $S\nu_Z$ are high-dimensional spheres, by the long exact sequence of a fibration we see that the projections $\operatorname{Tot}(S\nu_Z) \to Z$ and $\operatorname{Tot}(S\nu_Z|_{\partial Z}) \to \partial Z$ induce isomorphisms on $\pi_1$. Therefore by Seifert-van Kampen, we find that
        \begin{equation}
            \pi_1(\partial P) \cong \pi_1(Z) *_{\pi_1 \partial Z} \pi_1 \partial Z \cong \pi_1 Z
        \end{equation}
        It follows that the inclusion $\partial P \hookrightarrow P$ induces an isomorphism on $\pi_1$. Exactly the same argument shows that the inclusion $\partial M \hookrightarrow M \simeq P$ does too.

        Since the handle dimension of $M$ is at most the dimension of $N$ and thus bounded above independently of $L$, for $L$ sufficiently large any loop in $P$ can be generically perturbed away from the skeleton of some handle decomposition of (a smoothing of) $M$, 
        and therefore can be homotoped to live in $W$. Similarly given any loops in $W$ which are homotopic in $P$, the homotopy can be generically perturbed away from the same skeleton, and therefore can be homotoped to live entirely in $W$. It follows that $\partial M, \partial P \hookrightarrow W$ induce isomorphisms on $\pi_1$.\par 
        Now by excision and using the above isomorphisms on $\pi_1$, the relative homology group with universal local coefficients $H_*(W, \partial M; \bZ[\pi_1]) \cong H_*(P, M; \bZ[\pi_1]) = 0$ vanishes. Using Alexander duality, we see also that $H_*(W, \partial P; \bZ[\pi_1])$ also vanishes. It follows that $W$ is an $h$-cobordism.
    \end{proof}
    \begin{remark}
        Note that the $h$-cobordism $W$ here may have corners, but this does not affect any of the following arguments or constructions.
    \end{remark}

    The inclusion $j:M  \hookrightarrow P$ now satisfies the conditions of Section \ref{sec: cop defect}. Choose a strong deformation retraction $F: W \times [0,1] \to W$ and extend it by the identity to $F: P \times [0,1] \to P$; let $\overline F_1$ be as in (\ref{eq: overline F1}).\par 
    We next define a map \begin{equation}\label{eq: h.e on -TN}
         f_!:\frac{\cL N^{-TN}}{\partial \cL N^{-TN}} \to \frac{\cL Z^{-TZ}}{\partial \cL Z^{-TZ}}, 
         \end{equation}and give an alternative characterisation of it in the case that $N$ and $Z$ have no boundary.\par 
    Since $\overline F_1$ and $\alpha^Z$ are homotopy equivalences, we may choose a map $f_!$ such that the following diagram commutes up to homotopy, and this choice is well-defined up to homotopy:
    \begin{equation}
        \begin{tikzcd}
            \frac{\cL N^{-TN}}{\partial \cL N^{-TN}} \arrow[r, "f_!"] \arrow[d, "\alpha^N"] &
            \frac{\cL Z^{-TZ}}{\partial \cL Z^{-TZ}} \arrow[d, "\alpha^Z"] \\
            \frac{\cL M^{-TM}}{\partial \cL M^{-TM}} &
            \frac{\cL P^{-TP}}{\partial \cL P^{-TP}} \arrow[l, "\overline F_1"]
        \end{tikzcd}
    \end{equation}
    \begin{proposition}\label{prop: ati}
        Suppose that $N$ and $Z$ are both closed manifolds. Then $f_!$ is homotopic to the following composition:
        \begin{equation*}
            \cL N^{-TN} \xrightarrow{\simeq} \cL N^{-f^*TZ} \xrightarrow{f} \cL Z^{-TZ}
        \end{equation*}
        where the first map is given by Atiyah's equivalence \cite{Atiyah} between $-TN$ and $-f^*TZ$, as stable spherical fibrations.\par 
        In particular, if $N$ and $Z$ are oriented and $f$ is orientation-preserving, then the following diagram commutes:
        \begin{equation*}
            \begin{tikzcd}
                H_{*+n}(\cL N) \arrow[r, "(\cL f)_*"] \arrow[d, "\mathrm{Thom}"] &
                H_{*+n}(\cL Z) \arrow[d, "\mathrm{Thom}"] \\
                H_*(\cL N^{-TN}) \arrow[r, "(f_!)_*"] &
                H_*(\cL Z^{-TZ})
            \end{tikzcd}
        \end{equation*}
    \end{proposition}
    \begin{proof}
        We first recap (a version of) the construction of the equivalence of stable spherical fibrations $-TN \simeq -f^* TZ$ from \cite{Atiyah}. We construct this as a map $\mathrm{Ati}: f^*D\nu_Z \to D\nu_N$ of fibre bundles over $N$, sending boundaries to boundaries, that is a fibrewise homotopy equivalence of pairs. We make use of the fact that using the vector bundle structure, between any two points in the same fibre of the disc bundle of a vector bundle, there is a canonical path given by taking the convex hull of these two points; we call this a \emph{fibre line path} and write these paths $\operatorname{Fib^\pi}$ for a vector bundle $\pi: E \to B$; in general it should be unambiguous what the endpoints are.\par
        Let $j, \iota^N, \iota^Z$ be as in   \cref{diag: emb comm}. Let $h'$ be a homotopy from $h'_0 = j \circ \iota^N$ to $h'_1 = \iota^Z \circ f: N \to P$, and let $h = F_1 \circ h'$, a homotopy between $\iota^N, F_1\circ \iota^Z \circ f: N \to M$. \par 
        Let $x \in N$, and choose a vector $v \in (f^*D\nu_Z)_x = (D\nu_Z)_{f(x)}$. Let $u = F_1(v) \in P \cong D\nu_N$. $u$ does not necessarily live in the fibre over $x$; it instead lives in the fibre over $\pi^N\circ F_1(v)$. We parallel transport along a natural path between these two points. \par 
        Consider the path in $N$:
        \begin{equation}\label{eq: delta v x}
            \delta^{v,x}: \pi^N \circ F_1(v) \overset{\pi^N \circ F_1 \circ \operatorname{Fib^{\pi^Z}}}{\rightsquigarrow}
            \pi^N \circ F_1 \circ \iota^Z \circ f(x) \overset{\overline{\pi^N \circ F_1 \circ h'(x)}}{\rightsquigarrow}
            \pi^N \circ F_1 \circ j \circ \iota^N(x) = x
        \end{equation}
     where the first path in the concatenation is $\pi^N\circ F_1$ composed with a fibre line path of the disc bundle $M \to N$. We define $\operatorname{Ati(v)}$ to be the image of $F_1(v)$ under the parallel transport map along the path $\delta^{v,x}$; this lives in the fibre over $x$ by construction, and assuming we parallel transport along a metric-compatible connection, if $|v|=1$ then $|\operatorname{Ati}(v)|=1$, so this induces a well-defined map of spherical fibrations.\par 
        It suffices to show that the following diagram commutes up to homotopy, which we do by writing down an explicit homotopy:
        \begin{equation*}
            \begin{tikzcd}
                \cL N^{f^* D\nu_Z} \arrow[r, "\operatorname{Ati}"] \arrow[d, "f"] &
                \cL N^{D\nu_N} \arrow[dr, "\alpha^N"] & \\
                \cL Z^{D\nu_Z} \arrow[r, "\alpha^Z"] &
                \frac{\cL Z}{\partial \cL Z} \arrow[r, "\overline F_1"] &
                \frac{\cL M}{\partial \cL M}
            \end{tikzcd}
        \end{equation*}
        We define a homotopy $\{H_t\}_{t \in [0,1]}: \cL N^{f^*D\nu_Z} \to \frac{\cL M}{\partial \cL M}$ as follows. Choose $(\gamma, v) \in \cL N^{f^* D\nu_Z}$ and $t \in [0,1]$.\par 
        We first define $u_t^{v,\gamma}\in P$ to be the image of $v$ along the parallel transport map along the path in $Z$:
        \begin{equation*}
            f \circ \gamma(0) \overset{\overline{\pi^Z \circ h'|_{[t,1]} \circ \gamma(0)}}{\rightsquigarrow} 
            \pi^Z \circ h'_t \circ \gamma(0)
        \end{equation*}
        Note $u_1^{v,\gamma} = v$. We also define a path $\delta^{v,\gamma}_t$ in $N$:
        \begin{equation*}
            \pi^N \circ F_1(v) \overset{\pi^N \circ F_1 \circ \operatorname{Fib^{\pi^Z}}}{\rightsquigarrow}
            \pi^N \circ F_1 \circ \iota^Z \circ f \circ \gamma(0) 
            \overset{\overline{\pi^N \circ F_1 \circ  h'|_{[t,1]} \circ \gamma(0)}}{\rightsquigarrow} 
            \pi^N \circ F_1 \circ h'_t\circ \gamma(0) \overset{\pi^N \circ F_1 \circ \operatorname{Fib}^{\pi^Z}}{\rightsquigarrow}
            \pi^N \circ F_1 (t \cdot u_t^{v,\gamma})
        \end{equation*}
        where $t \cdot u_t$ denotes $u_t$ rescaled by $t$. Let $w_t^{v,\gamma} \in M$ be the image of $F_1(v)$ under the parallel transport along the path $\delta^{v,\gamma}_t$; note that $w_1^{v, \gamma} = F_1(v)$ since $\delta_1^{v,\gamma}$ consists of a path concatenated with its inverse. By inspection of (\ref{eq: delta v x}) we see that $\delta^{v,\gamma}_0 = \delta^{v, \gamma(0)}$; from this we also see that $w_0^{v, \gamma} = \operatorname{Ati}_{\gamma(0)}(v)$.\par 
        We define $H_t(v,\gamma)$ to be the following loop:
        \begin{equation*}
            w_t^{v,\gamma} \overset{\operatorname{Fib}^{\pi^N}}{\rightsquigarrow}
            F_1(t \cdot u_t^{v,\gamma}) \overset{F_1 \circ \operatorname{Fib}^{\pi^Z}}{\rightsquigarrow} 
            h_t \circ \gamma(0) \overset{h_t \circ \gamma}{\rightsquigarrow} 
            h_t \circ \gamma(0) \rightsquigarrow 
            F_1(t \cdot u_t^{v,\gamma}) \rightsquigarrow w_t^{v,\gamma} 
        \end{equation*}
        where the last two paths are the reverses of the first two paths. \par 
        Then since $w_1^{\gamma, v} = F_1(v)$ and $h_1 = F_1 \circ \iota^Z \circ f$, we see that $H_1(v, \gamma) = (v,\overline{F_1 \circ \alpha^Z \circ f \circ \gamma})$.\par 
        Similarly, since $\delta^{v,x}_0 = \delta^{v,x}$, $0 \cdot u_0^{\gamma,v} = \pi^Z u_0^{\gamma,v}$ and $h_0 = \iota^N$, we see that $H_0 = \alpha^N \circ \operatorname{Ati}$.
        
    \end{proof}
\begin{proof}[Proof of \cref{main theorem}]
    Now consider the following diagram.
    \begin{equation}\label{diag: cube}
        \begin{tikzcd}
             \frac{\cL N^{-TN}}{\partial \cL N^{-TN}} \wedge S^1 \arrow[rr, "\Delta^N"] \arrow[dd, "f_! \wedge Id_{S^1}"] \arrow[dr, "\alpha^N \wedge Id_{S^1}"] & &
            \Sigma^\infty \frac{\cL N}N \wedge \frac{\cL N}N \arrow[dd, near start, "f \wedge f"] \arrow[dr, "\iota^N \wedge \iota^N"] & \\
            &  \frac{\cL M^{-TM}}{\partial \cL M^{-TM}} \wedge S^1 \arrow[rr, near start,  "\Delta^M"] &&
            \Sigma^\infty \frac{\cL M}M \wedge \frac{\cL M}M \arrow[dd, "j \wedge j"] \\
            \frac{\cL Z^{-TZ}}{\partial \cL Z^{-TZ}} \wedge S^1 \arrow[dr, swap, "\alpha^Z \wedge Id_{S^1}"] \arrow[ rr, near start, "\Delta^Z"] &&
            \Sigma^\infty \frac{\cL Z}Z\wedge \frac{\cL Z}Z \arrow[dr, "\iota^Z \wedge \iota^Z"] & \\
            & \frac{\cL P^{-TP}}{\partial \cL P^{-TP}} \wedge S^1 \arrow[rr, "\Delta^P"] \arrow[uu, near start, "\overline F_1 \wedge Id_{S^1}"] &&
            \Sigma^\infty \frac{\cL P}P \wedge \frac{\cL P}P
        \end{tikzcd}
    \end{equation}
    where $\alpha^N, \alpha^Z$ are the homotopy equivalences from Lemma \ref{lem: stab dom}. The back cube is the square (\ref{diag: fail of hom inv }) whose failure to homotopy commute we wish to determine.\par 
    The top and bottom squares in (\ref{diag: cube}) homotopy commute by Theorem \ref{thm: stability}. The left square homotopy commutes by construction. The right square homotopy commutes by homotopy commutativity of (\ref{diag: emb comm}).

    Now we invoke the invariant from Section \ref{sec: cop defect}:
    \begin{definition}\label{def: T hom eq}
        Let $[T] \in \Omega^{fr}_1(\cL P, P)$ be the framed bordism fixed-point invariant associated to the inclusion $j: M \hookrightarrow P$, as in Section \ref{sec: cop defect}. We also write $[T]: \Sigma^\infty S^1 \to \Sigma^\infty \frac{\cL Z}Z$ for the corresponding stable homotopy class under the Pontrjagin-Thom isomorphism.\par 
        As in Section \ref{sec: tr is product}, we let $[T_{diag}]$ and $[\overline T_{diag}]$ be given by $[T]$ composed with the two antidiagonal maps.
    \end{definition}
    A proof similar to Definition \ref{thm: stability} shows that the class $[T]  \in \Omega^{fr}_1(\cL Z, Z)$ only depends on the homotopy equivalence $f: N \to Z$, and none of the auxiliary choices.\par 
    The front square of (\ref{diag: cube}) does not necessarily commute, but its failure to commute is determined by Corollary \ref{cor: 9+10}:
    
    \begin{equation}\label{eq: hom in main proof}
        \Delta^P- (j \wedge j) \circ\Delta^M \circ \overline F_1 \simeq \mu_{{{l}}}((\cdot \times [P]), [T_{diag}]) - \mu_{{{r}}}([\overline T_{diag}], [P] \times \cdot)
    \end{equation}
    where the maps on the right are as in Section \ref{sec: tr is product}.

    We now use a form of Theorem \ref{thm: prod stab} for the setting at hand:
    \begin{lemma}\label{lem: stab incorporated}
        The following diagram commutes up to homotopy:
        \begin{equation}
            \begin{tikzcd}
                \frac{\cL Z^{-TZ}}{\partial \cL Z^{-TZ}} \wedge S^1 
                \arrow[d, "\alpha^Z \wedge Id_{S^1}"]
                \arrow[rrr, "{\mu_{{{l}}}^{Z \times Z}(\cdot \times [Z], [T_{diag}])}"] &&&
                \Sigma^\infty \frac{\cL Z}Z \wedge \frac{\cL Z}Z \\
                \frac{\cL P^{-TP}}{\partial \cL P^{-TP}} \wedge S^1 
                \arrow[rrr, "{\mu^{P \times P}_{{{l}}}(\cdot \times [P], [T_{diag}])}"] 
                &&&
                \Sigma^\infty \frac{\cL P}P \wedge \frac{\cL P}P 
                \arrow[u, "\pi^Z \wedge \pi^Z"]
            \end{tikzcd}
        \end{equation}
        where the horizontal maps are defined as in Theorem \ref{thm: T Xi}.\par  
        A similar diagram commutes with the top and bottom horizontal arrows replaced by $\mu_{{{r}}}^{Z \times Z}([\overline T_{diag}], [Z] \times \cdot)$ and $\mu_{{{r}}}^{P \times P}([\overline T_{diag}], [Z] \times \cdot)$ respectively.
    \end{lemma}
    \begin{proof}
        Follows from homotopy commutativity of (\ref{eq: stab fund class}) and Theorem \ref{thm: prod stab}, together applied to $Z \times Z$ along with its inclusion into $\bR^{2L}$ which has unit disc bundle $P \times P$.
    \end{proof}
    \cref{main theorem} then follows from the homotopy commutativity of four of the squares in (\ref{diag: cube}), along with (\ref{eq: hom in main proof}) and Lemma \ref{lem: stab incorporated}.
\end{proof}

\section{Proof of Corollary \ref{main cor}}
    Let $f: N \to Z$ be an orientation-preserving homotopy equivalence of closed oriented manifolds.
    \begin{proposition}\label{prop: hom comp inc}
        Let $M$ be a closed oriented manifold. Let $\tau \in \pi^{st}_1(\frac{\cL (M \times M)}{M \times M})$.\par 
        Then the following diagram commutes up to a factor of $(-1)^{np}$:
        \begin{equation}\label{eq: ffff}
            \begin{tikzcd} 
                H_{p+1-n}\left(\cL M^{-TM} \wedge S^1\right)
                \arrow[d, "\operatorname{Thom} \wedge Id_{S^1}"]
                \arrow[rrr, "{\left(\mu_{{{l}}}\left(\cdot \times [M], \tau\right)\right)_*}"] 
                &&&
                H_{p+1-n}\left(\Sigma^\infty \frac{\cL M}M \wedge \frac{\cL M}M\right) 
                \arrow[dd, "="]\\
                \tilde H_{p+1}\left(\cL M_+ \wedge S^1\right) 
                &&&\\
                H_p(\cL M)
                \arrow[u, "\cdot \times{[0,1]}"]
                \arrow[rrr, "{\mu^{CS}(\cdot \times [M], h_*\tau)}"]
                &&&
                \tilde H_{p+1-n}\left(\frac{\cL M}M \wedge \frac{\cL M}M\right)
            \end{tikzcd}
        \end{equation}

        Similarly, the following diagram commutes up to a factor of $(-1)^{p}$:
        \begin{equation}\label{eq: kkkk}
            \begin{tikzcd}
                H_{p+1-n}\left(\cL M^{-TM} \wedge S^1\right)
                \arrow[d, "\operatorname{Thom} \wedge Id_{S^1}"]
                \arrow[rrr, "{\left(\mu_{{{r}}}({\tau}, [M] \times \cdot)\right)_*}"] 
                &&&
                H_{p+1-n}\left(\Sigma^\infty \frac{\cL M}M \wedge \frac{\cL M}M\right) 
                \arrow[dd, "="]\\
                \tilde H_{p+1}\left(\cL M_+ \wedge S^1\right) 
                &&&\\
                H_p(\cL M)
                \arrow[u, "\cdot \times{[0,1]}"]
                \arrow[rrr, "{\mu^{CS}(h_*\tau, [M] \times \cdot)}"]
                &&&
                \tilde H_{p+1-n}\left(\frac{\cL M}M \wedge \frac{\cL M}M\right)
            \end{tikzcd}
        \end{equation}
        
    \end{proposition}

    \begin{proof}
        Consider the following diagram:
        \begin{equation}
            \begin{tikzcd}\label{eq: gggg}
                H_{p+1-n}(\cL M^{-TM}\wedge S^1) 
                \arrow[d, "\operatorname{Thom}"]
                \arrow[r, "\simeq"]
                &
                H_{p+1-n}(\cL M^{-TM} \wedge \bS \wedge \Sigma^\infty S^1)
                \arrow[d, "\operatorname{Thom}"]
                \arrow[r, "Id \wedge {[M]} \wedge Id"]
                &
                H_{p+1}\left(\left(\cL M^{-TM} \right)^{\wedge 2}\wedge \cL M^{-TM} \wedge S^1\right) 
                \arrow[d, "\operatorname{Thom}"]
                \\
                \tilde H_{p+1}(\cL M_+ \wedge S^1) 
                \arrow[r, "\simeq"]
                &
                H_{p+1}(\cL M_+ \wedge \bS \wedge \Sigma^\infty S^1) 
                &
                H_{p+1+n}\left(\Sigma^\infty \cL M_+^{\wedge 2} \wedge \Sigma^\infty S^1\right)
                \\
                H_p(\cL M)
                \arrow[u, "{\cdot \times [0,1]}"]
                \arrow[r, "="]
                &
                H_p(\cL M)
                \arrow[u, "{\cdot \times [0,1]}"]
                \arrow[r, "{\cdot \times [M]}"]
                \arrow[ur, "{\cdot \times [M] \times [0,1]}"]
                &
                H_{p+n}(\cL M \times \cL M)
                \arrow[u, "{\cdot \times [0,1]}"]
            \end{tikzcd}
        \end{equation}
        All of (\ref{eq: gggg}) commutes except the top right trapezium, which commutes up to a factor of $(-1)^{pn}$, coming from commuting $x \in H_p(\cL M)$ past the Thom class of the second copy of $-TM$. Also consider:
        
        \begin{equation}\label{eq: hhhh}
            \begin{tikzcd}
                H_{p+1}\left(\left(\cL M^{-TM}\right)^{\wedge 2} \wedge S^1\right)
                \arrow[d, "\operatorname{Thom}"]
                \arrow[r, "Id \wedge Id \wedge \tau"]
                &
                H_{p+1-n}\left(\left(\cL M^{-TM}\right)^{\wedge 2} \wedge \frac{\cL (M \times M)}{M \times M}\right) 
                \arrow[d, "\operatorname{Thom}"]
                \arrow[r, "\mu^{M \times M}_{{{l}}}"]
                &
                H_{p+1-n}\left(\Sigma^\infty \frac{\cL (M \times M)}{M \times M}\right) 
                \arrow[d, "="] 
                \\
                H_{p+1-n}(\Sigma^\infty \cL M_+^{\wedge 2} \wedge \Sigma^\infty S^1) 
                \arrow[r, "Id \wedge Id \wedge \tau"]
                &
                H_{p+1-n}\left(\Sigma^\infty\cL M_+^{\wedge 2} \wedge \frac{\cL (M \times M)}{M \times M}\right)
                \arrow[r, "\mu^{CS}_{M \times M}"]
                &
                \tilde H_{p+1-n}\left(\frac{\cL (M \times M)}{M \times M}\right)
                \\
                H_{p+n}(\cL M \times \cL M) 
                \arrow[u, "{\cdot \times [0,1]}"]
                \arrow[r, "\cdot \times h_*\tau"]
                &
                \tilde H_{p+1+n}\left(\cL M_+^{\wedge 2} \wedge \frac{ \cL (M \times M)}{M \times M}\right)
                \arrow[u, "="]
                \arrow[r, "\mu^{CS}_{M \times M}"]
                &
                \tilde H_{p+1-n}\left(\frac{\cL (M \times M)}{M \times M}\right)
                \arrow[u, "="]
            \end{tikzcd}
        \end{equation}
        (\ref{eq: hhhh}) commutes; for the top right square this uses Corollary \ref{cor: prod CS spec} applied to $M \times M$ (which is even-dimensional).\par 
        Then the concatenation of (\ref{eq: gggg}) and (\ref{eq: hhhh}), followed by the natural collapse map 
        \begin{equation} \label{eq: coll}
            \tilde H_*\left(\frac{\cL (M \times M)}{M \times M}\right) \to \tilde H_*\left(\left(\frac{\cL M}M\right)^{\wedge 2}\right)
        \end{equation}
        has outer square given by (\ref{eq: ffff}), so (\ref{eq: ffff}) commutes up to a factor of $(-1)^{np}$.\par 
        Consider the following diagram, analagous to (\ref{eq: gggg}):
        \begin{equation}\label{eq: iiii}
            \begin{tikzcd}
                H_{p+1-n}\left(\cL M^{-TM}\right)
                \arrow[d, "\operatorname{Thom}"]
                \arrow[r, "\simeq"] 
                &
                H_{p+1-n}\left(\bS \wedge \cL M^{-TM} \wedge S^1\right)
                \arrow[d, "\operatorname{Thom}"]
                \arrow[r, "{[M] \wedge Id \wedge Id}"]
                &
                H_{p+1-n}\left(\left(\cL M^{-TM}\right)^{\wedge 2} \wedge S^1 \right) 
                \arrow[d, "\operatorname{Thom}"]
                \\
                \tilde H_{p+1}\left(\cL M_+ \wedge S^1\right) 
                \arrow[r, "\simeq"]
                &
                H_{p+1}\left(\bS \wedge \cL M_+ \wedge \Sigma^\infty S^1\right)
                &
                H_{p+1+n}\left(\Sigma^\infty \cL M_+^{\wedge 2}\wedge \Sigma^\infty S^1\right)
                \\
                H_p(\cL M) 
                \arrow[u, "{\cdot \times [0,1]}"]
                \arrow[r, "="]
                &
                H_p(\cL M)
                \arrow[u, "{\cdot \times [0,1]}"]
                \arrow[ur, "{[M] \times \cdot \times [0,1]}"]
                \arrow[r, "{[M] \times \cdot}"]
                &
                H_{p+n}(\cL M \times \cL M)
                \arrow[u, "{\cdot \times [0,1]}"]
            \end{tikzcd}
        \end{equation}
        All of (\ref{eq: iiii}) commutes except the top right trapezium, which commutes up to a factor of $(-1)^n$, coming from commuting $[M] \in H_n(\cL M)$ past the Thom class of the second copy of $-TM$. Also consider:
        \begin{equation}\label{eq: jjjj}
            \begin{tikzcd}
                H_{p+1-n}\left(\left(\cL M^{-TM}\right)^{\wedge 2} \wedge S^1\right)
                \arrow[d, "\operatorname{Thom}"]
                \arrow[r, "\operatorname{Swap}"]
                &
                H_{p+1}\left(\Sigma^\infty S^1 \wedge \left(\cL M^{-TM}\right)^{\wedge 2}\right)
                \arrow[d, "\operatorname{Thom}"]
                \arrow[r, "\tau \wedge Id"]
                &
                H_{p+1}\left(\Sigma^\infty \frac{\cL (M \times M)}{M \times M} \wedge \left(\cL M^{-TM}\right)^{\wedge 2}\right)
                \arrow[d, "\operatorname{Thom}"]
                \\
                H_{p+1+n}\left(\Sigma^\infty \cL M_+^{\wedge 2}\wedge \Sigma^\infty S^1\right)
                \arrow[r, "\operatorname{Swap}"]
                &
                \tilde H_{p+1}\left(\Sigma^\infty S^1 \wedge \cL M_+^{\wedge 2}\right)
                \arrow[r, "\tau \wedge Id"]
                &
                \tilde H_{p+1}\left(\frac{\cL (M \times M)}{M \times M} \wedge \cL M_+^{\wedge 2}\right)
                \\
                H_{p+n}\left(\cL M \times \cL M\right) 
                \arrow[u, "{\cdot \times [0,1]}"]
                \arrow[rr, "h_*\tau \times \cdot"]
                \arrow[ur, "{[0,1] \times \cdot}"]
                &
                &
                \tilde H_{p+1+n}\left(\frac{\cL (M \times M)}{M \times M} \wedge \cL M_+^{\wedge 2}\right)
                \arrow[u, "="]
            \end{tikzcd}
        \end{equation}
        All of (\ref{eq: jjjj}) commutes except the bottom left triangle, which commutes up to a sign of $(-1)^{p +n}$. \par 
        Then the diagram obtained by concatenating (\ref{eq: iiii}), (\ref{eq: jjjj}), composing with maps $\mu_{{{r}}}^{M \times M}$ and $\mu_{M \times M}^{CS}$ similarly to (\ref{eq: hhhh}) and then composing with the natural collapse map (\ref{eq: coll}), has outer square given by (\ref{eq: kkkk}), so (\ref{eq: kkkk}) commutes up to a factor of $(-1)^p$.
    \end{proof}
    
    \begin{proof}[Proof of Corollary \ref{main cor}]
        Combining Proposition \ref{prop: ati}, Corollary \ref{cor: GH spec cop}, Proposition \ref{prop: hom comp inc} and plugging these into Theorem \ref{main theorem}, we find that for $x \in H_p(\cL N)$:
        \begin{multline}
            (-1)^n \Delta^{GH}\circ f_*(x) - (-1)^n (f\times f)_*\circ \Delta^{GH}(x) 
            \\
            = 
            (-1)^{np}\mu^{CS}(x \times [M], h_*[T_{diag}]) - (-1)^{p} \mu^{CS}(h_*[\overline T_{diag}], [M] \times \cdot)
        \end{multline}
        Multiplying through by $(-1)^n$ then gives the result.
    \end{proof}
    
\appendix 
\section{Conventions for stable homotopy theory}\label{sec: sp}
    We work with spectra throughout this paper. We work with the sign conventions of \cite{Adams:book}, mirrored: for example, we apply $\Sigma$ on the \emph{left} when considering the structure maps of spectrum, whereas \emph{loc. cit.} applies $\cdot \wedge S^1$ on the \emph{right}. In this section, we recap the properties and definitions that we need: all results here are standard, but it will be convenient to have a self-contained treatment of all the sign and order conventions we require.
    \subsection{Spectra} 
        \begin{remark}
            When the spaces in the spectra are not of finite type, the definition given below does not necessarily include all morphisms of spectra considered in \cite{Adams:book}.
            However all morphisms that we need in this paper are of this form, so the definition given below is sufficient for our purposes.
        \end{remark}
       
        \begin{definition}
            A \emph{spectrum} $X$ consists of a sequence of based spaces $\{X_n\}_{n\gg 0}$ for $n$ sufficiently large, along with \emph{structure maps} $\sigma_n^X: \Sigma X_n \to X_{n+1}$. \par 
            A \emph{map of spectra} $f: X \to Y$ consists of based maps $f_n: X_n \to Y_n$ for sufficiently large $n$, compatible with the structure maps.\par
            A \emph{homotopy} between two maps $X \to Y$ consists of homotopies between the corresponding maps $X_n \to Y_n$ for sufficiently large $n$, compatible with the structure maps up to homotopy. \par 
            We consider two spectra or maps of spectra the same if they agree for sufficiently large $n$.\par 
            For $k \in \bZ$, the functor $\Sigma^k$ from spectra to itself sends a spectrum $X=\{X_n, \sigma_n^X\}_{n \gg 0}$ to $\{X_{n+k}, \sigma_{n+k}^X\}_{n \gg 0}$, and acts similarly on maps of spectra.
        \end{definition}
        The homotopy category of spectra is enriched in abelian groups, and as such, given a map of spectra $f: X \to Y$ and $n \in \bZ$, there is a map of spectra $n \cdot f: X \to Y$ well-defined up to homotopy. Similarly on the level of spaces, for $i \geq 1$ and spaces $S$ and $T$, the set of homotopy classes of maps $\Sigma^i S \to T$ is an abelian group.\begin{definition}
            A \emph{suspension spectrum} is one in which all structure maps are homotopy equivalences, and are cofibrations.
        \end{definition}
        \begin{example}
            The \emph{sphere spectrum} $\bS$ has $i^{th}$ space $\Sigma^i S^0 \cong [-1,1]^i/\partial [-1,1]^i$.
        \end{example}
        In this paper, we always work in the homotopy category of spectra. For $n \leq n'$, we sometimes write $\sigma^X_{nn'}$ as shorthand for $\sigma^X_{n'-1} \circ \ldots \circ \Sigma^{n'-n}\sigma^X_{n}: \Sigma^{n'-n}X_n \to X_{n'}$. All spectra that we consider are suspension spectra.\par 
        The advantage of working with suspension spectra is that we have the following lemmas:
        \begin{lemma}\label{lem: spectra hom}
            Let $f, g: X \to Y$ be maps betweem two suspension spectra, and $n\gg0$ large enough that $f_n$ and $g_n$ are defined. Then $f$ and $g$ are homotopic if and only if for sufficiently large $n$, $f_n$ and $g_n$ are homotopic as maps of spaces.
        \end{lemma}
        \begin{lemma}\label{lem: spectra unst}
            Let $X$ and $Y$ be suspension spectra, and $n\gg 0$ large enough that $X_n$ and $Y_n$ are defined. Then for any map $g: X_n \to Y_n$ there is a (unique up to homotopy) map of spectra $f: X \to Y$ whose associated map $f_n: X_n \to Y_n$ is $g$.
        \end{lemma}
        \begin{proof}
            Since all $\sigma^X$ and $\sigma^Y$ are homotopy equivalences, for each $i \geq n$, we choose maps $f_i$ inductively in $i$ such that the following diagram commutes:
            \begin{equation*}
                \begin{tikzcd}
                    \Sigma X_i
                    \arrow[r, "f_i"]
                    \arrow[d, "\sigma^X_i"]
                    &
                    \Sigma Y_i
                    \arrow[d, "\sigma^Y_i"]
                    \\
                    X_{i+1} 
                    \arrow[r, "f_{i+1}"]
                    &
                    Y_{i+1}
                \end{tikzcd}
            \end{equation*}
            These exist by, for example, choosing a deformation retraction $X_{i+1} \to \Sigma X_i$ of $\sigma^X_i$.
        \end{proof}
        
        \begin{definition}
            Let $X$ be a spectrum and $S$ a space. The spectrum $X \wedge S$ has $i^{th}$ space $(X \wedge S)_i := X_i \wedge S$ and structure maps $\sigma^{X \wedge S}_i := \sigma^X_i \wedge Id_S$.
        \end{definition}

    \subsection{Homology}
        \begin{definition}\label{def: hom sp}
            Let $X$ be a suspension spectrum. We define its homology to be 
            \begin{equation}
                H_*(X) := \tilde H_{*+i}(X_i)
            \end{equation}
            for some $i \gg 0$. We identify these groups for different choices of $i$ as follows: for $i \leq i'$, we use the isomorphism 
            \begin{equation}\label{eq: hom sp stab}
                \tilde H_{*+i}(X_i) \xrightarrow{[-1,1]^{i'-i} \times \cdot} 
                \tilde H_{*+i'}(\Sigma^{i'-i}X_i) \xrightarrow{\tilde H_*(\sigma^X)}
                \tilde H_{*+i'}(X_{i'})
            \end{equation}
            These isomorphisms are compatible with each other in the sense that composing (\ref{eq: hom sp stab}) for $i \leq i'$ and $i' \leq i''$ gives (\ref{eq: hom sp stab}) for $i \leq i''$.
        \end{definition}
    \subsection{Thom spectra}\label{sec: Thom sp}
        Let $E \to B$ be a vector bundle of rank $r$. We assume that either $B$ is a finite CW complex or that $E = f^*E'$ where $E' \to B'$ is a vector bundle over a finite CW complex and $f: B \to B'$.\par 
        If $E$ is equipped with a metric, we write $DE$ for its unit disc bundle, $SE$ for its unit sphere bundle and $B^{DE}$ for the Thom space $DE / SE$. This is canonically homeomorphic to the quotient space $E/(E \setminus DE^\circ)$; we use these two models for the Thom space interchangeably.
        \begin{definition}
            The \emph{Thom spectrum} $B^{-E}$ of $-E$ is the suspension spectrum defined as follows. \par 
            Choose an embedding $e: E \hookrightarrow \bR^L$ of vector bundles, for some $L\gg 0$. If $B$ is not finite CW, we assume this embedding is obtained by choosing an embedding $E' \hookrightarrow \bR^L$ and pulling back.\par 
            Let $\nu_e$ be the orthogonal complement of $E$ in $\bR^L$. Then for $i \geq L$ the $i^{th}$ space of $B^{-E}$ is defined to be
            \begin{equation}
                (B^{-E})_i := B^{D(\bR^{i-L} \oplus \nu_e)} = \frac{\operatorname{Tot(D(\bR^{i-L}\oplus \nu_e}) \to B)}{\operatorname{Tot(S(\bR^{i-L}\oplus \nu_e} )\to B)}
            \end{equation}
            The structure maps 
            \begin{equation}\label{eq: hom str maps}
                \Sigma B^{D(\bR^{i-L} \oplus \nu_e)} \to B^{D(\bR^{1+i-L} \oplus \nu_e)}
            \end{equation}
            send the $[-1,1]$-coordinate from $\Sigma$ to the first coordinate in $\bR^{1+i-L}$: more precisely, $(t, (u, v, b))$ is sent to $((t, u), v, b)$, where $t \in [-1,1]$, $b \in B$, $u \in \bR^{i-L}$ and $v \in (D\nu_e)_b$.
        \end{definition}
        This definition depended on a choice of embedding $e$. For different choices of $e$, there is a natural identification between the resulting spectra. 
    \subsection{Thom isomorphism}
        We work in the same setting as Section \ref{sec: Thom sp}. Assume also that $E$ is oriented, with corresponding Thom class $\tau_E \in \tilde H^r(B^E)$. 
        \begin{definition}\label{def: thom iso}
            The \emph{Thom isomorphism} is the isomorphism
            \begin{equation}
                \operatorname{Thom}: H_{*-r}(B^{-E}) \to H_*(B)
            \end{equation}
            given by the composition:
            \begin{equation}
                H_{*-r}(B^{-E}) = \tilde H_{*-r+i}\left(B^{D(\bR^{i-L} \oplus \nu_e)}\right) \xrightarrow{\tau_{\bR^{i-L\oplus \nu_e}} \cap -} H_*(B)
            \end{equation}
            where $\tau_{\bR^{i-L} \oplus \nu_e}$ is a Thom class for the vector bundle $\bR^{i-L} \oplus \nu_e$, which we orient so that the canonical isomorphiam
            \begin{equation}
                \bR^{i-L} \oplus \nu_e \oplus E \cong \bR^{i-L} \oplus \bR^L = \bR^i
            \end{equation}
            is orientation-preserving.
        \end{definition}
        This map is independent of choices, in the sense that it is compatible with the maps (\ref{eq: hom str maps}) for different choices of $i$.

    \subsection{Smash product}\label{sec: smash}
        We recap the construction of the smash product of spectra from \cite[Section III.4]{Adams:book}. 
        \begin{definition}\label{def: smash}
            Let $X$, $Y$ be suspension spectra. Choose sequences of nonnegative integers $\vec u= (u_i)_i$ and $\vec v=(v_i)_i$ (which we only require to be defined for sufficiently large $i\gg 0$) such that
            \begin{itemize}
                \item $\vec u$ and $\vec v$ are both monotonically increasing and unbounded.
                \item $u_i+v_i = i$ for all $i$.
            \end{itemize}
            We define the smash product $X \wedge Y$ as follows. The $i^{th}$ space is
            \begin{equation}
                (X\wedge Y)_i = X_{u_i} \wedge Y_{v_i}
            \end{equation}
            and the structure maps are as follows.\par 
            If $u_{i+1} = u_i + 1$ (so $v_{i+1} = v_i$), $\sigma^{X \wedge Y}_i$ is the composition
            \begin{equation}
                \Sigma(X \wedge Y)_i = \Sigma X_{u_i} \wedge Y_{v_i} \xrightarrow{\sigma^X \wedge Id} X_{u_{i+1}} \wedge Y_{v_{i+1}} = (X \wedge Y)_{i+1}
            \end{equation}
            If $v_{i+1} = v_i+1$ (so $u_{i+1} = u_i$), $\sigma^{X \wedge Y}_i$ is the composition
            \begin{equation}
                \Sigma (X \wedge Y)_i = \Sigma X_{u_i} \wedge Y_{v_i} \xrightarrow{\operatorname{swap}} X_{u_i} \wedge \Sigma Y_{v_i} \xrightarrow{(-1)^{u_i} \cdot Id \wedge \sigma^Y} X_{u_{i+1}} \wedge Y_{v_{i+1}} = (X \wedge Y)_{i+1}
            \end{equation}
        \end{definition}
        
        \begin{remark}
            The definition of smash product above depends on the choice of sequences $\vec u$ and $\vec v$; however the resulting spectra for different choices are canonically identified up to homotopy equivalence, see \cite[Theorem III.4.2]{Adams:book}.
        \end{remark}
        
        \begin{remark}
            Let $X,Y,Z$ be suspension spectra. Let $f: X_i \wedge Y_j \to Z_{i+j}$ be a map of spaces. We may choose sequences $\vec u, \vec v$ as in Definition \ref{def: smash} with $u_{i+j} = i$ and $v_{i+j} = j$ and apply Lemma \ref{lem: spectra unst} to obtain a well-defined map of spectra $X \wedge Y \to Z$.
        \end{remark}
        \begin{lemma}
            Let $X$ be a spectrum. Then there is a homotopy equivalence of spectra 
            \begin{equation}
                f: X \wedge \bS \to X
            \end{equation}
        \end{lemma}
        \begin{proof}
            Let $(u_i)_i, (v_i)_i$ be sequences as in Definition \ref{def: smash}. We define $f$ on $i^{th}$ spaces to be the composition
            \begin{equation}
                (X \wedge \bS)_i = X_{u_i} \wedge \Sigma^{v_i} S^0 \xrightarrow{\operatorname{swap}} \Sigma^{v_i} X_{u_i} \wedge S^0 \cong \Sigma^{v_i} X_{u_i} \xrightarrow{(-1)^{u_i v_i} \cdot \sigma^X} X_i
            \end{equation}
            This is a map of spectra.
        \end{proof}
    \subsection{Pontrjagin-Thom theory}\label{sec: PT}
        In this section, we record a concrete model for the Pontrjagin-Thom construction.
        \begin{definition}
            A \emph{stable framing} on a manifold $X$ consists of an equivalence class of isomorphisms of vector bundles over $X$ $\psi: \bR^{i-k} \oplus TX \to \bR^i$. The equivalence relation is generated by the following relations:
            \begin{itemize}
                \item $\psi, \psi': \bR^{i-k} \oplus TX \to \bR^i$ are equivalent if they are homotopic (through isomorphisms of vector bundles).
                \item $\psi$ is equivalent to $Id_\bR \oplus \psi: \bR^{1+i-k} \oplus TX \to \bR^{1+i}$.
            \end{itemize}
        \end{definition}
        Let $A\subseteq B$ be a CW subcomplex of a CW complex, and $X^k$ a compact manifold, possibly with boundary, equipped with a stable framing. Let $f: X \to B$ be a map sending $\partial X$ to $A$.
        \begin{definition}
            \emph{Pontrjagin-Thom data} of \emph{rank $L$} for the data above consists of a tuple $(i, \sigma, \psi)$:
            \begin{enumerate}
                \item $i: X \hookrightarrow (-1, 1)^L$ is an embedding. Write $\nu_i$ for the normal bundle of this embedding.
                \item $\sigma: D\nu_i \hookrightarrow [-1, 1]^L$ is a tubular neighbourhood of the embedding $i$.
                \item $\psi: \nu_i \to \bR^{L-k}$ is an isomorphism of vector bundles such that the following composition is a representative for the stable framing on $X$:
                \begin{equation}
                    \bR^{L-k} \oplus TX \xrightarrow{\psi^{-1} \oplus Id} 
                    \nu_i \oplus TX \xrightarrow{=} \bR^L
                \end{equation}
                and such that 
                \begin{equation}\label{eq: PT fram cond}
                    |\psi(v)| \geq |v|
                \end{equation}
                for all $v \in \nu_i$.
            \end{enumerate}
        \end{definition}
        Given Pontrjagin-Thom data as above, we construct a map of spectra $\Sigma^k \bS \to \Sigma^\infty \frac BA$ as follows.\par 
        This map is defined on $(L-k)^{th}$ spaces to be the composition, which we call $[X]_{unst}$:
        \begin{equation}
            \Sigma^L S^0 \xrightarrow{\operatorname{Collapse}}
            \frac{X^{D\nu_i}}{\partial X^{D\nu_i}}
            \xrightarrow{\psi} \Sigma^{L-k} \frac{X}{\partial X} \xrightarrow{ \Sigma^{L-k}f} 
            \Sigma^{L-k} \frac BA
        \end{equation}
        Here the first map $\operatorname{Collapse}$ sends $p \in [-1, 1]^L$ to $\sigma^{-1}(p)$ if $p \in \operatorname{Im}(p)$ and to the basepoint otherwise, and the second map $\psi$ sends $(v, x)$ (where $x \in X$ and $v \in (D\nu_i)_x)$ to $(\psi(v), x)$.\par 
        Standard arguments (e.g. \cite[Section IV]{Rudyak}) show that Pontrjagin-Thom data always exists, and that the induced map of spectra is independent of the choice of Pontrjagin-Thom data up to homotopy.\par 
        \begin{definition}\label{def: fund class}
            Let $M$ be a compact manifold, possibly with boundary of corners. Its \emph{stable homotopy fundamental class} is the map $[M]: \bS \to \frac{M^{-TM}}{\partial M^{-TM}}$ constructed as follows.\par 
            Let $i: M \hookrightarrow (-1, 1)^L$ be an embedding, and $\sigma$
            A map of spaces $[M]_{unst}$ is defined to be the map $\Sigma^L S^0 \to \frac{M^{D\nu_i}}{\partial M^{D\nu_i}}$ sending $x \in [-1,1]^L$ to $\sigma^{-1}(x)$ if $x \in \operatorname{Im}(\sigma)$, and $*$ otherwise. The map of spectra $[M]$ is then induced by Lemma \ref{lem: spectra unst}. 
        \end{definition}
        This map of spectra is independent of choices up to homotopy.

\bibliography{Refs.bib}{} \bibliographystyle{abbrv}

@misc{chas-sullivan,
  doi = {10.48550/ARXIV.MATH/9911159},
  
  url = {https://arxiv.org/abs/math/9911159},
  
  author = {Chas, Moira and Sullivan, Dennis},
  
  title = {String Topology},
  
  publisher = {arXiv},
  
  year = {1999},
  
  copyright = {Assumed arXiv.org perpetual, non-exclusive license to distribute this article for submissions made before January 2004}
}

@article{Goresky_2009,
	doi = {10.1215/00127094-2009-049},
  
	url = {https://doi.org/10.1215%2F00127094-2009-049},
  
	year = 2009,
	month = {oct},
  
	publisher = {Duke University Press},
  
	volume = {150},
  
	number = {1},
  
	author = {Mark Goresky and Nancy Hingston},
  
	title = {Loop products and closed geodesics},
  
	journal = {Duke Mathematical Journal}
}

@misc{Naef,
  doi = {10.48550/ARXIV.2106.11307},
  
  url = {https://arxiv.org/abs/2106.11307},
  
  author = {Naef, Florian},
  
  title = {The string coproduct ``knows'' {R}eidemeister/{W}hitehead torsion},
  
  publisher = {arXiv},
  
  year = {2021},
  
  copyright = {arXiv.org perpetual, non-exclusive license}
}

@misc{BokstedtTHH,
  
  author = {B\"okstedt, Marcel},
  
  title = {Topological {H}ochschild homology},
  
  publisher = {Bielefeld},
  
  year = {1985}
}

@article{GN,
 ISSN = {00029327, 10806377},
 URL = {http://www.jstor.org/stable/2374935},
 author = {Ross Geoghegan and Andrew Nicas},
 journal = {American Journal of Mathematics},
 number = {2},
 pages = {397--446},
 publisher = {Johns Hopkins University Press},
 title = {Parameterized Lefschetz-Nielsen Fixed Point Theory and Hochschild Homology Traces},
 urldate = {2022-09-06},
 volume = {116},
 year = {1994}
}

@misc{cohen2002homotopy,
      title={A homotopy theoretic realization of string topology}, 
      author={Ralph L. Cohen and John D. S. Jones},
      year={2002},
      eprint={math/0107187},
      archivePrefix={arXiv},
      primaryClass={math.GT}
}

@article{hingston2017product,
  title={Product and coproduct in string topology},
  author={Hingston, Nancy and Wahl, Nathalie},
  journal={arXiv preprint arXiv:1709.06839},
  year={2017}
}

@book {Adams:book,
    AUTHOR = {Adams, J. F.},
     TITLE = {Stable homotopy and generalised homology},
    SERIES = {Chicago Lectures in Mathematics},
      NOTE = {Reprint of the 1974 original},
 PUBLISHER = {University of Chicago Press, Chicago, IL},
      YEAR = {1995},
     PAGES = {x+373},
      ISBN = {0-226-00524-0},
   MRCLASS = {55-02 (57-02)},
  MRNUMBER = {1324104},
}

@misc{malkiewichparametrized,
      title={Parametrized spectra, a user guide}, 
      author={Cary Malkiewich},
      
      url={https://people.math.binghamton.edu/malkiewich/users_guide_parametrized.pdf},
      
}

@article {Jakob,
    AUTHOR = {Jakob, Martin},
     TITLE = {A bordism-type description of homology},
   JOURNAL = {Manuscripta Math.},
  FJOURNAL = {Manuscripta Mathematica},
    VOLUME = {96},
      YEAR = {1998},
    NUMBER = {1},
     PAGES = {67--80},
      ISSN = {0025-2611,1432-1785},
   MRCLASS = {55N20 (55N15 55N22 55P42)},
  MRNUMBER = {1624352},
MRREVIEWER = {Masayoshi\ Kamata},
       DOI = {10.1007/s002290050054},
       URL = {https://doi.org/10.1007/s002290050054},
}

@article {Chataur,
    AUTHOR = {Chataur, David},
     TITLE = {A bordism approach to string topology},
   JOURNAL = {Int. Math. Res. Not.},
  FJOURNAL = {International Mathematics Research Notices},
      YEAR = {2005},
    NUMBER = {46},
     PAGES = {2829--2875},
      ISSN = {1073-7928,1687-0247},
   MRCLASS = {55P48 (55N20 55N22 55P35 55P43)},
  MRNUMBER = {2180465},
MRREVIEWER = {Hossein\ Abbaspour},
       DOI = {10.1155/IMRN.2005.2829},
       URL = {https://doi.org/10.1155/IMRN.2005.2829},
}

@article {Naef-Rivera-Wahl,
    AUTHOR = {Naef, Florian and Rivera, Manuel and Wahl, Nathalie},
     TITLE = {String topology in three flavors},
   JOURNAL = {EMS Surv. Math. Sci.},
  FJOURNAL = {EMS Surveys in Mathematical Sciences},
    VOLUME = {10},
      YEAR = {2023},
    NUMBER = {2},
     PAGES = {243--305},
      ISSN = {2308-2151,2308-216X},
   MRCLASS = {55P50},
  MRNUMBER = {4667421},
MRREVIEWER = {Saeid\ Jafari},
       DOI = {10.4171/emss/72},
       URL = {https://doi.org/10.4171/emss/72},
}

@article {Atiyah,
    AUTHOR = {Atiyah, M. F.},
     TITLE = {Thom complexes},
   JOURNAL = {Proc. London Math. Soc. (3)},
  FJOURNAL = {Proceedings of the London Mathematical Society. Third Series},
    VOLUME = {11},
      YEAR = {1961},
     PAGES = {291--310},
      ISSN = {0024-6115,1460-244X},
   MRCLASS = {57.30},
  MRNUMBER = {131880},
MRREVIEWER = {R.\ Bott},
       DOI = {10.1112/plms/s3-11.1.291},
       URL = {https://doi.org/10.1112/plms/s3-11.1.291},
}

@article {Mazur,
    AUTHOR = {Mazur, Barry},
     TITLE = {Differential topology from the point of view of simple
              homotopy theory},
   JOURNAL = {Inst. Hautes \'Etudes Sci. Publ. Math.},
  FJOURNAL = {Institut des Hautes \'Etudes Scientifiques. Publications
              Math\'ematiques},
    NUMBER = {15},
      YEAR = {1963},
     PAGES = {93},
      ISSN = {0073-8301,1618-1913},
   MRCLASS = {57.10 (55.40)},
  MRNUMBER = {161342},
MRREVIEWER = {T.\ Stewart},
       URL = {http://www.numdam.org/item?id=PMIHES_1963__15__93_0},
}

@book {Bredon,
    AUTHOR = {Bredon, Glen E.},
     TITLE = {Topology and geometry},
    SERIES = {Graduate Texts in Mathematics},
    VOLUME = {139},
      NOTE = {Corrected third printing of the 1993 original},
 PUBLISHER = {Springer-Verlag, New York},
      YEAR = {1997},
     PAGES = {xiv+557},
      ISBN = {0-387-97926-3},
   MRCLASS = {55-01 (54-01 57-01)},
  MRNUMBER = {1700700},
}

@book {Rudyak,
    AUTHOR = {Rudyak, Yuli B.},
     TITLE = {On {T}hom spectra, orientability, and cobordism},
    SERIES = {Springer Monographs in Mathematics},
      NOTE = {With a foreword by Haynes Miller},
 PUBLISHER = {Springer-Verlag, Berlin},
      YEAR = {1998},
     PAGES = {xii+587},
      ISBN = {3-540-62043-5},
   MRCLASS = {55-02 (55N22 55P42 57-02)},
  MRNUMBER = {1627486},
MRREVIEWER = {Donald\ M.\ Davis},
}

@article{naef2024simple,
  title={Simple homotopy invariance of the loop coproduct},
  author={Naef, Florian and Safronov, Pavel},
  journal={arXiv preprint arXiv:2406.19326},
  year={2024}
}

@article{hingston2019invariance,
  title={On the invariance of the string topology coproduct},
  author={Wahl, Nathalie},
  journal={arXiv preprint arXiv:1908.03857},
  year={2024}
}

@article {Gruher-Salvatore,
    AUTHOR = {Gruher, Kate and Salvatore, Paolo},
     TITLE = {Generalized string topology operations},
   JOURNAL = {Proc. Lond. Math. Soc. (3)},
  FJOURNAL = {Proceedings of the London Mathematical Society. Third Series},
    VOLUME = {96},
      YEAR = {2008},
    NUMBER = {1},
     PAGES = {78--106},
      ISSN = {0024-6115,1460-244X},
   MRCLASS = {55P43 (55R10 55R35 55R70 55S12 57R19)},
  MRNUMBER = {2392316},
MRREVIEWER = {Steven\ R.\ Costenoble},
       DOI = {10.1112/plms/pdm030},
       URL = {https://doi.org/10.1112/plms/pdm030},
}

@article {Crabb,
    AUTHOR = {Crabb, M. C.},
     TITLE = {Loop homology as fibrewise homology},
   JOURNAL = {Proc. Edinb. Math. Soc. (2)},
  FJOURNAL = {Proceedings of the Edinburgh Mathematical Society. Series II},
    VOLUME = {51},
      YEAR = {2008},
    NUMBER = {1},
     PAGES = {27--44},
      ISSN = {0013-0915,1464-3839},
   MRCLASS = {55R70 (55P43)},
  MRNUMBER = {2391632},
MRREVIEWER = {Donald\ W.\ Kahn},
       DOI = {10.1017/S0013091505001483},
       URL = {https://doi.org/10.1017/S0013091505001483},
}

@article {Cohen-Klein-Sullivan,
    AUTHOR = {Cohen, Ralph L. and Klein, John R. and Sullivan, Dennis},
     TITLE = {The homotopy invariance of the string topology loop product
              and string bracket},
   JOURNAL = {J. Topol.},
  FJOURNAL = {Journal of Topology},
    VOLUME = {1},
      YEAR = {2008},
    NUMBER = {2},
     PAGES = {391--408},
      ISSN = {1753-8416,1753-8424},
   MRCLASS = {55N45 (55P35 55R80)},
  MRNUMBER = {2399136},
MRREVIEWER = {Sadok\ Kallel},
       DOI = {10.1112/jtopol/jtn001},
       URL = {https://doi.org/10.1112/jtopol/jtn001},
}

@article {Rivera-Wang,
    AUTHOR = {Rivera, Manuel and Wang, Zhengfang},
     TITLE = {Invariance of the {G}oresky-{H}ingston algebra on reduced
              {H}ochschild homology},
   JOURNAL = {Proc. Lond. Math. Soc. (3)},
  FJOURNAL = {Proceedings of the London Mathematical Society. Third Series},
    VOLUME = {125},
      YEAR = {2022},
    NUMBER = {2},
     PAGES = {219--257},
      ISSN = {0024-6115,1460-244X},
   MRCLASS = {16E40 (16S38 18G10 55P50)},
  MRNUMBER = {4466234},
MRREVIEWER = {Xiaojun\ Chen},
}

@book {WJR,
    AUTHOR = {Waldhausen, Friedhelm and Jahren, Bj\o rn and Rognes, John},
     TITLE = {Spaces of {PL} manifolds and categories of simple maps},
    SERIES = {Annals of Mathematics Studies},
    VOLUME = {186},
 PUBLISHER = {Princeton University Press, Princeton, NJ},
      YEAR = {2013},
     PAGES = {vi+184},
      ISBN = {978-0-691-15776-4},
   MRCLASS = {57-02 (19D10 57Q10 57Q20 57Q25 57Q60)},
  MRNUMBER = {3202834},
MRREVIEWER = {Ryo\ Ohashi},
       DOI = {10.1515/9781400846528},
       URL = {https://doi.org/10.1515/9781400846528},
}
\Addresses
    
\end{document}